\documentclass[ejs]{imsart} 
\arxiv{2209.04757} 

\RequirePackage{amsthm,amsmath,amsfonts,amssymb}
\RequirePackage[authoryear]{natbib}
\RequirePackage[colorlinks,citecolor=blue,urlcolor=blue]{hyperref}
\RequirePackage{graphicx}
\usepackage{mathtools}
\usepackage{dsfont} 
\usepackage{booktabs}
\usepackage{subcaption} 
\usepackage{enumerate} 
\usepackage{float} 

\startlocaldefs

\numberwithin{equation}{section}
\theoremstyle{plain}
\newtheorem{theorem}{Theorem}[section]
\newtheorem{proposition}[theorem]{Proposition}
\newtheorem{lemma}[theorem]{Lemma}
\newtheorem{corollary}[theorem]{Corollary}

\theoremstyle{definition}
\newtheorem{definition}{Definition}[section]
\newtheorem{remark}[definition]{Remark}

\newcommand{\N}{\mathbb{N}}
\newcommand{\R}{\mathbb{R}}
\newcommand{\PP}{\mathsf{P}} 
\newcommand{\QQ}{\mathsf{Q}} 
\newcommand{\EE}{\mathsf{E}} 
\newcommand{\Bias}{\mathsf{Bias}} 
\newcommand{\Var}{\mathsf{Var}} 
\newcommand{\Cov}{\mathsf{Cov}} 
\newcommand{\bb}[1]{\boldsymbol{#1}} 
\newcommand{\mat}[1]{\mathbf{#1}} 
\newcommand{\OO}{\mathcal{O}}
\newcommand{\oo}{\mathrm{o}}
\newcommand{\rd}{\mathrm{d}}
\newcommand{\ind}{\mathds{1}}
\newcommand{\e}{\varepsilon}
\DeclareMathOperator*{\argmin}{\mathrm{argmin}}
\DeclareMathOperator*{\argmax}{\mathrm{argmax}}
\DeclareMathOperator{\diag}{\mathrm{diag}}

\endlocaldefs

\allowdisplaybreaks

\begin{document}

\begin{frontmatter}

\title{\texorpdfstring{Normal approximations for the multivariate inverse Gaussian distribution and asymmetric kernel smoothing on $d$-dimensional half-spaces}{Normal approximations for the multivariate inverse Gaussian distribution and asymmetric kernel smoothing on d-dimensional half-spaces}}
\runtitle{Asymmetric kernel smoothing on d-dimensional half-spaces}

\begin{aug}
\author[A]{\fnms{L\'eo}~\snm{R. Belzile}\ead[label=e1]{leo.belzile@hec.ca}\orcid{0000-0002-9135-014X}},
\author[B]{\fnms{Alain}~\snm{Desgagn\'e}\ead[label=e2]{desgagne.alain@uqam.ca}\orcid{0000-0001-9047-8566}}
\author[C]{\fnms{Christian}~\snm{Genest}\ead[label=e3]{christian.genest@mcgill.ca}\orcid{0000-0002-1764-0202}}
\and
\author[D]{\fnms{Fr\'ed\'eric}~\snm{Ouimet}\ead[label=e4]{frederic.ouimet2@uqtr.ca}\orcid{0000-0001-7933-5265}}

\address[A]{D\'epartement de sciences de la d\'ecision, HEC Montr\'eal\printead[presep={\\}]{e1}}
\address[B]{D\'epartement de math\'ematiques, Universit\'e du Qu\'ebec \`a Montr\'eal\printead[presep={\\}]{e2}}
\address[C]{Department of Mathematics and Statistics, McGill University\printead[presep={\\}]{e3}}
\address[D]{D\'epartement de math\'ematiques et d'informatique, Universit\'e du Qu\'ebec \`a Trois-Rivi\`eres\printead[presep={\\}]{e4}}

\runauthor{Belzile et al.}
\end{aug}

\begin{abstract}
This paper introduces a novel density estimator supported on $d$-dimensional half-spaces. It stands out as the first asymmetric kernel density estimator for half-spaces in the literature. Using the multivariate inverse Gaussian (MIG) density from \cite{MR2019130} as the kernel and incorporating locally adaptive parameters, the estimator achieves desirable boundary properties. To analyze its mean integrated squared error (MISE) and asymptotic normality, a local limit theorem and probability metric bounds are established between the MIG and the corresponding multivariate Gaussian distribution with the same mean vector and covariance matrix, which may also be of independent interest. Additionally, a new algorithm for generating MIG random vectors is developed, proving to be faster and more accurate than Minami's algorithm based on a Brownian first-hitting location representation. This algorithm is then used to discuss and compare optimal MISE and likelihood cross-validation bandwidths for the estimator in a simulation study under various target distributions. As an application, the MIG asymmetric kernel is used to smooth the posterior distribution of a generalized Pareto model fitted to large electromagnetic storms.
\end{abstract}


\bigskip

\end{frontmatter}

\newpage

\tableofcontents 

\newpage

\section{Introduction}\label{sec:intro}

\subsection{Overview of the literature}

The inverse Gaussian (IG) distribution is a versatile stochastic model often used to represent data with positive support and skewness. It was first introduced by~\cite{MR10927} as a limiting distribution in the context of sequential probability ratio tests. Since then, it has found a large number of applications across diverse fields, including actuarial science (claim cost analysis), demography (population growth, fertility), engineering (electrical networks, material fatigue, hydrology), medicine (length of hospital stay, pharmacokinetics, tracer dilution curves), and environmental science (chemical movements in soil, species abundance), among others; see~\cite{MR1622488} for a review.

Several multivariate extensions of the IG distribution have been proposed in the literature, each designed to address specific modeling challenges or theoretical considerations. An early example is the bivariate IG distribution introduced by~\cite{MR613202}. In their approach, the joint density is modeled as the product of univariate IG marginal densities (in variables $x$ and $y$), multiplied by a factor of the form $1 + \rho \Psi(x, y)$, where the correlation coefficient $\rho$ is selected as a function of other parameters to ensure that the joint density integrates to $1$. This construction was suggested by \citet[p.~292]{Parzen:1960}. However, the range of admissible values for $\rho$ depends on the marginal parameters, which limits the flexibility of this model.

Alternatively, \cite{MR707928} defined reproductive exponential families and proposed bivariate IG models through a conditional approach which is convenient for analysis of variance, but their distributions do not have IG margins.

Bivariate IG distributions have also been derived from stochastic processes. For instance, in~\cite{MR813460}, a bivariate IG distribution arises naturally from the joint law of the first-hitting times of a two-dimensional driftless Brownian motion with correlated components, generalizing the well-known fact that the first-hitting time of a standard Brownian motion follows an IG distribution. An earlier proposal had been made by \cite{MR242241}, who investigated a specific continuous-time process for which the joint law of two increments is bivariate~IG. \cite{MR836585} corrected the expression of the bivariate density originally derived by Wasan, extended the construction, and established the basic properties of the resulting model. For a review of bivariate IG distributions, refer to Section~11.21 of~\cite{MR2840643}.

In arbitrary dimension $d \geq 2$, constructing flexible models with (possibly generalized) IG margins is challenging. \cite{MR1183075} addressed this issue by leveraging random additive effects and Poisson mixtures to generate the dependence structure. \cite{MR2979762} mention the possibility of constructing such models by combining IG margins with a copula \citep{Genest/Neslehova:2012a}, but their own extension, which is closed under marginalization, is based instead on a stochastic representation of the IG distribution using a transform of a $t_2$-skewed normal distribution. Unfortunately, both approaches lack detailed information about moments, cumulants, and other fundamental properties such as infinite divisibility, limiting their practical and theoretical utility.

A more advantageous form of multivariate inverse Gaussian (MIG) distribution was proposed by \cite{MR2019130}. The latter is defined on $d$-dimensional half-spaces and is based on a reduced inverse relationship with multivariate Gaussian distributions; see Definition~\ref{def:MIG} below for its density. It allows for easy derivations of its mean vector, covariance matrix, and higher-order cumulants. It is also reproductive, infinitely divisible, closed under linear transformations, and occurs as a limiting case of the multivariate Lagrange distribution \citep{MR2363283}. Furthermore, given its representation as the location of a $d$-variate Brownian motion with correlated components upon hitting a specified hyperplane for the first time, it is straightforward to simulate.

In practice, Minami's MIG distribution has also been the most widely used version of the multivariate inverse Gaussian distribution. For example, \cite{doi:10.1088/0256-307X/24/6/017,doi:10.1142/S0218127409024736} calculate the multiscale entropy of $1/f$ noise under the MIG distribution in the context of physical and physiological signals. \cite{Masuadi_2013_PhD} introduces MIG frailty models with competing risks for their tractability, robustness, and flexible correlation structure. \cite{doi:10.1049/joe.2019.0142} use MIG textures to model spatially correlated background interference and improve target detection when the exact target location is not fully known. \cite{doi:10.1109/LWC.2021.3110892} model propagation delays in molecular communication systems using the MIG distribution and estimate channel parameters with maximum likelihood. As for \cite{doi:10.1109/JSYST.2023.3256061}, they focus on clock synchronization and range estimation in networks using MIG-distributed propagation delays.

All this makes Minami's MIG distribution a valuable tool from a theoretical standpoint and for a wide range of statistical applications. Therefore, from hereon, the MIG distribution will specifically refer to Minami's version of the multivariate inverse Gaussian distribution. It is the version studied in this paper.

\subsection{Density estimation and boundary bias on half-spaces}\label{sec:density.estimation.half.spaces}

While the MIG distribution can be used to model data on half-spaces, it may not always offer sufficient flexibility to capture the underlying distribution that generated the observations. Rather than imposing a global parametric framework, a nonparametric kernel density estimation approach can be adopted instead.

Despite half-spaces being ubiquitous in machine learning research due to concepts such as support vector machine classifiers \citep[e.g.,][]{blanco2022mathematical,liang2022uncertainty}, perceptrons \citep[e.g.,][]{ghazi2021robust,lin2024geometrical}, half-space depth of data \citep[e.g.,][]{kotik2017weighted,pokorny2024another}, and optimization \citep[e.g.,][]{frei2021agnostic}, there appears to be no paper that addresses kernel density estimation on $d$-dimensional half-spaces ($d \geq 2$) and the corresponding \emph{boundary bias problem} that arises when applying standard kernel methods \citep[Section~4.2]{MR1319818} on such domains. The absence of established methods for this issue provided the main motivation for the study reported here.

The boundary bias problem refers to the systematic underestimation of density values near the boundary of the support, caused by the spill-over effect of fixed, symmetric kernels that fail to adapt their shape near boundaries. When the estimation point lies close to the boundary, a substantial portion of the kernel's mass extends outside the support (into regions where the true density is zero), reducing the effective weight assigned to points within the support and leading to biased estimates. For example, the commonly used Gaussian kernel \cite[e.g.,][p.~15]{MR3822372} spreads its mass around the estimation point without accounting for the support's constraints. This mismatch causes a portion of the kernel's contribution to vanish outside the valid region, and the closer the estimation point is to the boundary, the larger the proportion of the kernel's mass that lies outside the support, further amplifying the bias.

Several specialized techniques have been developed to mitigate this issue in other settings, such as $[0,1]$, $[0,\infty)$, products of these intervals, and the $d$-dimensional simplex. For instance, the reflection method \citep[e.g.,][]{Schuster1985,ClineHart1991} mirrors data across the boundary, ensuring that kernels allocate their mass within the support. Boundary kernels \citep[e.g.,][]{GasserMuller1979,GasserMullerMammitzsch1985,Muller1991,Jones1993,ZhangKarunamuni1998,ZhangKarunamuni2000} adjust their shape near boundaries to eliminate spill-over and conform to domain constraints. The transformation approach \citep[e.g.,][]{MarronRuppert1994,RupperCline1994} maps the data to an unbounded space, applies a standard kernel smoothing technique, and then transforms the estimate back to the original domain. Asymmetric kernels \citep[e.g.,][]{AitchisonLauder1985,Chen1999,MR2568128,MR2756441,MR4319409} offer another direct solution by adapting their shape through the parameters of the underlying distribution family for each estimation point. Furthermore, the support's constraint is intrinsically taken into account because the asymmetric kernel's support matches that of the target density.

The primary method put forward in this paper for estimating densities on half-spaces involves using the MIG distribution of \citet{MR2019130} as an asymmetric kernel. It serves as the central foundation of the paper, with the other main developments building upon or complementing this approach, as explained in the next section.

\subsection{Contributions and motivations}

The proposed MIG kernel density estimator, introduced formally in Section~\ref{sec:MIG.kernel.density estimator}, has asymptotically negligible bias near the boundary. It stands out as the first boundary-adapted kernel designed specifically for $d$-dimensional ($d\geq 2$) half-spaces in the literature. Section~\ref{sec:asymptotics} explores its asymptotic properties in detail, including its mean integrated squared error (MISE) and asymptotic normality. Section~\ref{sec:simulation.study} compares its performance to other boundary-adapted competitors of our own making, highlighting its advantages in certain scenarios. An application of the estimator to the problem of smoothing the posterior distribution of a generalized Pareto model fitted to large electromagnetic storms, presented in Section~\ref{sec:data.application}, demonstrates its practical relevance. The introduction of a boundary-adapted density estimation method for half-spaces, together with its theoretical and practical investigation, constitutes the central and first major contribution of this paper.

To derive the asymptotic expression for the pointwise variance of the MIG kernel density estimator in Section~\ref{sec:asymptotics}, which is a fundamental step in obtaining the asymptotic MISE, it was necessary to develop a local approximation of the MIG density with respect to the multivariate normal density sharing the same mean vector and covariance matrix. Hence, Section~\ref{sec:normal.approximations} develops this approximation, where probability metric bounds are also provided as a complement. These normal approximations facilitate the derivation of the MIG kernel density estimator's asymptotic properties, but are likely of independent interest for other large-sample settings that involve the MIG distribution. For completeness, one-dimensional versions of the local limit theorem and probability metric bounds are also established under a slightly different (but commonly used) parametrization, as they seem to be unavailable in the literature. These normal approximations form the second major contribution of this paper.

When including the MIG density among the target densities in the simulation study in Section~\ref{sec:simulation.study}, the need arose to generate MIG random samples. Minami's original proposition for generating such samples \citep[Theorem~2]{MR2019130} relies on the aforementioned Brownian first-hitting location representation. Because the Brownian paths must be discretized in practice, Minami's algorithm is approximate. It becomes computationally very expensive as well, especially in higher dimensions. Section~\ref{sec:rng.MIG} introduces a new algorithm that surpasses Minami's approach in both precision and speed. The method is exact and avoids the inefficiencies of iterative simulations by leveraging the stochastic representation of MIG random vectors as a transformation of a location and scale mixture of Gaussian random vectors, with inverse Gaussian driving noise. The same stochastic representation also enables the provision of an estimator of the cumulative distribution function (cdf) via importance sampling in two dimensions. By providing faster and more accurate sampling, the algorithm reduces computational overhead, facilitates the simulation study, and supports efficient Monte Carlo estimation of optimal bandwidth matrices for the MIG kernel density estimator. This aspect is crucial for comparing optimal MISE and likelihood cross-validation bandwidths in Section~\ref{sec:simulation.study}. This improved MIG simulation algorithm represents the third major contribution of this paper.

\subsection{Outline}

The rest of this paper is organized as follows. Section~\ref{subsec:notation} introduces the notation used throughout, as well as the MIG distribution of~\cite{MR2019130}, along with some of its basic properties, including estimators for its parameters and a useful stochastic representation.

Section~\ref{sec:normal.approximations} is devoted to local limit theorems and probability metric bounds, stated first for the univariate case in Section~\ref{sec:univariate.IG} and for the general multivariate case in Section~\ref{sec:multivariate.IG}. The proofs are deferred to Appendix~\ref{app:proofs.normal.approximations}.

Section~\ref{sec:MIG.kernel.density estimator} introduces the MIG kernel density estimator with positive definite bandwidth matrix $\mat{H}$. Its asymptotic properties are stated in Section~\ref{sec:asymptotics} and proved in Appendix~\ref{app:proofs.application.asymptotics}. Some technical lemmas used in the proofs are relegated to Appendix~\ref{app:technical.lemmas}. A simulation study is conducted in Section~\ref{sec:simulation.study}, and a real-data application is investigated in Section~\ref{sec:data.application}. Information regarding the \textsf{R} code that generated the figures, the simulation study results and the
real-data application is provided in Appendix~\ref{app:R.code}.

Section~\ref{sec:rng.MIG} presents a new algorithm for generating MIG random vectors that is both faster and more accurate than the original method due to \citet{MR2019130}. In Appendix~\ref{app:cdf.evaluation}, the stochastic representation is leveraged to provide a Monte Carlo estimator of the cdf of MIG random vectors.

Section~\ref{sec:conclusion} concludes with a discussion of some of the limitations of the MIG kernel density estimator. Future directions for research are also mentioned. For ease of reference, a list of the abbreviations used herein appears in Appendix~\ref{app:abbreviations}.

\subsection{Notation and basic properties} \label{subsec:notation}

Throughout the paper, let $d \in \N$ denote the dimension of the ambient space. For any normal vector $\bb{\beta} \in \R^d$, define the associated (open) half-space by
\[
\mathcal{H}(\bb{\beta}) = \{\bb{x} \in \R^d : \bb{\beta}^{\top} \bb{x} \in (0, \infty)\}.
\]
The symbols $\bb{0}_d$ and $\bb{1}_d$ denote the $d$-vectors of $0$s and $1$s, respectively, and $\mat{I}_d$ denotes the identity matrix of order $d$. The space of $d \times d$ real positive definite matrices is denoted by $\mathcal{S}_{++}^d$. The notations $\|\cdot\|_2$, $|\cdot|$, and $\mathrm{tr}(\cdot)$ refer to the spectral norm, determinant, and trace, respectively. For a Boolean condition $A$, the indicator function $\ind_A$ returns $1$ if $A$ is true and zero otherwise.

For the purpose of comparing asymptotic behaviors, writing $u = \OO(v)$ means that $\limsup |u/v| \leq C < \infty$ as $\lambda/\mu \to \infty$, $\mu/\omega \to \infty$, $n \to \infty$, or $|\mat{H}|_2 \to 0$, depending on the context. The positive constant $C\in (0,\infty)$ may be a function of the target density $f$ or the dimension $d$, but of no other variable unless explicitly written as a subscript. If $u = \OO(v)$ and $v = \OO(u)$ both hold, one writes $u\asymp v$. Similarly, the notation $u = \oo(v)$ means that $\lim |u/v| = 0$ as $\lambda/\mu\to \infty$ or $\mu/\omega\to \infty$ or $n\to \infty$ or $\|\mat{H}\|_2\to 0$, depending on the context. Subscripts indicate the parameters upon which the convergence rate is contingent.

\begin{definition}[MIG distribution] \label{def:MIG}
Let $\bb{\beta},\bb{\xi}\in \R^d$ be such that $\bb{\beta}^{\top} \bb{\xi}\in (0,\infty)$, and let $\mat{\Omega}\in \mathcal{S}_{++}^d$. A random $d$-vector $\bb{X}$ is said to be $\mathrm{MIG}(\bb{\beta}, \bb{\xi}, \mat{\Omega})$-distributed, and one writes $\bb{X}\sim \mathrm{MIG}(\bb{\beta},\bb{\xi},\mat{\Omega})$, if its density function is given, for every $\bb{x}\in \mathcal{H}(\bb{\beta})$, by
\begin{equation}\label{eq:general.multivariate.density}
k_{\bb{\beta},\bb{\xi},\mat{\Omega}}(\bb{x})
= \frac{\bb{\beta}^{\top} \bb{\xi}\,\,|\mat{\Omega}|^{-1/2}}{(2\pi)^{d/2}(\bb{\beta}^{\top} \bb{x})^{d/2 + 1}}\exp\left\{- \frac{1}{2 \bb{\beta}^{\top} \bb{x}} (\bb{x} - \bb{\xi})^{\top}\mat{\Omega}^{-1} (\bb{x} - \bb{\xi})\right\}.
\end{equation}
\end{definition}

The expectation and covariance matrix of the MIG distribution are
\begin{equation}\label{eq:esp.var.multivariate}
\EE(\bb{X}) = \bb{\xi}, \qquad \Var(\bb{X}) = \bb{\beta}^{\top} \bb{\xi} \, \mat{\Omega},
\end{equation}
by Property~1 of~\cite{MR2019130}. Intuitively, the parameter $\bb{\xi}$ determines the center or location of the distribution. As for $\mat{\Omega}$, it plays the role of a scale matrix, but its influence is modulated by the scalar product $\bb{\beta}^{\top} \bb{\xi}$, which captures the degree of alignment between the expectation $\bb{\xi}$ and the normal vector $\bb{\beta}$.

The MIG distribution is also known to be closed under linear transformations, reproducible and infinitely divisible; refer to Properties~2,~3,~4 of \cite{MR2019130} for precise statements.

The proposition below specifies the maximum likelihood and method-of-moments estimators for the MIG model. The maximum likelihood derivation appears in Section~\ref{app:MLE}, while the expression for the method-of-moments estimator follows directly from \eqref{eq:esp.var.multivariate}.

\begin{proposition}[Maximum likelihood and method-of-moments estimators]\label{prop:MLE}
Let $\bb{\beta}\in \R^d$ be given. For an iid random sample $\bb{X}_1,\ldots,\bb{X}_n\sim \mathrm{MIG}(\bb{\beta},\bb{\xi},\mat{\Omega})$, consider the problem of estimating the pair $(\bb{\xi},\mat{\Omega})\in \mathcal{H}(\bb{\beta}) \times \mathcal{S}_{++}^d$. Then the maximum likelihood estimator is $(\bb{\xi}_n^{\star},\mat{\Omega}_n^{\star})$, where
\[
\bb{\xi}_n^{\star} = \frac{1}{n}\sum_{i=1}^n \bb{X}_i \equiv \bar{\bb{X}}_n,
\quad
\mat{\Omega}_n^{\star} = \frac{1}{n} \sum_{i=1}^n \frac{1}{\bb{\beta}^{\top} \bb{X}_i}(\bb{X}_i - \bar{\bb{X}}_n)(\bb{X}_i - \bar{\bb{X}}_n)^{\top}.
\]
Furthermore, the method-of-moments estimator is $(\bb{\xi}_n^{\dagger},\mat{\Omega}_n^{\dagger})$, where
\[
\bb{\xi}_n^{\dagger} = \bar{\bb{X}}_n,
\quad
\mat{\Omega}_n^{\dagger} = \frac{1}{\bb{\beta}^{\top} \bar{\bb{X}}_n}\frac{1}{n} \sum_{i=1}^n (\bb{X}_i - \bar{\bb{X}}_n)(\bb{X}_i - \bar{\bb{X}}_n)^{\top}.
\]
\end{proposition}

The next proposition restates a known stochastic representation of MIG random vectors originally established by \citet{MR2019130}. This stochastic representation serves as the basis for the exact simulation algorithm for MIG random vectors introduced in Section~\ref{sec:rng.MIG}. The projection matrices $\mat{Q}$ and $\mat{Q}_2$ are also used to build Gaussian kernel estimators on the rotated half-space $\R_{+} \times \R^{d-1}$ for the simulation study in Section~\ref{sec:simulation.study}.

\begin{proposition}[Stochastic representation]\label{prop:MIG.stochastic.representation}
Consider a linear map $\bb{x} \mapsto \mat{Q}\bb{x}$ from $\mathcal{H}(\bb{\beta})$ to $\R_{+} \times \R^{d-1}$, where $\mat{Q} = (\bb{\beta}^{\top}, \mat{Q}_2)^{\top}$. The first row of the matrix $\mat{Q}$ is $\bb{\beta}^{\top}$, and the remaining rows come from the $(d-1) \times d$ matrix $\mat{Q}_2$. The latter is built from the set of the first $d-1$ eigenvectors of the orthogonal projection matrix $\mat{I}_d - \bb{\beta}\bb{\beta}^{\top}/(\bb{\beta}^{\top}\bb{\beta})$, so that $\mat{Q}_2^{\top}\bb{\beta} = \bb{0}_{d-1}$ and $\mat{Q}_2\mat{Q}_2^{\top} = \mat{I}_{d-1}$. In this setting, Theorem~1~(3) of \cite{MR2019130} states that if
\begin{equation}\label{eq:Z2}
\begin{aligned}
&R \sim \mathrm{IG} \big[ \bb{\beta}^{\top}\bb{\xi}, (\bb{\beta}^{\top}\bb{\xi})^2/(\bb{\beta}^{\top}\mat{\Omega}\bb{\beta}) \big], \\
&\bb{Z} \mid \{R = r\} \sim \mathcal{N}_{d-1} [\bb{\mu}(r), r\bb{\Sigma}],
\end{aligned}
\end{equation}
where
\begin{equation}\label{eq:cond.mean.var}
\bb{\mu}(r)
= \mat{Q}_2\{\bb{\xi} + \mat{\Omega}\bb{\beta} (\bb{\beta}^{\top}\mat{\Omega}\bb{\beta})^{-1} (r-\bb{\beta}^{\top}\bb{\xi})\},
\quad
\bb{\Sigma}
= (\mat{Q}^{\vphantom{\top}}_2\mat{\Omega}^{-1}\mat{Q}_2^{\top})^{-1},
\end{equation}
then $(R, \bb{Z}^{\top}){\vphantom{Z}}^{\top}$ is equal in distribution to $\mat{Q}\bb{X}$, where $\bb{X}\sim \mathrm{MIG}(\bb{\beta}, \bb{\xi}, \mat{\Omega})$.
\end{proposition}

\section{Normal approximations}\label{sec:normal.approximations}

\subsection{The univariate case}\label{sec:univariate.IG}

Consider the most common parametrization of the inverse Gaussian distribution, as defined, e.g., in Chapter~15 of~\cite{MR1299979}. For any positive reals $\mu,\lambda\in (0,\infty)$, a random variable $X$ is said to be $\mathrm{IG}(\mu,\lambda)$-distributed, and one writes $X\sim \mathrm{IG}(\mu,\lambda)$, if its density is given, for all $x\in (0,\infty)$, by
\begin{equation}\label{eq:IG.distribution.two.parameters}
k_{\mu,\lambda}(x) = \sqrt{\frac{\lambda}{2\pi x^3}} \exp\left\{-\frac{\lambda (x - \mu)^2}{2 \mu^2 x}\right\}.
\end{equation}
The expectation, variance and skewness are well known to be
\begin{equation}\label{eq:esp.var}
\EE(X) = \mu, \qquad \Var(X) = \frac{\mu^3}{\lambda}, \qquad \frac{\EE\{(X - \mu)^3\}}{\{\Var(X)\}^{3/2}} = 3 \sqrt{\frac{\mu}{\lambda}} \,,
\end{equation}
as reported, e.g., in Section~15.4 of~\cite{MR1299979}. In particular, this means that $\mu$ is a location parameter and that the ratio $\mu/\lambda$ controls the degree of right asymmetry in the distribution.

To facilitate comparison with the MIG distribution in Definition~\ref{def:MIG}, it is convenient to reparametrize the univariate density~\eqref{eq:IG.distribution.two.parameters} as
\begin{equation}\label{eq:IG.distribution.two.parameters.omega}
k_{\mu,\mu^2/\omega}(x) = \frac{\mu}{\sqrt{2\pi \omega x^3}} \exp\left\{-\frac{(x - \mu)^2}{2 \omega x}\right\},
\end{equation}
where $\mu\in (0,\infty)$ as before and $\omega = \mu^2/\lambda\in (0,\infty)$. Under this new parametrization, the expectation, variance and skewness become
\[
\EE(X) = \mu, \qquad \Var(X) = \mu \omega, \qquad \frac{\EE\{(X - \mu)^3\}}{\{\Var(X)\}^{3/2}} = 3 \sqrt{\frac{\omega}{\mu}} \, .
\]
Set $d=1$ and $\bb{\beta} = 1$ in \eqref{eq:general.multivariate.density} and \eqref{eq:esp.var.multivariate} to recover the above.

It is well known that if $X\sim \mathrm{IG}(\mu,\lambda)$, then, as $\lambda/\mu\to \infty$, or equivalently as $\mu/\omega\to \infty$, one has
\begin{equation}\label{eq:CLT.univariate}
\delta_X \equiv \frac{X-\mu}{\sqrt{\mu^3/\lambda}} \rightsquigarrow \mathcal{N}(0,1);
\end{equation}
see, e.g., \citet[p.~5]{MR1622488}. Here and in what follows, $\rightsquigarrow$ refers to convergence in distribution.

The purpose of this section is to study, as $\lambda/\mu\to \infty$ and for any given point $x$, the log-ratio $\mathrm{LR}(x) \equiv \ln\{k_{\mu,\lambda}(x)/\phi_{\mu,\mu^3/\lambda}(x)\}$ between the inverse Gaussian density $k_{\mu,\lambda}(x)$ and the Gaussian density $\phi_{\mu,\mu^3/\lambda}(x)$ which has the same mean and variance. The latter is defined, for all $x\in \R$, by
\begin{equation}\label{eq:normal.distribution}
\phi_{\mu,\mu^3/\lambda}(x) = \sqrt{\frac{\lambda}{2\pi \mu^3}} \exp\left\{-\frac{\lambda(x - \mu)^2}{2 \mu^3}\right\}.
\end{equation}
As long as $x$ is located in the {\it bulk} of the inverse Gaussian distribution (i.e., not too far from the mean $\mu$), the terms in the series expansion of the log-ratio $\mathrm{LR}(x)$ will be shown explicitly in Theorem~\ref{thm:LLT}, along with the rate of convergence of any truncation.

For any $\mu, \lambda\in (0, \infty)$ and $\tau\in (0, \sqrt{\lambda/\mu}]$, the bulk of radius $\tau$ for the inverse Gaussian distribution, and its closure, are defined, respectively, by
\[
\begin{aligned}
&B_{\mu,\lambda}(\tau) = \left\{x\in (0,\infty) : |\delta_x| < \tau\right\}, \\
&\overline{B_{\mu,\lambda}}(\tau) = \left\{x\in (0,\infty) : |\delta_x| \leq \tau\right\}.
\end{aligned}
\]
Note that, by selecting $\tau$ large enough, the probability that $X\sim \mathrm{IG}(\mu,\lambda)$ lands into the bulk region $B_{\mu,\lambda}(\tau)$ has probability arbitrarily close to $1$ as $\lambda/\mu\to \infty$.

\begin{theorem}[Local limit theorem]\label{thm:LLT}
For any positive reals $\mu,\lambda\in (0,\infty)$ and $x\in \smash{B_{\mu,\lambda}(\sqrt{\lambda/\mu})}$, one has
\begin{equation}\label{eq:LLT.univariate.1}
\mathrm{LR}(x) \equiv \ln\left\{\frac{k_{\mu,\lambda}(x)}{\phi_{\mu,\mu^3/\lambda}(x)}\right\} = \sum_{k=1}^{\infty} (-1)^k \left(\frac{3}{2k} - \frac{\delta_x^2}{2}\right) \left(\delta_x \sqrt{\frac{\mu}{\lambda}}\right)^k.
\end{equation}
Furthermore, given any integer $n\in \N$ and any real $\tau\in (0, \sqrt{\lambda/\mu})$, one has, uniformly for all $x\in \overline{B_{\mu,\lambda}}(\tau)$ and as $\lambda/\mu\to \infty$,
\begin{equation}\label{eq:LLT.univariate.2}
\mathrm{LR}(x) - \sum_{k=1}^{n-1} (-1)^k \left(\frac{3}{2k} - \frac{\delta_x^2}{2}\right) \left(\delta_x \sqrt{\frac{\mu}{\lambda}}\right)^k = \OO_{\tau}\left\{\left(\frac{\mu}{\lambda}\right)^{n/2}\right\}.
\end{equation}
In particular, if $n=1$, one has, as $\lambda/\mu\to \infty$,
\begin{equation}\label{eq:LLT.univariate.3}
\mathrm{LR}(x) = \OO_{\tau}\big(\sqrt{ {\mu}/{\lambda}}\big).
\end{equation}
\end{theorem}

The expansion of the log-ratio in Theorem~\ref{thm:LLT} builds on two general ingredients: the infinite divisibility of the inverse Gaussian distribution, which makes it possible to represent an inverse Gaussian random variable as a sum of independent inverse Gaussian random variables with suitably chosen parameters, and the classical Edgeworth series for sums of independent random variables, originally introduced by \citet{Edgeworth1905}.

The convergence properties of Edgeworth series were rigorously analyzed by \cite{MR14626} through Fourier methods and earlier results of Cram\'er. In their general form, these expansions are well documented; see, e.g., \cite{MR1295242} and \cite{MR3396213}. A key contribution of the present article is to provide an explicit derivation of such expansions for the inverse Gaussian distribution, where explicit computations of this kind are less commonly found in the literature. Analogous explicit expansions were previously obtained for the central and noncentral chi-square distributions by \citet{MR4466042}.

It is worth noting that Edgeworth-type expansions are applicable to both continuous and lattice distributions. For instance, \cite{MR207011} was the first to employ the Fourier approach of Esseen to derive explicit local approximations, akin to those in Theorem~\ref{thm:LLT}, for the Poisson, binomial, and negative binomial distributions.

An alternative elementary method based on Taylor expansions and Stirling's formula was introduced by \cite{MR538319} to obtain local approximations for the binomial distribution, along with refined continuity corrections. Non-asymptotic extensions of those results were provided by \cite{MR4340237} to sharpen Tusnády's inequality in the ``bulk'' regime, an important ingredient in the Komlós--Major--Tusnády approximation controlling deviations between the empirical process and a corresponding sequence of Brownian bridges. Cressie's line of reasoning was also adapted by \cite{MR4630543} to study the asymptotic behavior of the median of the negative binomial distribution jittered by a uniform, settling an open question of \cite{MR4135709}. The proof of Theorem~\ref{thm:LLT} uses a similar elementary approach; see Appendix~\ref{app:proofs.univariate} for details.

Below, graphical evidence is provided for the validity of the log-ratio expansion in Theorem~\ref{thm:LLT}. Three levels of approximation are compared for the case $\lambda = \mu^2$ (or equivalently, $\omega=1$). Define
\[
\begin{aligned}
E_1
&= \sup_{x\in \overline{B_{\mu,\lambda}}(1)} \left|\mathrm{LR}(x)\right|, \\
E_2
&= \sup_{x\in \overline{B_{\mu,\lambda}}(1)} \left|\mathrm{LR}(x) + \left(\frac{3}{2} - \frac{\delta_x^2}{2}\right) \left(\delta_x \sqrt{\frac{\mu}{\lambda}}\right)\right|, \\
E_3
&= \sup_{x\in \overline{B_{\mu,\lambda}}(1)} \left|\mathrm{LR}(x) + \left(\frac{3}{2} - \frac{\delta_x^2}{2}\right) \left(\delta_x \sqrt{\frac{\mu}{\lambda}}\right) - \left(\frac{3}{4} - \frac{\delta_x^2}{2}\right) \left(\delta_x \sqrt{\frac{\mu}{\lambda}}\right)^2\right|.
\end{aligned}
\]
Recall that $x\in \overline{B_{\mu,\lambda}}(1)$ implies $|\delta_x| \leq 1$, so one expects from~\eqref{eq:LLT.univariate.2} that the maximal errors $E_n$ above will have, for every integer $n\in \{1,2,3\}$, the asymptotic behavior
\begin{equation}\label{eq:liminf.exponent.bound.unidim}
\frac{1}{E_n} = \OO\left\{\left(\frac{\mu}{\lambda}\right)^{-n/2}\right\} = \OO(\mu^{n/2}) \quad \Leftrightarrow \quad \liminf_{\mu\to \infty} - \frac{\ln(E_n)}{\ln(\mu)} \geq \frac{n}{2}.
\end{equation}
The two equivalent properties in~\eqref{eq:liminf.exponent.bound.unidim} are illustrated on the left-hand side and right-hand side of Figure~\ref{fig:unidim}, respectively. For a link to the \textsf{R} code used to generate this figure, refer to Appendix~\ref{app:R.code}.

\begin{figure}[!ht]
\centering\hspace{-3mm}
\includegraphics[width = 1\linewidth]{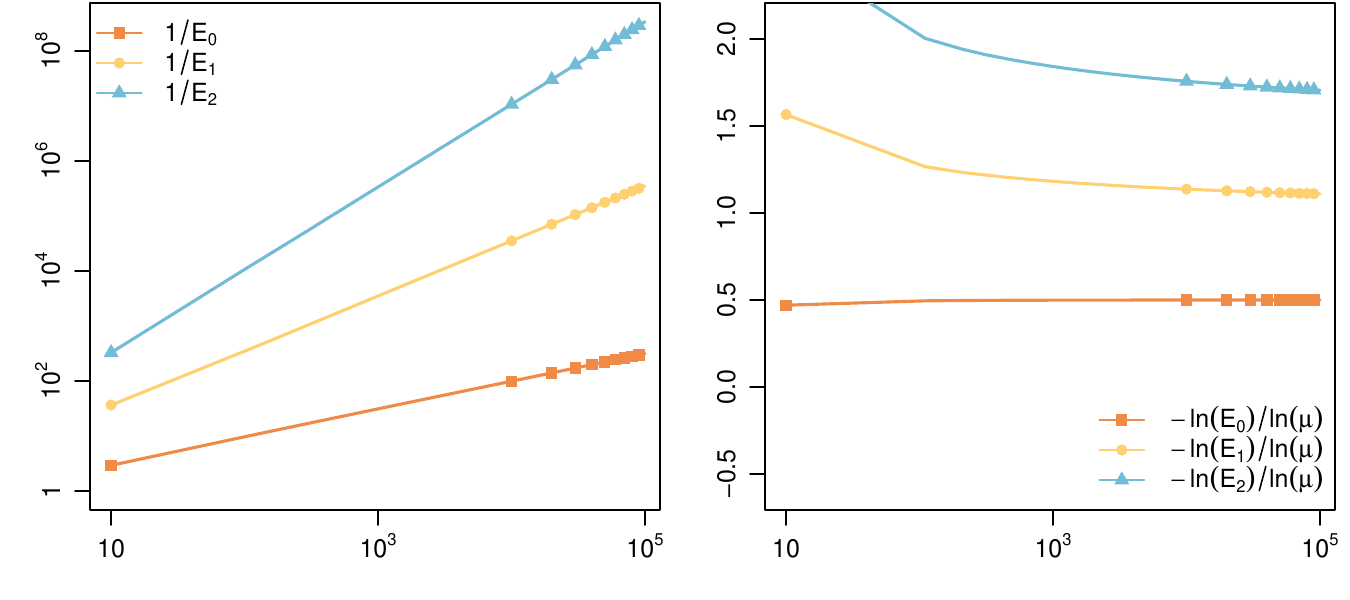}\vspace{-2mm}
\caption{The plot on the left-hand side displays $1/E_n$ as a function of $\mu$, utilizing a logarithmic scale for both the horizontal and vertical axes. It illustrates the improvement achieved by incorporating correction terms from Theorem~\ref{thm:LLT} to the base approximation. On the right-hand side, the plot displays $-\ln(E_n)/\ln(\mu)$ as a function of $\mu$, with the horizontal axis logarithmically scaled. The plot confirms the asymptotic orders of the liminfs in~\eqref{eq:liminf.exponent.bound.unidim} and provides compelling numerical evidence for the validity of Theorem~\ref{thm:LLT}.}\label{fig:unidim}
\end{figure}

By applying the local approximations of Theorem~\ref{thm:LLT} in the bulk and by showing that the contributions outside the bulk are negligible using concentration inequalities, probability metric bounds between the measures induced by the densities $k_{\mu,\lambda}$ and $\phi_{\mu, \mu^3/\lambda}$ are established. The following is a global result.

\begin{theorem}[Probability metric bounds]\label{thm:total.variation}
Let $\PP_{\mu,\lambda}$ be the measure on $\R$ induced by the IG density $k_{\mu,\lambda}$ in~\eqref{eq:IG.distribution.two.parameters}. Let $\QQ_{\mu,\lambda}$ be the measure on $\R$ induced by the Gaussian density $\phi_{\mu, \mu^3/\lambda}$ in~\eqref{eq:normal.distribution}. Then, for any positive reals $\mu,\lambda\in (0,\infty)$, one has
\begin{equation}\label{eq:bound.Hellinger.univariate}
\mathfrak{H}(\PP_{\mu,\lambda},\QQ_{\mu,\lambda}) \leq C \sqrt{{\mu}/{\lambda}},
\end{equation}
where $C\in (0,\infty)$ is a universal constant and $\mathfrak{H}(\cdot,\cdot)$ denotes the Hellinger distance. The bound~\eqref{eq:bound.Hellinger.univariate} is also valid if one replaces the Hellinger distance by any of the following probability metrics: discrepancy metric, Kolmogorov (or uniform) metric, L\'evy metric, Prokhorov metric, total variation.
\end{theorem}

Global normal approximations represent another direction in this area, with Stein's method \citep{MR3194737,MR3654808,MR3655852,AnastasiouEtAl2023} providing a particularly effective approach for bounding distributional distances under less restrictive moment or dependence assumptions. Although this perspective is beyond the scope of the present work, it offers a promising avenue for further research.

\subsection{The multivariate case}\label{sec:multivariate.IG}

In order to study the limiting behavior of the MIG density in Definition~\ref{def:MIG}, the notation is extended by setting $\bb{\xi}=\mu \bb{\xi}_0$ and $\mat{\Omega} = \omega \mat{\Omega}_0$, where the factors $\mu,\omega\in (0,\infty)$ are positive reals, the vector $\bb{\xi}_0\in \R^d$ satisfies $\bb{\beta}^{\top} \bb{\xi}_0 \in (0,\infty)$, and $\mat{\Omega}_0\in \mathcal{S}_{++}^d$ is a positive definite matrix of size $d\times d$. Set $d = \bb{\beta} = \bb{\xi}_0 = \mat{\Omega}_0 = 1$ to obtain the univariate density function in~\eqref{eq:IG.distribution.two.parameters.omega}.

It is known that if $\bb{X}\sim \mathrm{MIG}(\bb{\beta},\mu \bb{\xi}_0,\omega \mat{\Omega}_0)$, then, as $\mu/\omega\to \infty$, one has
\begin{equation}\label{eq:CLT.multivariate}
\bb{\delta}_{\bb{X}} \equiv (\mu\omega \bb{\beta}^{\top} \bb{\xi}_0)^{-1/2} \smash{\mat{\Omega}_0^{-1/2}} (\bb{X} - \mu \bb{\xi}_0) \rightsquigarrow \mathcal{N}_d(\bb{0}_d,\mat{I}_d),
\end{equation}
where $\smash{\mat{\Omega}_0^{1/2}}$ denotes the symmetric square root of $\mat{\Omega}_0$. This is a straightforward consequence of the expression for the MIG cumulant generating function in Theorem~1 of~\cite{MR2019130} combined with L\'evy's continuity theorem for Laplace transforms; see, e.g., \citet[p.~431]{MR0270403}.

As in the univariate case, the goal is to study, as $\mu/\omega\to \infty$ and for any given point $\bb{x}$, the log-ratio $\mathrm{LR}(\bb{x}) \equiv \ln\{k_{\bb{\beta},\mu \bb{\xi}_0, \omega \mat{\Omega}_0}(\bb{x})/\phi_{\mu \bb{\xi}_0, \mu \omega \mat{\Omega}_0}(\bb{x})\}$ between the MIG density $k_{\bb{\beta},\mu \bb{\xi}_0, \omega \mat{\Omega}_0}(\bb{x})$ and the corresponding multivariate normal density $\phi_{\mu \bb{\xi}_0, \mu \omega \mat{\Omega}_0}(\bb{x})$ with the same mean vector and covariance matrix. The latter is defined, for all $\bb{x}\in \R^d$, by
\begin{equation}\label{eq:normal.distribution.multivariate}
\begin{aligned}
\phi_{\mu \bb{\xi}_0,\mu\omega \bb{\beta}^{\top} \bb{\xi}_0\mat{\Omega}_0}(\bb{x})
&= (2\pi)^{-d/2}(\mu\omega \bb{\beta}^{\top} \bb{\xi}_0)^{-d/2} |\mat{\Omega}_0|^{-1/2} \\
&\qquad\times \exp\left\{-\frac{1}{2\mu\omega \bb{\beta}^{\top} \bb{\xi}_0} (\bb{x} - \mu \bb{\xi}_0)^{\top} \mat{\Omega}_0^{-1} (\bb{x} - \mu \bb{\xi}_0)\right\}.
\end{aligned}
\end{equation}

In particular, as long as $\bb{x}$ is located in the {\it bulk} of the MIG distribution (i.e., not too far from the mean $\mu \bb{\xi}_0$), the terms in the series expansion of the log-ratio $\mathrm{LR}(\bb{x})$ will be shown explicitly in Theorem~\ref{thm:LLT.multivariate}, along with the rate of convergence of any truncation. This result rests indirectly on the fact that the MIG distribution is infinitely divisible, as per Property~4 of~\cite{MR2019130}.

For any positive reals $\mu,\omega\in (0,\infty)$ and $\tau\in (0, \sqrt{\mu/\omega}]$, the bulk of radius $\tau$ for the MIG distribution, and its closure, are defined, respectively, by
\[
\begin{aligned}
&B_{\mu,\omega,\bb{\beta},\bb{\xi}_0,\mat{\Omega}_0}(\tau) = \left\{\bb{x}\in \mathcal{H}(\bb{\beta}) : \frac{\big|\bb{\beta}^{\top} \mat{\Omega}_0^{1/2} \bb{\delta}_{\bb{x}}\big|}{\sqrt{\bb{\beta}^{\top} \bb{\xi}_0}} < \tau\right\}, \\
&\overline{B_{\mu,\omega,\bb{\beta},\bb{\xi}_0,\mat{\Omega}_0}}(\tau) = \left\{\bb{x}\in \mathcal{H}(\bb{\beta}) : \frac{\big|\bb{\beta}^{\top} \mat{\Omega}_0^{1/2} \bb{\delta}_{\bb{x}}\big|}{\sqrt{\bb{\beta}^{\top} \bb{\xi}_0}} \leq \tau\right\}.
\end{aligned}
\]

Note that, by selecting a large enough $\tau$, the probability that $\bb{X}\sim \mathrm{MIG}(\bb{\beta},\bb{\xi},\mat{\Omega})$ lands into the bulk region $B_{\mu,\omega,\bb{\beta},\bb{\xi}_0,\mat{\Omega}_0}(\tau)$ has probability arbitrarily close to $1$ as $\mu/\omega\to \infty$.

\begin{theorem}[Local limit theorem]\label{thm:LLT.multivariate}
Let the real vectors $\bb{\beta}\in \R^d$, $\bb{\xi}_0\in \mathcal{H}(\bb{\beta})$, and the positive definite matrix $\mat{\Omega}_0\in \mathcal{S}_{++}^d$ be given. For any positive reals $\mu,\omega\in (0,\infty)$ and any real vector $\bb{x}\in \smash{B_{\mu,\omega,\bb{\beta},\bb{\xi}_0,\mat{\Omega}_0}(\sqrt{\mu/\omega})}$, one has
\begin{equation}\label{eq:LLT.multivariate.1}
\begin{aligned}
\mathrm{LR}(\bb{x})
&\equiv \ln\left\{\frac{k_{\bb{\beta},\mu \bb{\xi}_0, \omega\mat{\Omega}_0}(\bb{x})}{\phi_{\mu \bb{\xi}_0,\mu \omega \bb{\beta}^{\top} \bb{\xi}_0 \mat{\Omega}_0}(\bb{x})}\right\} \\
&= \sum_{k=1}^{\infty} (-1)^k \left(\frac{d + 2}{2 k} - \frac{\bb{\delta}_{\bb{x}}^{\top} \bb{\delta}_{\bb{x}}}{2}\right) \left(\frac{\bb{\beta}^{\top} \mat{\Omega}_0^{1/2} \bb{\delta}_{\bb{x}}}{\sqrt{\bb{\beta}^{\top} \bb{\xi}_0}} \sqrt{\frac{\omega}{\mu}}\right)^k.
\end{aligned}
\end{equation}
Furthermore, given any integer $n\in \N$ and any real $\tau\in (0, \sqrt{\mu/\omega})$, one has, uniformly for $\bb{x}\in \overline{B_{\mu,\omega,\bb{\beta},\bb{\xi}_0,\mat{\Omega}_0}}(\tau)$ and as $\mu/\omega\to \infty$,
\begin{multline}\label{eq:LLT.multivariate.2}
\mathrm{LR}(\bb{x}) - \sum_{k=1}^{n-1}(-1)^k \left(\frac{d + 2}{2 k} - \frac{\bb{\delta}_{\bb{x}}^{\top} \bb{\delta}_{\bb{x}}}{2}\right) \left(\frac{\bb{\beta}^{\top} \mat{\Omega}_0^{1/2} \bb{\delta}_{\bb{x}}}{\sqrt{\bb{\beta}^{\top} \bb{\xi}_0}} \sqrt{\frac{\omega}{\mu}}\right)^k \\
= \OO_{\tau,\bb{\beta},\bb{\xi}_0,\mat{\Omega}_0}\left\{\left(\frac{\omega}{\mu}\right)^{n/2}\right\}.
\end{multline}
In particular, if $n = 1$, one has, as $\mu/\omega\to \infty$,
\begin{equation}\label{eq:LLT.multivariate.3}
\mathrm{LR}(\bb{x}) = \OO_{\tau,\bb{\beta},\bb{\xi}_0,\mat{\Omega}_0}\big(\sqrt{{\omega}/{\mu}}\big).
\end{equation}
\end{theorem}

As mentioned in Section~\ref{sec:univariate.IG}, Edgeworth series and the Fourier method provide a well-established framework for deriving local distributional approximations. For the multivariate case, these methods are extensively detailed in Chapter 2 of \citet{MR3396213}. The derivations here instead rely on an elementary method using Taylor expansions applied directly to the density function. Explicit computations of this kind for specific distributions are less commonly found in the literature. Notable examples include \citet{MR4394974,MR4449387,MR4358612}, where explicit local expansions have been derived for the Dirichlet, matrix $T$, and Wishart distributions, respectively.

Graphical evidence is provided below for the validity of the log-ratio expansion in Theorem~\ref{thm:LLT.multivariate} when $d = 2$, $\bb{\beta} = (1/2, 1/2)^{\top}$ and $\bb{\xi}_0 = (1, 1)^{\top}$, so that $\smash{\bb{\beta}^{\top} \bb{\xi}_0} = 1$ in particular. The case in which both $\mu\to \infty$ and $\omega=1$ is considered. Three levels of approximation are compared for various choices of $\mat{\Omega}_0$. Define
\[
\begin{aligned}
E_1
&= \sup_{\bb{x}\in \overline{B_{\mu,\omega,\bb{\beta},\bb{\xi}_0,\mat{\Omega}_0}}(1)} \left|\mathrm{LR}(\bb{x})\right|, \\
E_2
&= \sup_{\bb{x}\in \overline{B_{\mu,\omega,\bb{\beta},\bb{\xi}_0,\mat{\Omega}_0}}(1)} \left|\mathrm{LR}(\bb{x}) + \left(\frac{d + 2}{2} - \frac{\bb{\delta}_{\bb{x}}^{\top} \bb{\delta}_{\bb{x}}}{2}\right) \left(\frac{\bb{\beta}^{\top} \mat{\Omega}_0^{1/2} \bb{\delta}_{\bb{x}}}{\sqrt{\bb{\beta}^{\top} \bb{\xi}_0}} \sqrt{\frac{\omega}{\mu}}\right)\right|, \\
E_3
&= \sup_{\bb{x}\in \overline{B_{\mu,\omega,\bb{\beta},\bb{\xi}_0,\mat{\Omega}_0}}(1)} \left|\mathrm{LR}(\bb{x}) + \left(\frac{d + 2}{2} - \frac{\bb{\delta}_{\bb{x}}^{\top} \bb{\delta}_{\bb{x}}}{2}\right) \left(\frac{\bb{\beta}^{\top} \mat{\Omega}_0^{1/2} \bb{\delta}_{\bb{x}}}{\sqrt{\bb{\beta}^{\top} \bb{\xi}_0}} \sqrt{\frac{\omega}{\mu}}\right) \right. \\
&\hspace{40mm}\left. - \left(\frac{d + 2}{4} - \frac{\bb{\delta}_{\bb{x}}^{\top} \bb{\delta}_{\bb{x}}}{2}\right) \left(\frac{\bb{\beta}^{\top} \mat{\Omega}_0^{1/2} \bb{\delta}_{\bb{x}}}{\sqrt{\bb{\beta}^{\top} \bb{\xi}_0}} \sqrt{\frac{\omega}{\mu}}\right)^2\right|.
\end{aligned}
\]
\normalsize

Recall that $\bb{x}\in \overline{B_{\mu,\omega,\bb{\beta},\bb{\xi}_0,\mat{\Omega}_0}}(1)$ implies $|\bb{\beta}^{\top} \mat{\Omega}_0^{1/2} \bb{\delta}_{\bb{x}}| / (\bb{\beta}^{\top} \bb{\xi}_0)^{1/2} \leq 1$, so one expects from~\eqref{eq:LLT.multivariate.2} that the maximal errors $E_n$ above will have, for every integer $n\in \{1,2,3\}$, the asymptotic behavior
\begin{multline}\label{eq:liminf.exponent.bound}
\frac{1}{E_n} = \OO_{\bb{\beta},\bb{\xi}_0,\mat{\Omega}_0}\left\{ \left(\frac{\omega}{\mu}\right)^{-n/2} \right\} = \OO_{\bb{\beta},\bb{\xi}_0,\mat{\Omega}_0}(\mu^{n/2}) \\[-2mm]
\Leftrightarrow \quad \liminf_{\mu\to \infty} - \frac{\ln(E_n)}{\ln(\mu)} \geq \frac{n}{2}.
\end{multline}
The two equivalent properties in~\eqref{eq:liminf.exponent.bound} are illustrated in Figure~\ref{fig:figures/loglog.errors.plots} and Figure~\ref{fig:error.exponents.plots}, respectively, for various choices of $\mat{\Omega}_0$.

\begin{figure}[!htp]
\captionsetup[subfigure]{labelformat=empty}
\centering
\begin{subfigure}[b]{0.333\linewidth}
\centering
\includegraphics[width=\linewidth,height=\linewidth]{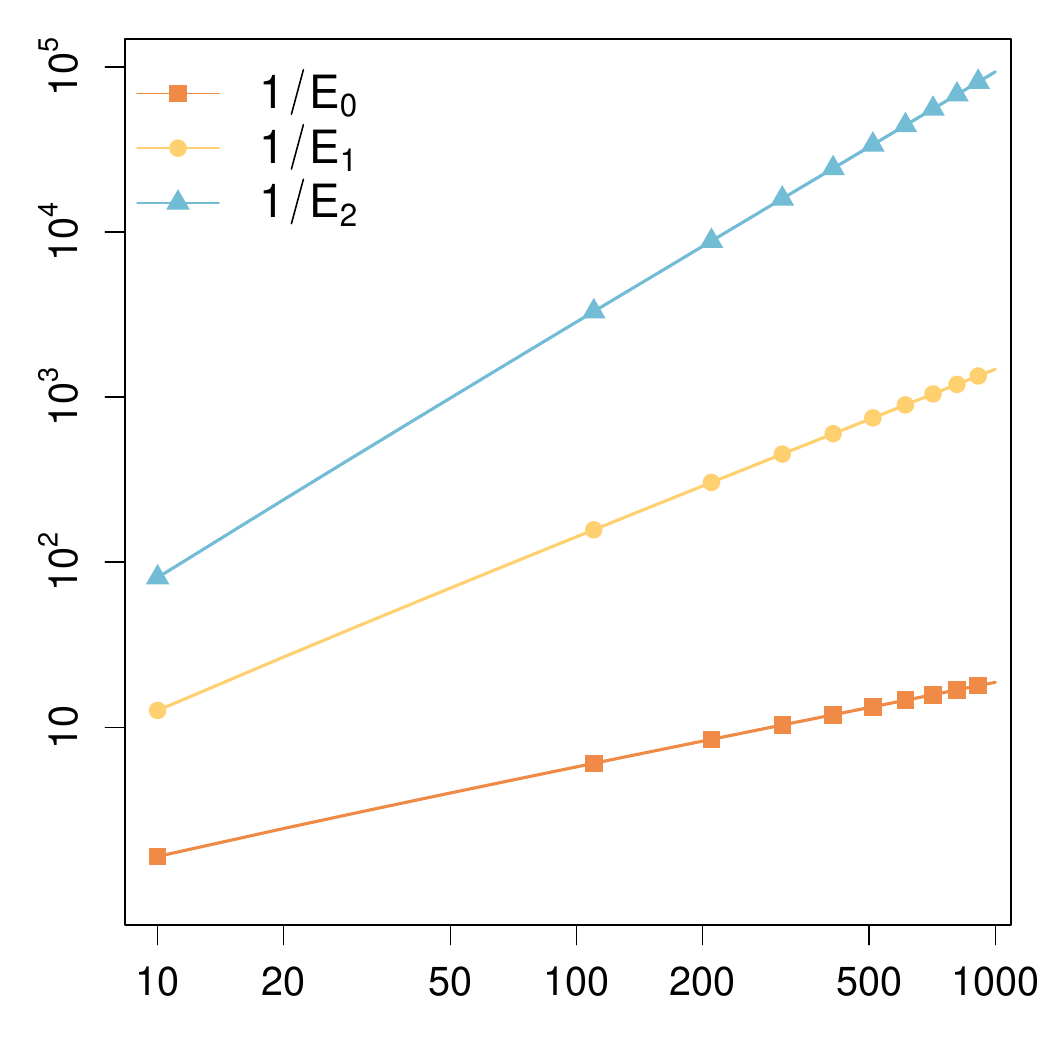}
\vspace{-0.5cm}
\caption{\footnotesize $\mat{\Omega}_0 = \left(\begin{smallmatrix} 2 & 1 \\ 1 & 2\end{smallmatrix}\right)$}
\end{subfigure}
\begin{subfigure}[b]{0.333\linewidth}
\centering
\includegraphics[width=\linewidth,height=\linewidth]{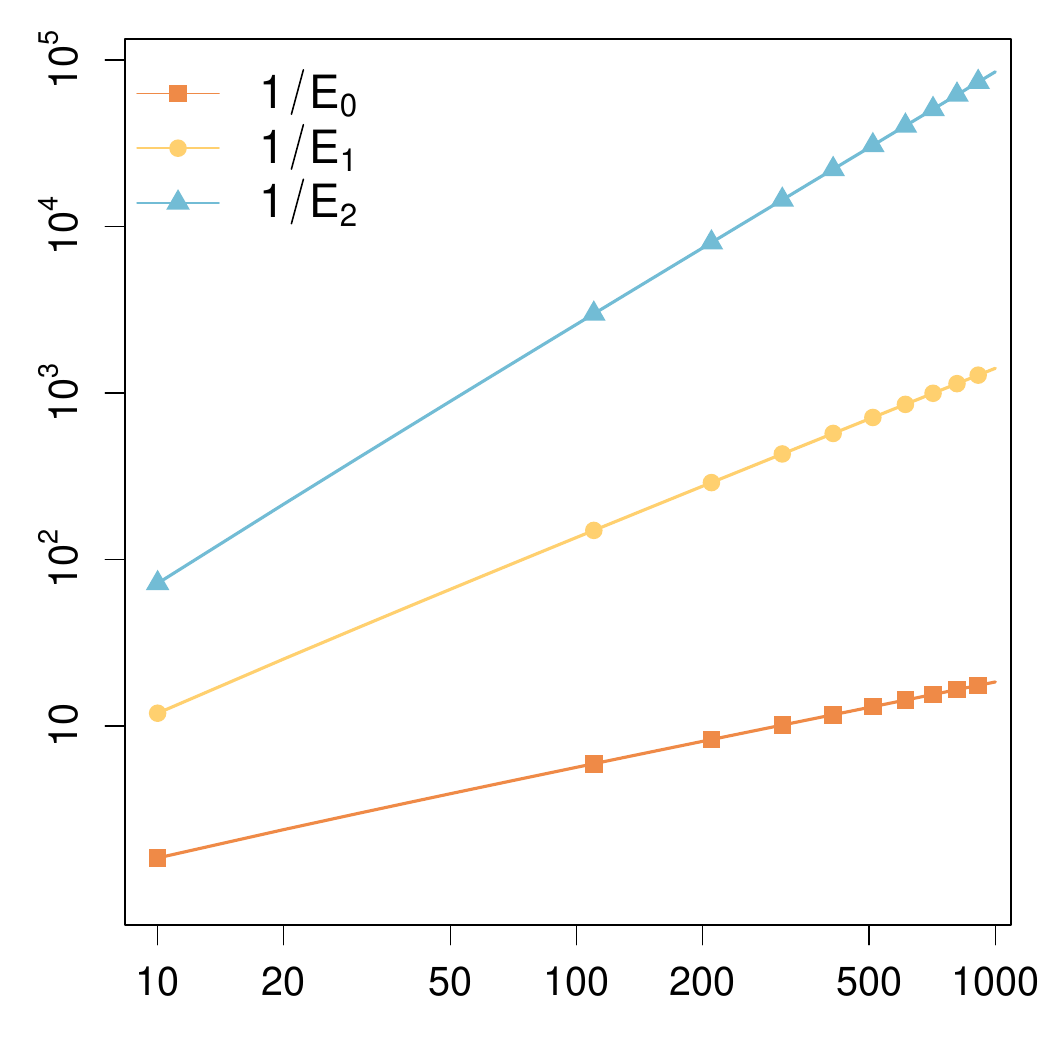}
\vspace{-0.5cm}
\caption{\footnotesize $\mat{\Omega}_0 = \left(\begin{smallmatrix} 2 & 1 \\ 1 & 3\end{smallmatrix}\right)$}
\end{subfigure}
\begin{subfigure}[b]{0.333\linewidth}
\centering
\includegraphics[width=\linewidth,height=\linewidth]{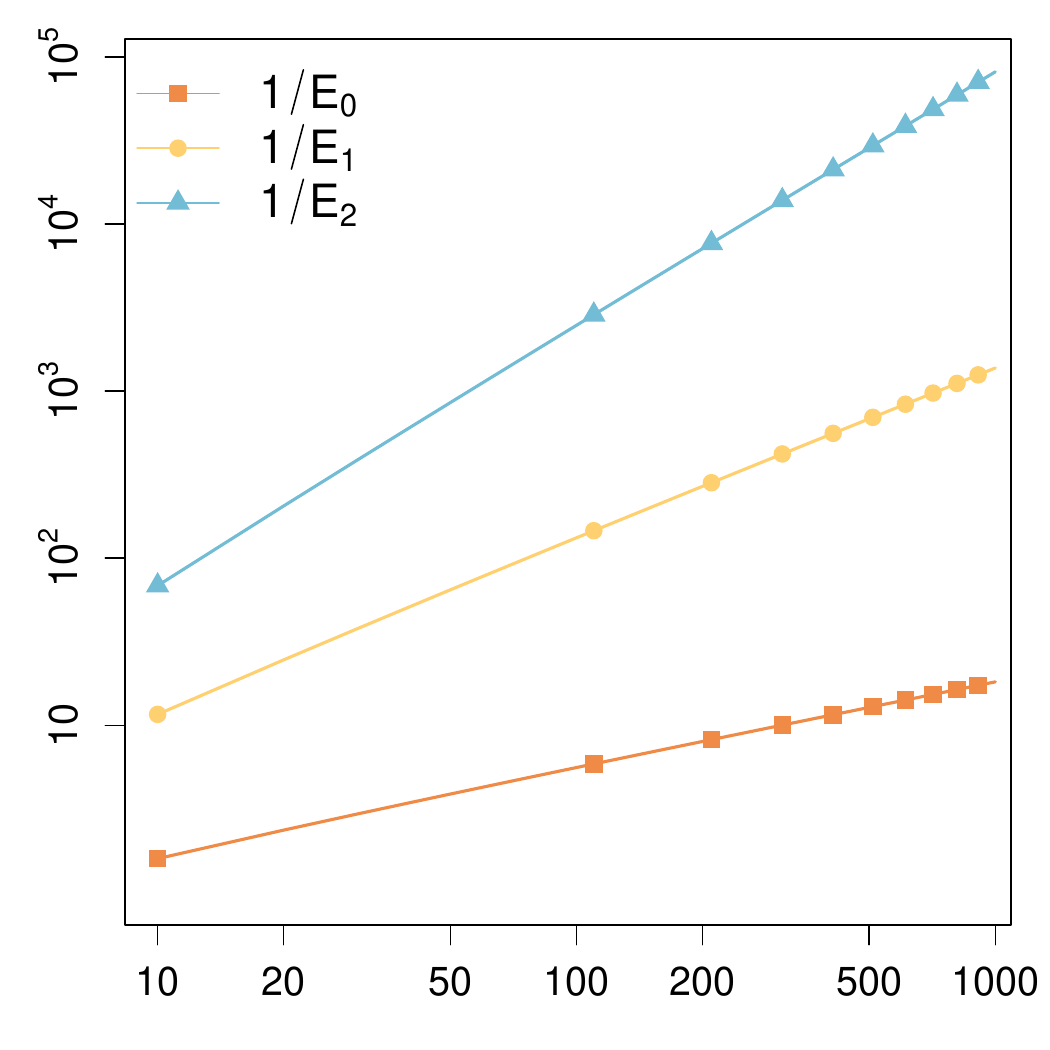}
\vspace{-0.5cm}
\caption{\footnotesize $\mat{\Omega}_0 = \left(\begin{smallmatrix} 2 & 1 \\ 1 & 4\end{smallmatrix}\right)$}
\end{subfigure}
\begin{subfigure}[b]{0.333\linewidth}
\centering
\includegraphics[width=\linewidth,height=\linewidth]{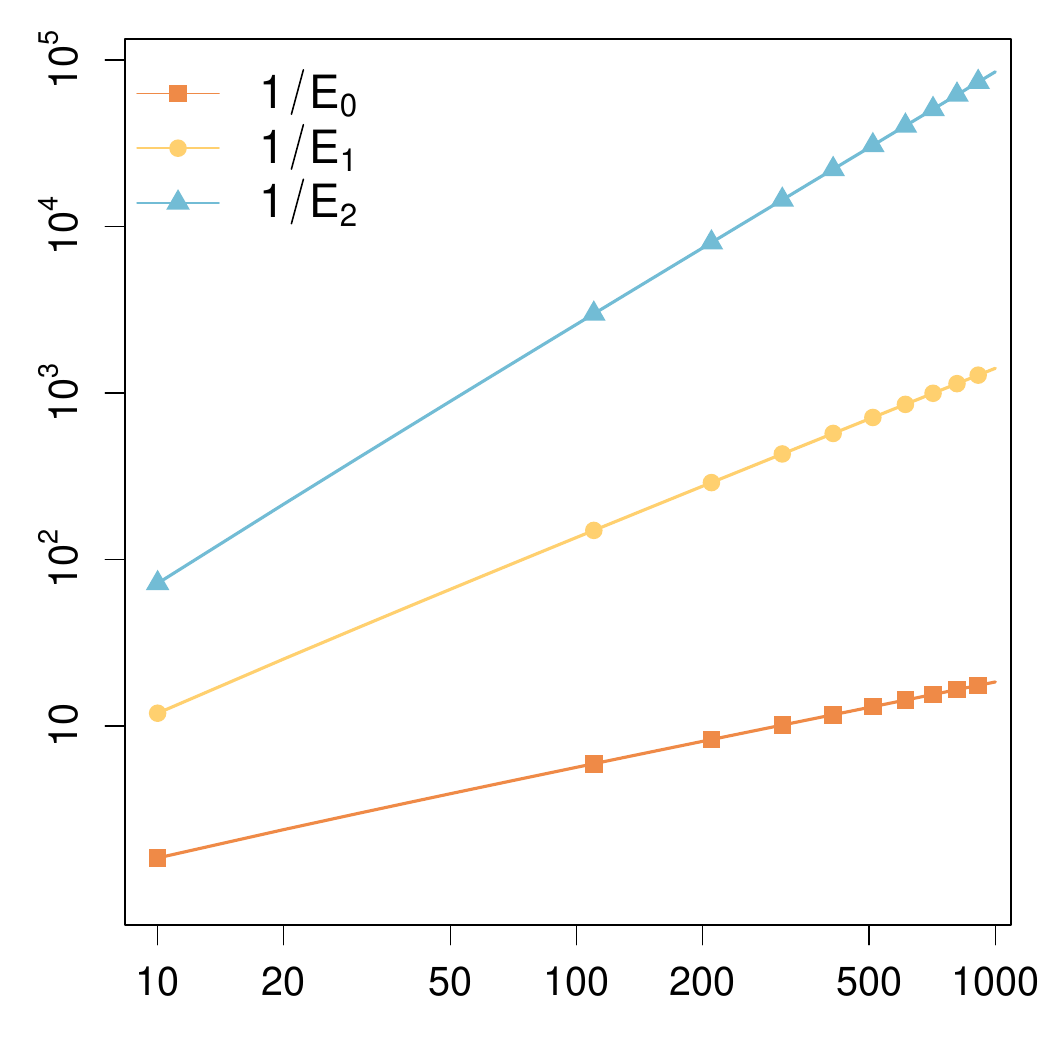}
\vspace{-0.5cm}
\caption{\footnotesize $\mat{\Omega}_0 = \left(\begin{smallmatrix} 3 & 1 \\ 1 & 2\end{smallmatrix}\right)$}
\end{subfigure}
\begin{subfigure}[b]{0.333\linewidth}
\centering
\includegraphics[width=\linewidth,height=\linewidth]{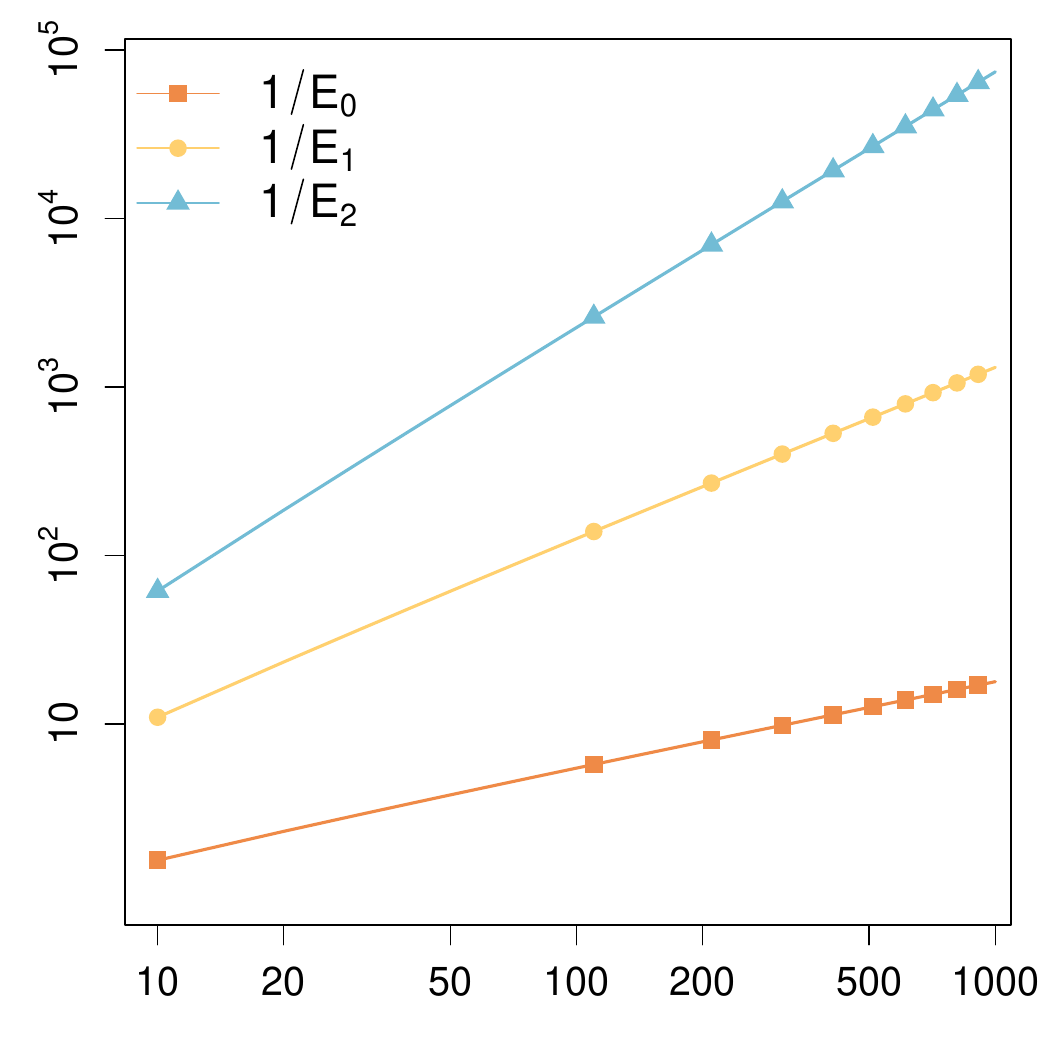}
\vspace{-0.5cm}
\caption{\footnotesize $\mat{\Omega}_0 = \left(\begin{smallmatrix} 3 & 1 \\ 1 & 3\end{smallmatrix}\right)$}
\end{subfigure}
\begin{subfigure}[b]{0.333\linewidth}
\centering
\includegraphics[width=\linewidth,height=\linewidth]{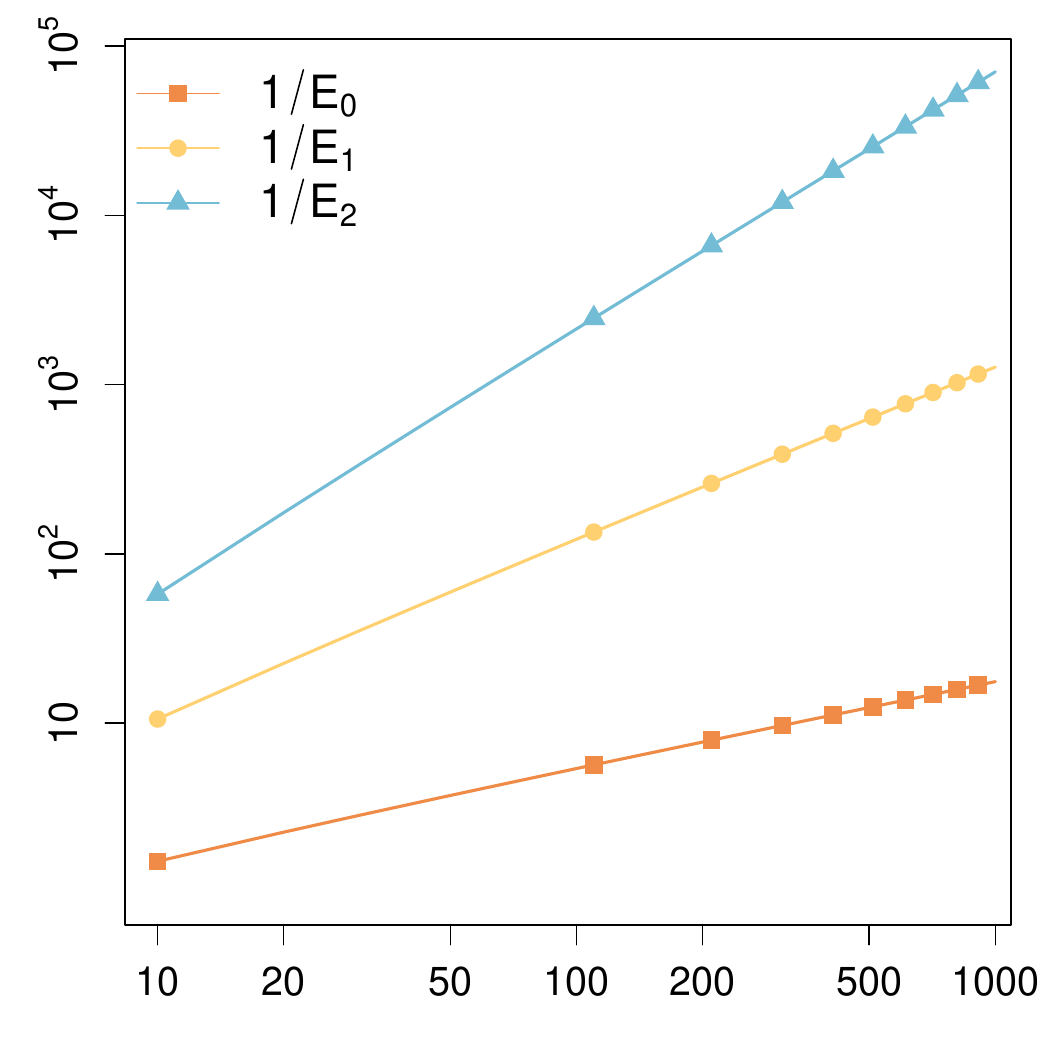}
\vspace{-0.5cm}
\caption{\footnotesize $\mat{\Omega}_0 = \left(\begin{smallmatrix} 3 & 1 \\ 1 & 4\end{smallmatrix}\right)$}
\end{subfigure}
\begin{subfigure}[b]{0.333\linewidth}
\centering
\includegraphics[width=\linewidth,height=\linewidth]{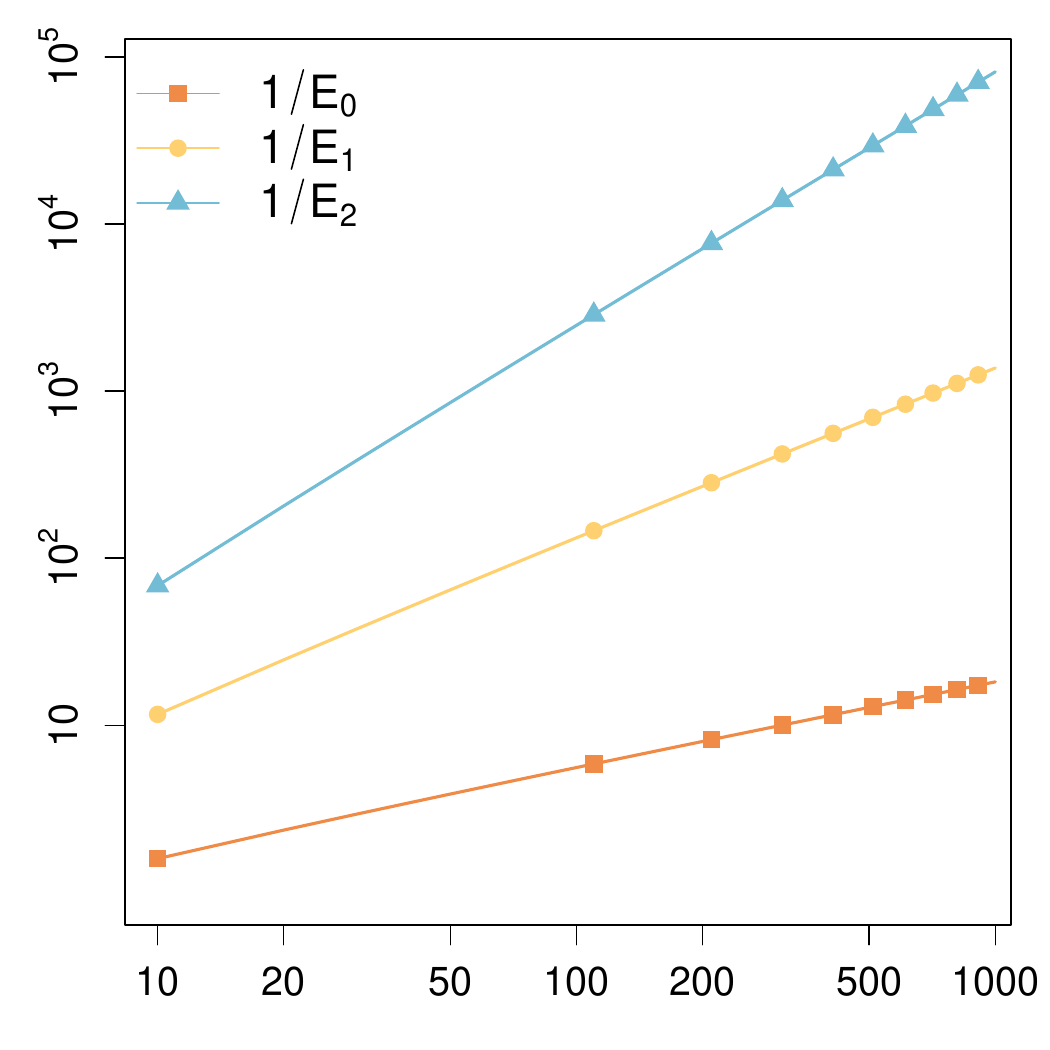}
\vspace{-0.5cm}
\caption{\footnotesize $\mat{\Omega}_0 = \left(\begin{smallmatrix} 4 & 1 \\ 1 & 2\end{smallmatrix}\right)$}
\end{subfigure}
\begin{subfigure}[b]{0.333\linewidth}
\centering
\includegraphics[width=\linewidth,height=\linewidth]{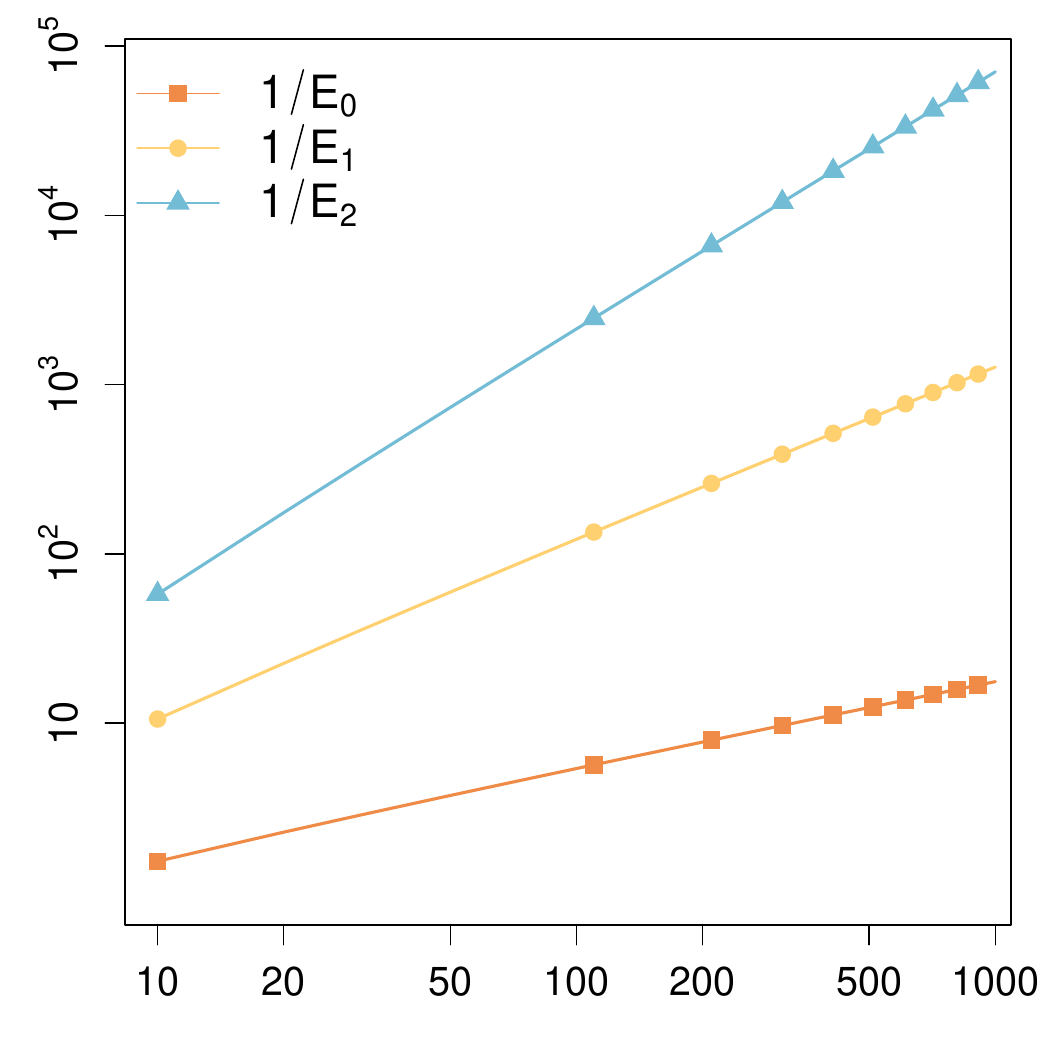}
\vspace{-0.5cm}
\caption{\footnotesize $\mat{\Omega}_0 = \left(\begin{smallmatrix} 4 & 1 \\ 1 & 3\end{smallmatrix}\right)$}
\end{subfigure}
\begin{subfigure}[b]{0.333\linewidth}
\centering
\includegraphics[width=\linewidth,height=\linewidth]{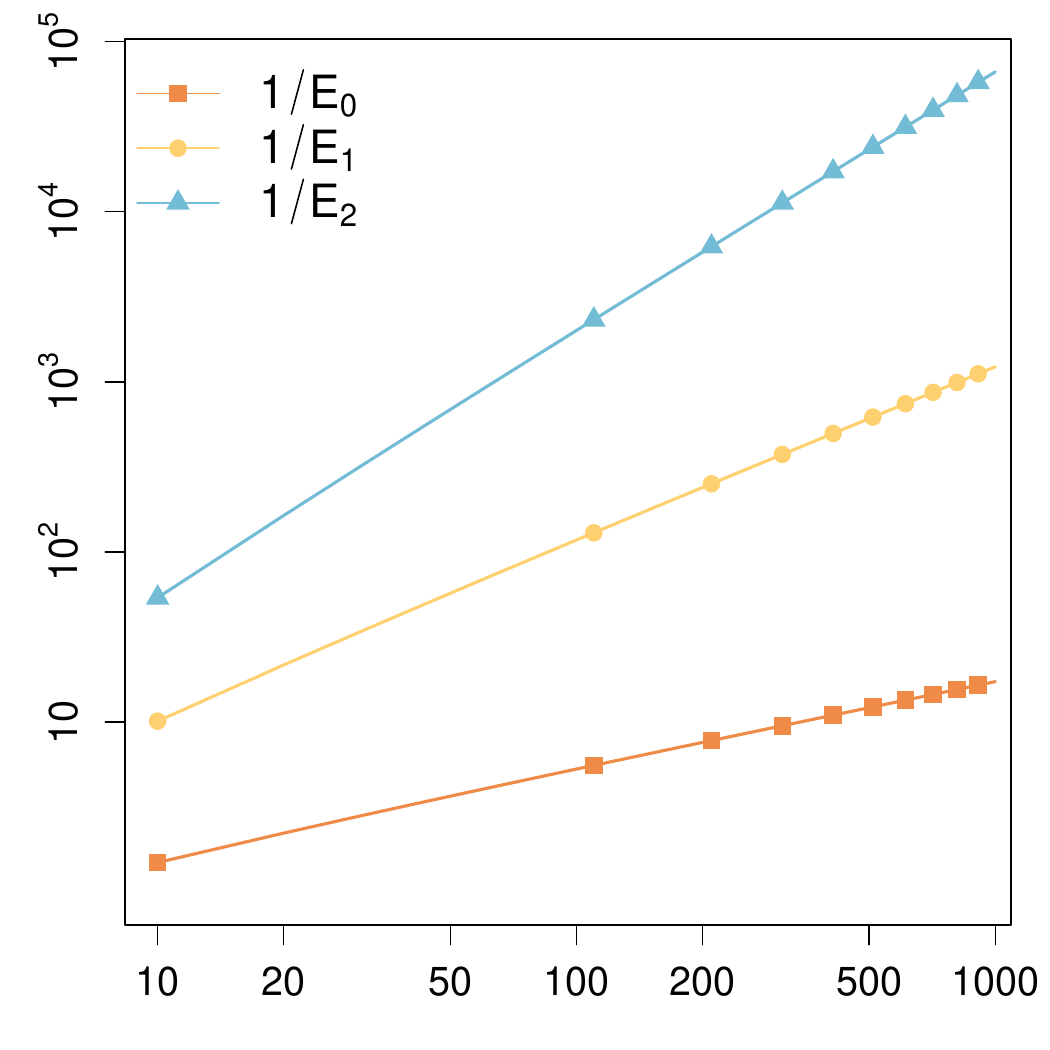}
\vspace{-0.5cm}
\caption{\footnotesize $\mat{\Omega}_0 = \left(\begin{smallmatrix} 4 & 1 \\ 1 & 4\end{smallmatrix}\right)$}
\end{subfigure}
\caption{The plots display $1/E_n$ as a function of $\mu$, for various choices of $\mat{\Omega}=\mat{\Omega}_0$, utilizing a logarithmic scale for both the horizontal and vertical axes. The plots illustrate the improvement achieved by incorporating correction terms from Theorem~\ref{thm:LLT.multivariate} to the base approximation.}\label{fig:figures/loglog.errors.plots}
\end{figure}

\begin{figure}[!htp]
\captionsetup[subfigure]{labelformat=empty}
\centering
\begin{subfigure}[b]{0.333\linewidth}
\centering
\includegraphics[width=\linewidth,height=\linewidth]{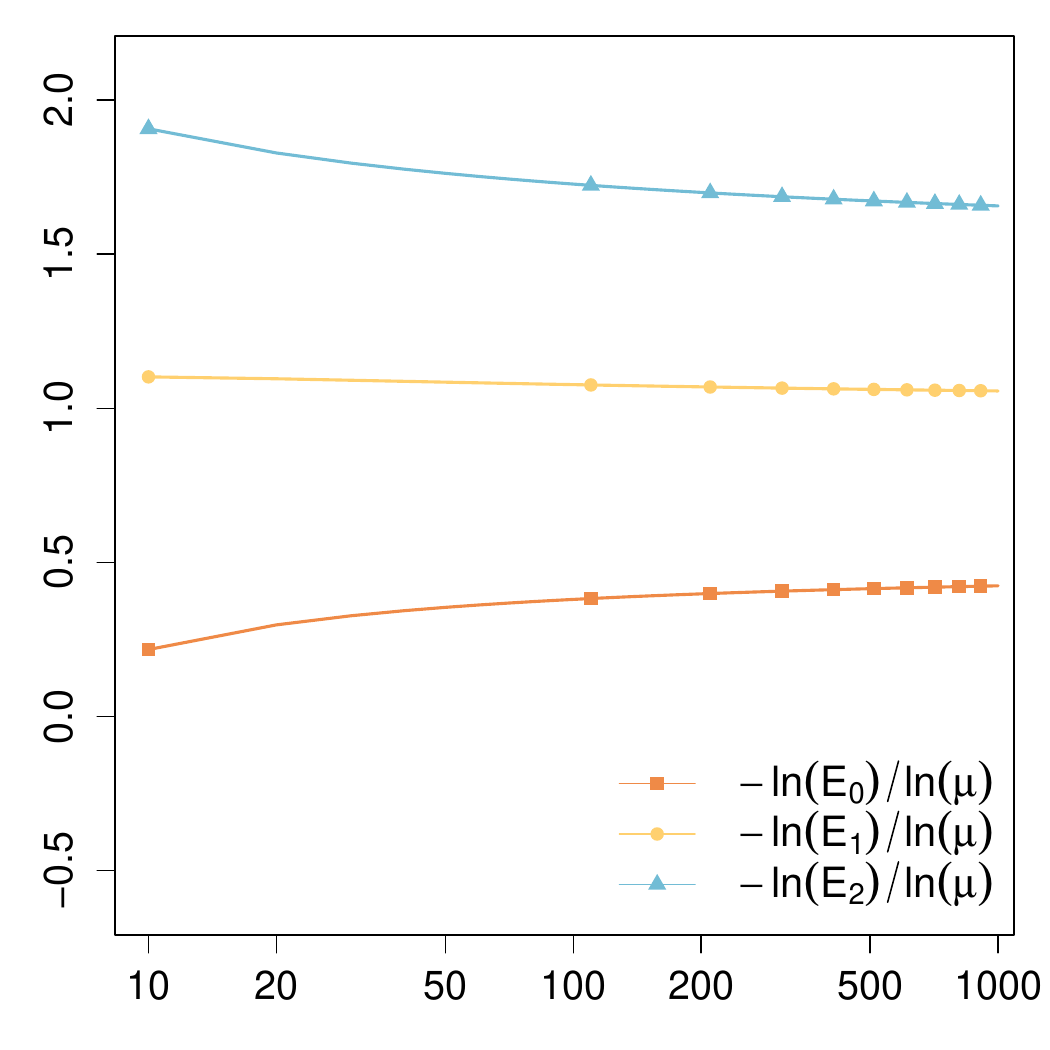}
\vspace{-0.5cm}
\caption{\footnotesize $\mat{\Omega}_0 = \left(\begin{smallmatrix} 2 & 1 \\ 1 & 2\end{smallmatrix}\right)$}
\end{subfigure}
\begin{subfigure}[b]{0.333\linewidth}
\centering
\includegraphics[width=\linewidth,height=\linewidth]{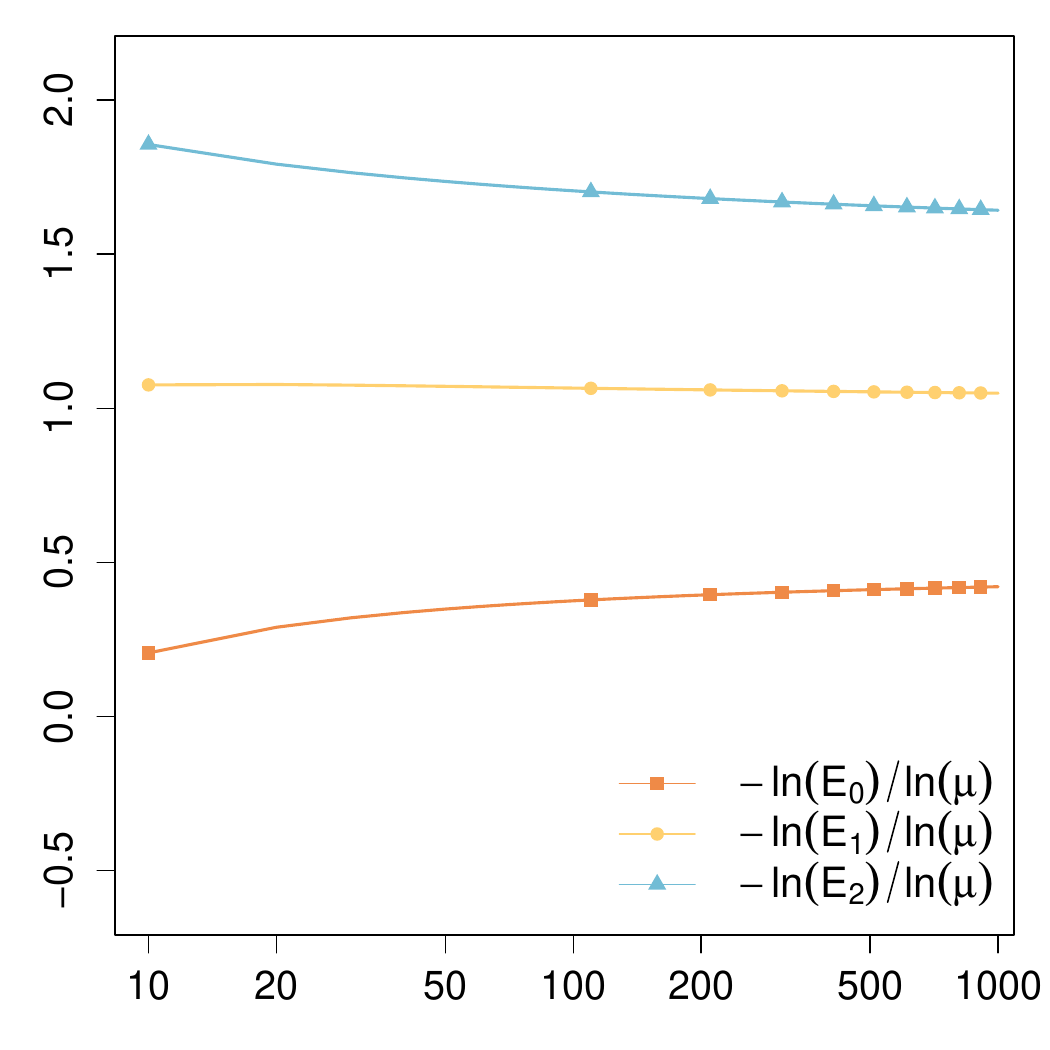}
\vspace{-0.5cm}
\caption{\footnotesize $\mat{\Omega}_0 = \left(\begin{smallmatrix} 2 & 1 \\ 1 & 3\end{smallmatrix}\right)$}
\end{subfigure}
\begin{subfigure}[b]{0.333\linewidth}
\centering
\includegraphics[width=\linewidth,height=\linewidth]{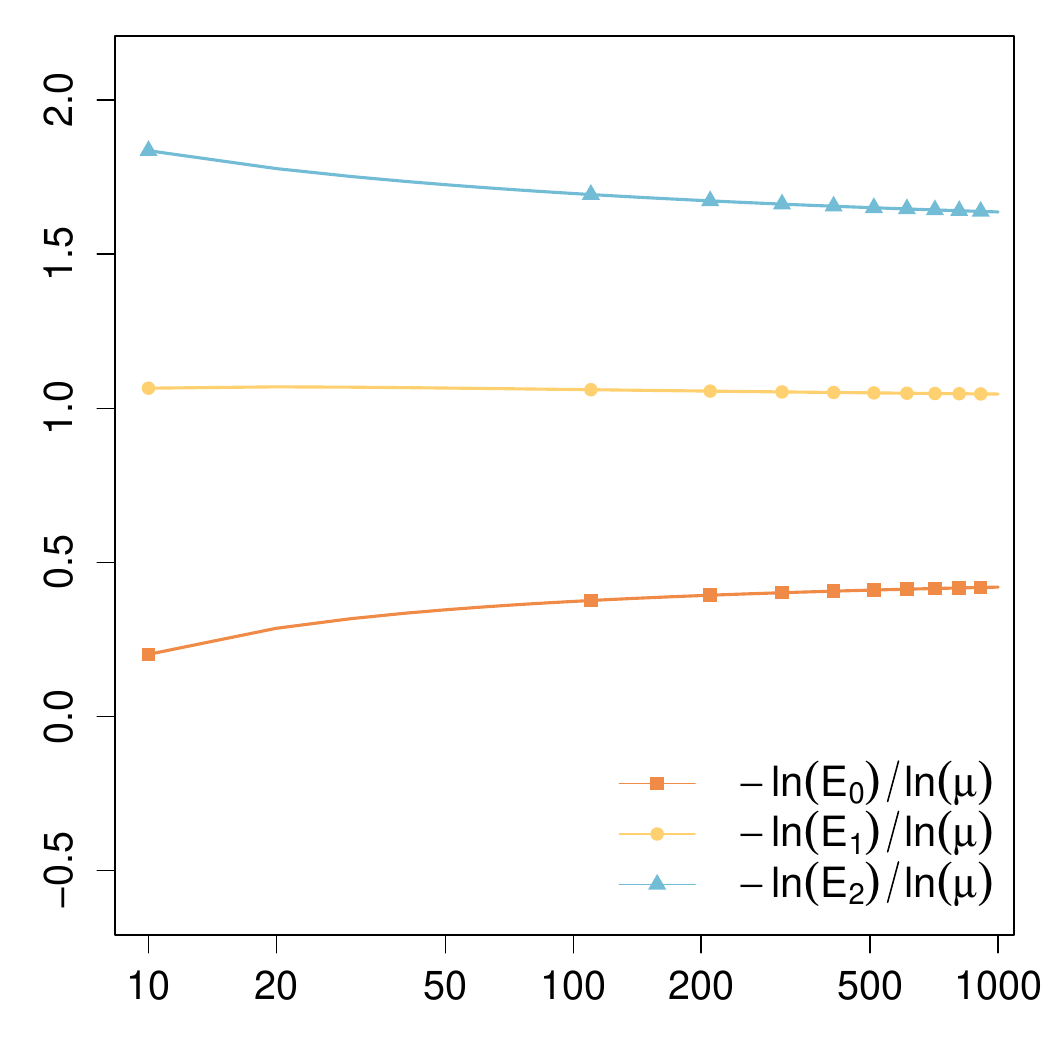}
\vspace{-0.5cm}
\caption{\footnotesize $\mat{\Omega}_0 = \left(\begin{smallmatrix} 2 & 1 \\ 1 & 4\end{smallmatrix}\right)$}
\end{subfigure}
\begin{subfigure}[b]{0.333\linewidth}
\centering
\includegraphics[width=\linewidth,height=\linewidth]{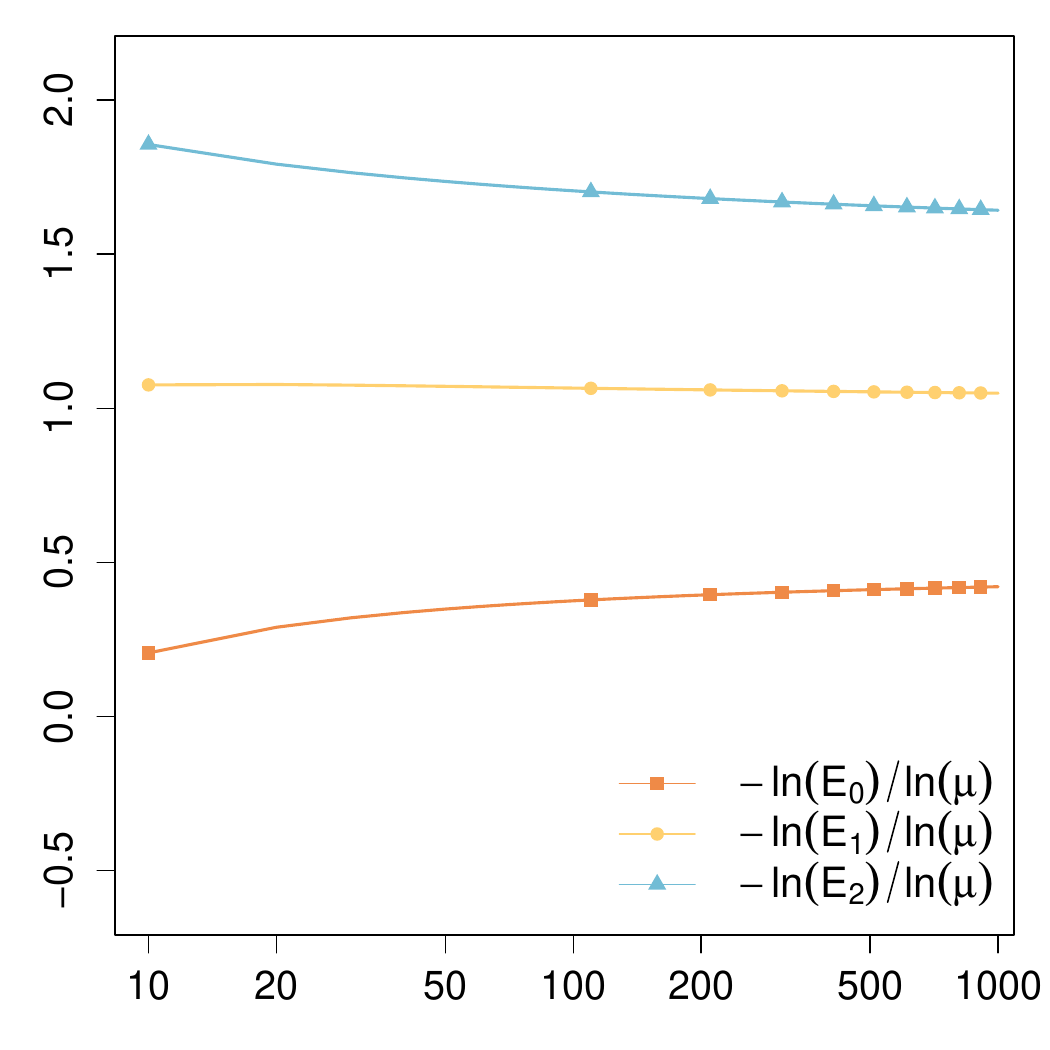}
\vspace{-0.5cm}
\caption{\footnotesize $\mat{\Omega}_0 = \left(\begin{smallmatrix} 3 & 1 \\ 1 & 2\end{smallmatrix}\right)$}
\end{subfigure}
\begin{subfigure}[b]{0.333\linewidth}
\centering
\includegraphics[width=\linewidth,height=\linewidth]{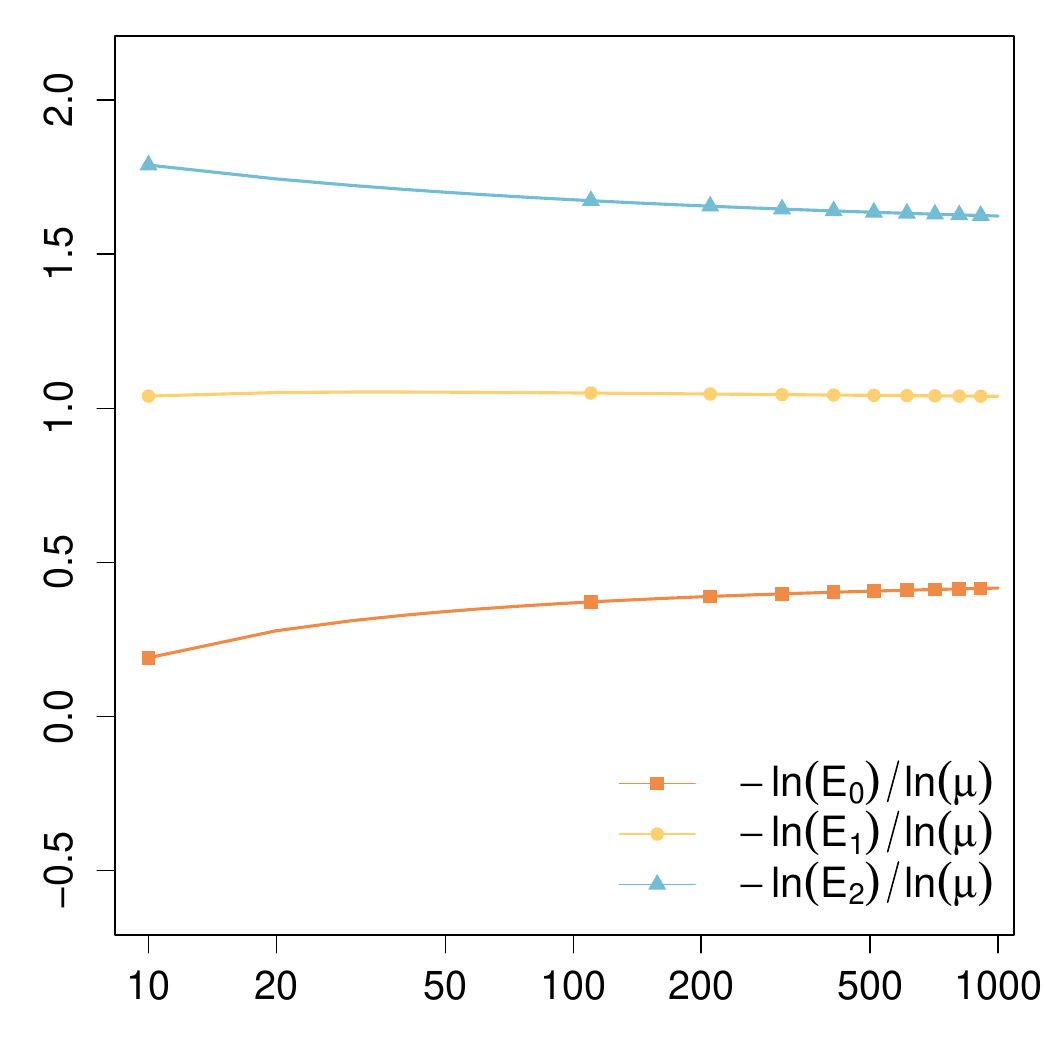}
\vspace{-0.5cm}
\caption{\footnotesize $\mat{\Omega}_0 = \left(\begin{smallmatrix} 3 & 1 \\ 1 & 3\end{smallmatrix}\right)$}
\end{subfigure}
\begin{subfigure}[b]{0.333\linewidth}
\centering
\includegraphics[width=\linewidth,height=\linewidth]{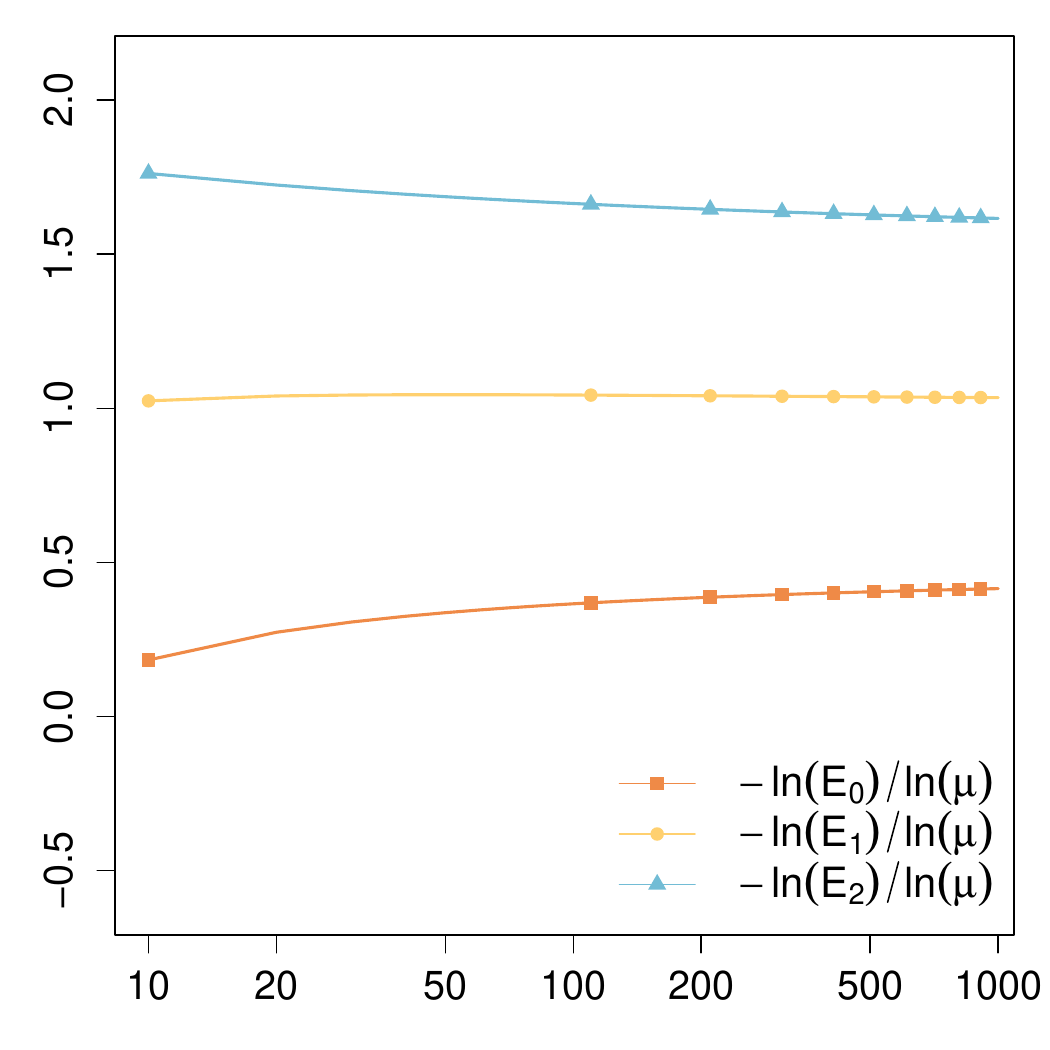}
\vspace{-0.5cm}
\caption{\footnotesize $\mat{\Omega}_0 = \left(\begin{smallmatrix} 3 & 1 \\ 1 & 4\end{smallmatrix}\right)$}
\end{subfigure}
\begin{subfigure}[b]{0.333\linewidth}
\centering
\includegraphics[width=\linewidth,height=\linewidth]{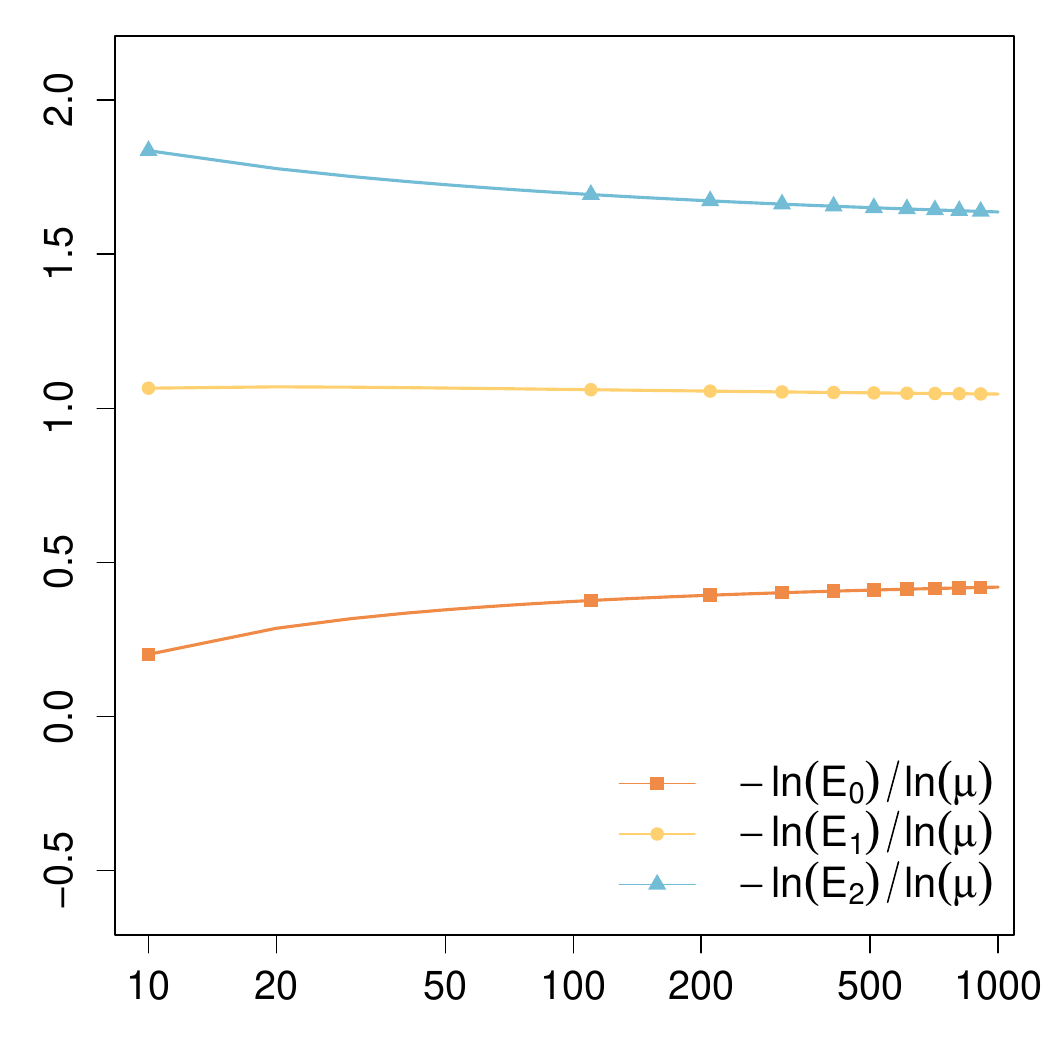}
\vspace{-0.5cm}
\caption{\footnotesize $\mat{\Omega}_0 = \left(\begin{smallmatrix} 4 & 1 \\ 1 & 2\end{smallmatrix}\right)$}
\end{subfigure}
\begin{subfigure}[b]{0.333\linewidth}
\centering
\includegraphics[width=\linewidth,height=\linewidth]{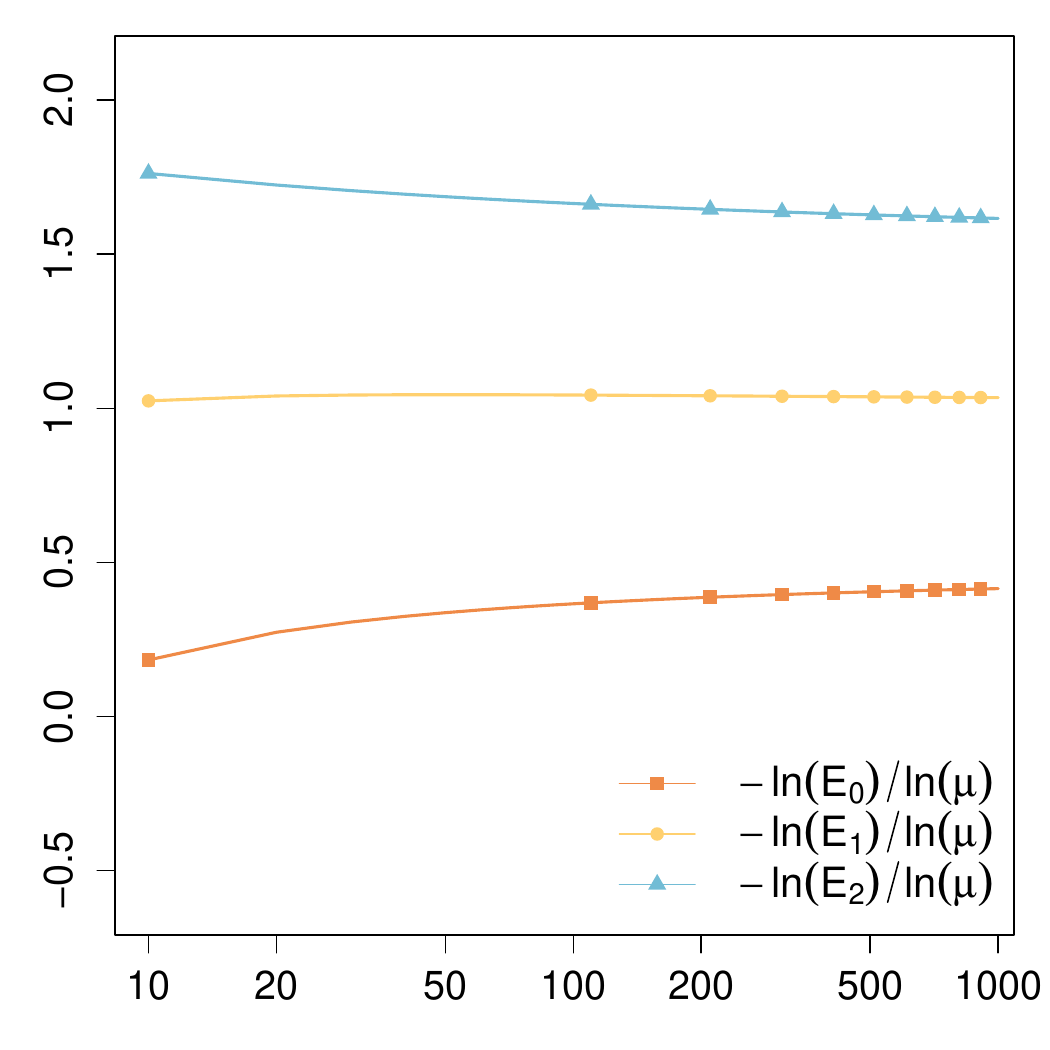}
\vspace{-0.5cm}
\caption{\footnotesize $\mat{\Omega}_0 = \left(\begin{smallmatrix} 4 & 1 \\ 1 & 3\end{smallmatrix}\right)$}
\end{subfigure}
\begin{subfigure}[b]{0.333\linewidth}
\centering
\includegraphics[width=\linewidth,height=\linewidth]{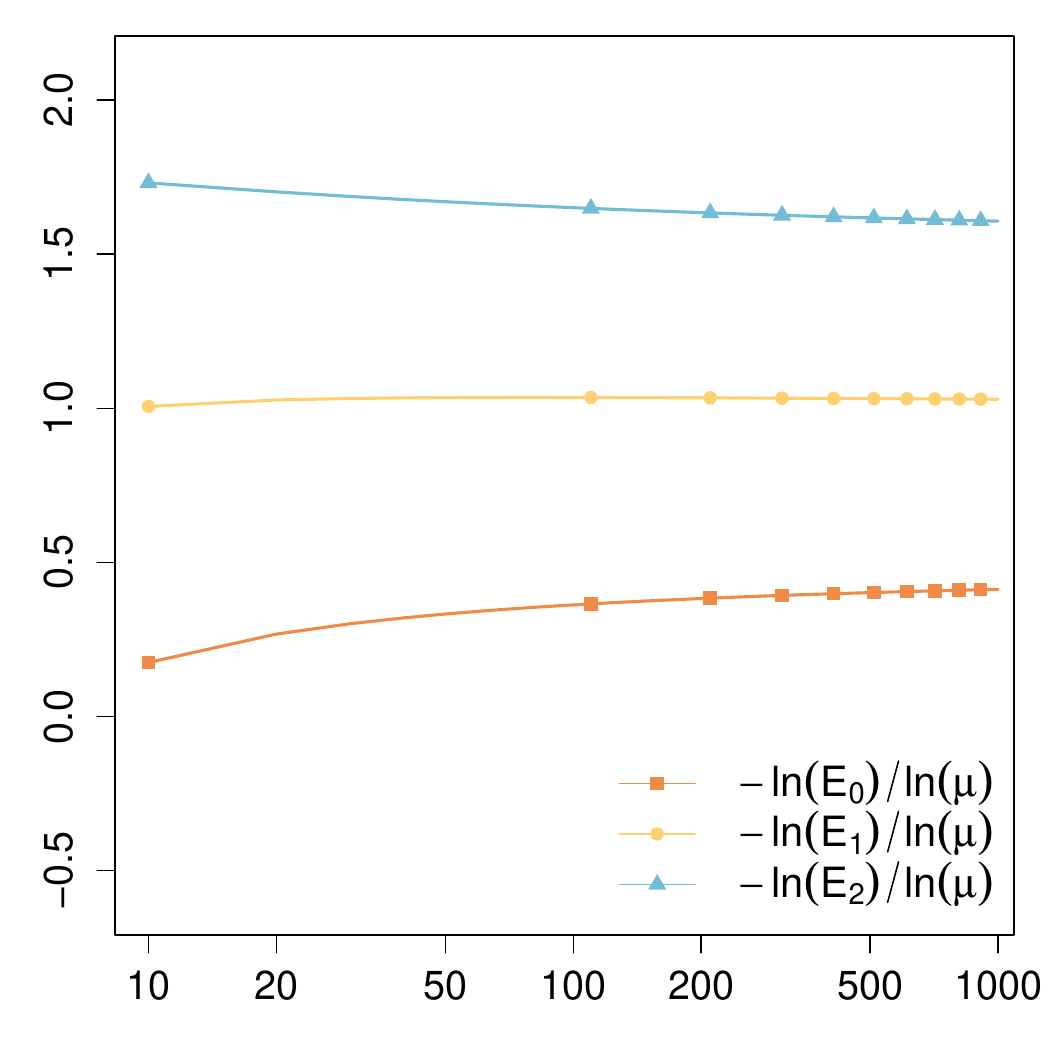}
\vspace{-0.5cm}
\caption{\footnotesize $\mat{\Omega}_0 = \left(\begin{smallmatrix} 4 & 1 \\ 1 & 4\end{smallmatrix}\right)$}
\end{subfigure}
\caption{The plots display $-\ln(E_n)/\ln(\mu)$ as a function of $\mu$, with the horizontal axis logarithmically scaled. The plots confirm the asymptotic orders of the liminfs in~\eqref{eq:liminf.exponent.bound} and provide compelling numerical evidence for the validity of Theorem~\ref{thm:LLT.multivariate}.}\label{fig:error.exponents.plots}
\end{figure}

By applying the local approximations of Theorem~\ref{thm:LLT.multivariate} in the bulk and by showing that the contributions outside the bulk are negligible using concentration inequalities, probability metric bounds between the measures induced by the densities $k_{\bb{\beta},\mu \bb{\xi}_0, \omega\mat{\Omega}_0}$ and $\phi_{\mu \bb{\xi}_0,\mu \omega\bb{\beta}^{\top} \bb{\xi}_0 \mat{\Omega}_0}$ are established. The following is a global result.

\newpage
\begin{theorem}[Probability metric bounds]\label{thm:total.variation.multivariate}
Let $\smash{\PP_{\mu, \omega, \bb{\beta}, \bb{\xi}_0, \mat{\Omega}_0}}$ be the measure on $\R^d$ induced by the MIG density $k_{\bb{\beta},\mu \bb{\xi}_0, \omega\mat{\Omega}_0}$ in~\eqref{eq:general.multivariate.density}. Let $\smash{\QQ_{\mu, \omega,\bb{\beta}, \bb{\xi}_0, \mat{\Omega}_0}}$ be the measure on $\R^d$ induced by multivariate Gaussian density $\smash{\phi_{\mu \bb{\xi}_0,\mu \omega \bb{\beta}^{\top} \bb{\xi}_0 \mat{\Omega}_0}}$ in~\eqref{eq:normal.distribution.multivariate}. Then, for any positive reals $\mu,\omega\in (0,\infty)$, one has
\begin{equation}\label{eq:bound.Hellinger.multivariate}
\mathfrak{H}(\PP_{\mu, \omega, \bb{\beta}, \bb{\xi}_0, \mat{\Omega}_0},\QQ_{\mu, \omega,\bb{\beta}, \bb{\xi}_0, \mat{\Omega}_0}) \leq C \sqrt{{\omega}/{\mu}},
\end{equation}
where $C = C(d,\bb{\beta}, \bb{\xi}_0, \mat{\Omega}_0)\in (0,\infty)$ is a positive constant that may depend on $d, \bb{\beta}$, $\bb{\xi}_0$ and $\mat{\Omega}_0$. The bound~\eqref{eq:bound.Hellinger.multivariate} is also valid if one replaces the Hellinger distance $\mathfrak{H}$ by any of the following probability metrics: discrepancy metric, Kolmogorov (or uniform) metric, L\'evy metric, Prokhorov metric, total variation.
\end{theorem}

\section{MIG kernel density estimator on half-spaces}\label{sec:MIG.kernel.density estimator}

Let $\bb{X}_1,\ldots,\bb{X}_n$ be a random sample from an unknown target density function~$f$ supported on the $d$-dimensional half-space $\mathcal{H}(\bb{\beta})$.

For a given positive definite bandwidth matrix $\mat{H}\in \mathcal{S}_{++}^d$, the MIG kernel density estimator for $f$ is defined, for any $\bb{\xi}\in \mathcal{H}(\bb{\beta})$, by
\begin{equation}\label{eq:MIG.estimator}
\hat{f}_{n,\mat{H}}(\bb{\xi}) = \frac{1}{n} \sum_{i=1}^n k_{\bb{\beta},\bb{\xi}, \mat{H}}(\bb{X}_i).
\end{equation}
If the observations were supported on the translated half-space $\bb{a} + \mathcal{H}(\bb{\beta})$ for some vector $\bb{a}\in \R^d$, then one would replace each vector $\bb{X}_i$ by the translated vector $\bb{X}_i - \bb{a}$ in~\eqref{eq:MIG.estimator}. The bandwidth matrix $\mat{H}$ here is the `scale' parameter $\mat{\Omega}$ of the MIG distribution but is denoted $\mat{H}$ to conform with the convention in kernel density estimation; see, e.g., \citet[Section~2.2]{MR3822372}.

The density estimator $\smash{\hat{f}_{n,\mat{H}}}$ is the first and only example in the literature of an asymmetric kernel density estimator tailored for general half-spaces in dimension $d\geq 2$. For the one-dimensional analog on the half-line $(0,\infty)$, see, e.g., \cite{MR2053071}, \cite{MR2179543}, \cite{MR2229885}, \cite{MR3131281}, \cite{MR4136585} and references therein. Two competing estimators are proposed in Section~\ref{subsec:simulation.competing.estimators}.

As mentioned in Section~\ref{sec:density.estimation.half.spaces}, asymmetric kernel estimators possess several desirable qualities. They achieve asymptotically negligible boundary bias, as shown by Proposition~\ref{prop:asymptotic.bias} below, and avoid spill-over near the boundary by locally adjusting the kernel's shape to the support's geometry. This inherent trait distinguishes them from the reflection method or boundary kernels, making them exceptionally user-friendly and among the simplest estimators in the class of methods that achieve asymptotic unbiasedness near the boundary. Moreover, they are nonnegative across the entire support of the target density, which sets them apart from many boundary bias-corrected estimators. For an overview of the literature on asymmetric kernel estimators, refer to \citet{MR3821525} and Section~2 of \citet{MR4319409}. For a review of associated kernels, which unify part of the theory for asymmetric kernels, see, e.g., \citet{MR3760293} and \cite{Kokonendji_Some_2021}.

The asymptotic properties of the new MIG kernel density estimator $\smash{\hat{f}_{n,\mat{H}}}$ are stated in Section~\ref{sec:asymptotics}. Its finite-sample performance, along with that of two competing estimators, is then examined by simulation in Section~\ref{sec:simulation.study}. To demonstrate practical utility, a bivariate version of $\smash{\hat{f}_{n,\mat{H}}}$ is applied in Section~\ref{sec:data.application} to smooth the posterior distribution of a generalized Pareto model fitted to large electromagnetic storms.

\subsection{Asymptotic results}\label{sec:asymptotics}

In this section, the local limit theorem for the MIG distribution (Theorem~\ref{thm:LLT.multivariate}) is applied to derive the asymptotic variance of the new density estimator $\smash{\hat{f}_{n,\mat{H}}}$ at each point $\bb{\xi}\in\mathcal{H}(\bb{\beta})$; see Proposition~\ref{prop:asymptotic.variance}. Combined with a careful analysis of the asymptotic bias (Proposition~\ref{prop:asymptotic.bias}), this result yields the leading-order expansions for both the mean squared error (MSE) (Corollary~\ref{cor:bias.var.implies.MSE.density}) and the MISE (Corollary~\ref{cor:MISE.optimal.density}) and enable the derivation of asymptotically optimal bandwidths. Moreover, the asymptotic normality of $\smash{\hat{f}_{n,\mat{H}}}$ at each point $\bb{\xi}$ is established in Theorem~\ref{thm:asymp.normality}. The proofs are deferred to Appendix~\ref{app:proofs.application.asymptotics}.

Throughout, expectations are taken with respect to the joint distribution of the mutually independent observations $\bb{X}_1,\ldots,\bb{X}_n$. Whether explicitly or not, the bandwidth matrix parameter $\mat{H} = \mat{H}(n)$ is assumed to be a function of the sample size whose spectral norm vanishes as $n \to \infty$, i.e., $\lim_{n\to \infty} \|\mat{H}\|_2 = 0$.

\begin{proposition}[Pointwise bias]\label{prop:asymptotic.bias}
Assume that $f$ is twice differentiable, and that its second order partial derivatives are uniformly continuous and bounded on the half-space $\mathcal{H}(\bb{\beta})$.
Then, for any real vector $\bb{\xi}\in \mathcal{H}(\bb{\beta})$, one has, as $n\to \infty$,
\[
\Bias\big\{\hat{f}_{n,\mat{H}}(\bb{\xi})\big\} = \frac{1}{2} \, \bb{\beta}^{\top} \bb{\xi} \sum_{i,j=1}^d \mat{H}_{ij} \frac{\partial^2}{\partial \xi_i \partial \xi_j} f(\bb{\xi}) + \oo(\bb{\beta}^{\top}\bb{\xi} \|\mat{H}\|_2).
\]
Alternatively, if $D^2 f$ denotes the Hessian matrix, one can rewrite the above as
\[
\Bias\big\{\hat{f}_{n,\mat{H}}(\bb{\xi})\big\} = \frac{1}{2} \, \bb{\beta}^{\top} \bb{\xi} \, \mathrm{tr}\{\mat{H} D^2 f(\bb{\xi})\} + \oo(\bb{\beta}^{\top}\bb{\xi} \|\mat{H}\|_2).
\]
\end{proposition}

\begin{remark}\label{rem:boundary.bias}
There are many ways to analyze the behavior of the bias as the estimation point $\bb{\xi}$ approaches the boundary. For example, if one considers a sequence $\bb{\xi} = \bb{\xi}(\mat{H})\in \mathcal{H}(\bb{\beta})$ such that $\bb{\beta}^{\top} \bb{\xi} \asymp \|\mat{H}\|_2$ as $n\to \infty$, then
\[
\Bias\{\hat{f}_{n,\mat{H}}(\bb{\xi})\} = \OO(\|\mat{H}\|_2^2).
\]
Alternatively, one may analyze the bias as $\bb{\xi}$ converges to a specific point on the boundary --- say, the origin $\bb{0}_d$ --- at the same rate in each coordinate as the eigenvalues of the bandwidth matrix converge to $0$. Specifically, assuming the spectral decomposition $\mat{H} = \mat{V}^{\top} \mat{\Lambda} \mat{V}$, where $\mat{V} = (v_{ij})_{1\leq i,j \leq d}$ and $\mat{\Lambda} = \diag(\lambda_1,\ldots,\lambda_d)$, one can let $\bb{\xi} = (c_1 \lambda_1,\ldots,c_d \lambda_d)^{\top}\!\in \mathcal{H}(\bb{\beta})$ for some appropriate constants $c_1,\ldots,c_d\in \R\backslash \{0\}$. Then one finds that, as $n\to \infty$,
\[
\Bias\big\{\hat{f}_{n,\mat{H}}(\bb{\xi})\big\}
= \frac{1}{2} \left\{ \sum_{\ell=1}^d \beta_{\ell} c_{\ell} \lambda_{\ell} \right\} \sum_{k=1}^d \lambda_k \left\{ \sum_{i,j=1}^d \frac{v_{ik} v_{jk}}{c_i c_j} \frac{\partial^2 f(\bb{\xi})}{\partial \lambda_i \partial \lambda_j} \right\}
= \OO(\|\mat{H}\|_2^2).
\]
More generally, if $\bb{\xi}$ remains fixed and nonzero in at least one coordinate (i.e., if there exists a nonempty subset of indices $\mathcal{J}\subseteq \{1,\ldots,d\}$ such that $\xi_i = c_i \lambda_i$ for all $i\in \mathcal{J}$ while $\xi_i\neq 0$ remains fixed for all $i\in \{1,\ldots,d\}\setminus \mathcal{J}$), then the bias behaves asymptotically as $\OO(\|\mat{H}\|_2)$, which matches the behavior observed in the interior of the support. For classical kernel density estimators, however, the bias near the boundary is typically $\asymp 1$, i.e., it does not vanish as $n \to \infty$. This well-known boundary bias issue, mentioned in Section~\ref{sec:density.estimation.half.spaces}, arises because fixed symmetric kernels allocate mass outside the support when the estimation point lies close to the support's edges \citep[Section~2]{MarronRuppert1994}. A similar phenomenon occurs in nonparametric kernel regression; see, e.g., \citet{MR1193323}. In contrast, the bias of the MIG kernel density estimator decreases near the boundary at least at the same rate as in the interior, ensuring consistent performance across the entire support.
\end{remark}

Although $\smash{\hat{f}_{n,\mat{H}}}$ exhibits superior boundary bias behavior, this benefit comes at the cost of increasing the variance near the boundary, as the next proposition reveals. This phenomenon is a direct consequence of the adaptive mechanism of the asymmetric MIG kernel. To avoid the boundary bias problem that affects symmetric kernels, the MIG kernel adapts its shape based on the estimation point $\xi$. As $\xi$ approaches the boundary of the half-space (i.e., as $\bb{\beta}^{\top} \bb{\xi} \to 0$), the kernel becomes increasingly skewed and concentrated to confine its mass within the support. This necessary adaptation ensures that the bias remains low. However, by becoming more ``peaked,'' the resulting density estimate at that point relies more heavily on a smaller effective number of local data points. This heightened sensitivity to the specific locations of a few nearby observations leads to greater sampling variability, and thus, a higher variance. This behavior is typical for asymmetric kernel estimators; see, e.g., \cite{Chen1999} or \cite{MR4319409}. It is also a recurring feature in the study of Bernstein estimators; see, e.g., \cite{MR2925964} or \cite{doi:10.1515/stat-2022-0111}.

\medskip

In contrast, a classical kernel uses a fixed, symmetric shape, which leads to some of the kernel's mass to spill over and higher boundary bias. The variance of this type of estimator, however, is not structurally affected by the boundary. Given that the variance depends on the kernel's fixed shape and the target density $f$, rather than the point's proximity to the boundary, it remains stable across the domain. This illustrates the alternative tradeoff: putting up with high bias in exchange for stable variance.

\begin{proposition}[Pointwise variance]\label{prop:asymptotic.variance}
Assume that $f$ is Lipschitz continuous and bounded on the half-space $\mathcal{H}(\bb{\beta})$. Then, for any real vector $\bb{\xi}\in \mathcal{H}(\bb{\beta})$, one has, as $n\to \infty$,
\[
\Var\big\{\hat{f}_{n,\mat{H}}(\bb{\xi})\big\} = n^{-1} |\mat{H}|^{-1/2} \frac{f(\bb{\xi})}{(4\pi \bb{\beta}^{\top} \bb{\xi})^{d/2}} + \oo_{\bb{\beta},\bb{\xi}}(n^{-1} |\mat{H}|^{-1/2}).
\]
\end{proposition}

\newpage
\begin{remark}\label{rem:boundary.variance}
As in Remark~\ref{rem:boundary.bias}, if a sequence $\bb{\xi} = \bb{\xi}(\mat{H})\in \mathcal{H}(\bb{\beta})$ is chosen such that $\smash{\bb{\beta}^{\top} \bb{\xi} \asymp \|\mat{H}\|_2}$ as $n\to \infty$, then the variance satisfies
\[
\Var\big\{\hat{f}_{n,\mat{H}}(\bb{\xi})\big\} = \OO(n^{-1} |\mat{H}|^{-1/2} \|\mat{H}\|_2^{-d/2}).
\]
Furthermore, consider the scenario in which $\bb{\xi}$ approaches $\bb{0}_d$ in the coordinates of a nonempty subset $\mathcal{J}\subseteq \{1,\ldots,d\}$ (i.e., $\xi_i = c_i \lambda_i$ for all $i\in \mathcal{J}$ while $\xi_i\neq 0$ remains fixed for all $i\in \{1,\ldots,d\}\setminus \mathcal{J}$), then
\[
\Var\big\{\hat{f}_{n,\mat{H}}(\bb{\xi})\big\} = \OO\big(n^{-1} |\mat{H}|^{-1/2} \max_{i\in \mathcal{J}} \lambda_i^{-d/2}\big).
\]
In particular, if all eigenvalues of $\mat{H}$ remain equal as $n\to \infty$, then the variance of the estimator near the origin is $\OO(n^{-1} |\mat{H}|^{-1})$.
\end{remark}

Combining Propositions~\ref{prop:asymptotic.bias}--\ref{prop:asymptotic.variance} yields the asymptotic expansion of the MSE, as outlined below.

\begin{corollary}[Mean squared error]\label{cor:bias.var.implies.MSE.density}
Assume that $f$ is twice differentiable, and that its second order partial derivatives are uniformly continuous and bounded on the half-space $\mathcal{H}(\bb{\beta})$. Then, for any real vector $\bb{\xi}\in \mathcal{H}(\bb{\beta})$, one has, as $n\to \infty$,
\[
\begin{aligned}
\mathrm{MSE}\big\{\hat{f}_{n,\mat{H}}(\bb{\xi})\big\}
&= \EE\left\{|\hat{f}_{n,\mat{H}}(\bb{\xi}) - f(\bb{\xi})|^2\right\} \\
&= n^{-1} |\mat{H}|^{-1/2} \frac{f(\bb{\xi})}{(4\pi \bb{\beta}^{\top} \bb{\xi})^{d/2}} + \frac{1}{4} (\bb{\beta}^{\top} \bb{\xi})^2 \mathrm{tr}^2\{\mat{H} D^2 f(\bb{\xi})\} \\
&\qquad+ \oo_{\bb{\beta},\bb{\xi}}(n^{-1} |\mat{H}|^{-1/2}) + \oo\{(\bb{\beta}^{\top}\bb{\xi})^2 \|\mat{H}\|_2^2\}.
\end{aligned}
\]
The asymptotically optimal choice of $\mat{H}$ with respect to $\mathrm{MSE}$ is
\[
\mat{H}_{\mathrm{opt}}(\bb{\xi}) = \argmin_{\mat{H}\in \mathcal{S}_{++}^d} \left[n^{-1} |\mat{H}|^{-1/2} \frac{f(\bb{\xi})}{(4\pi \bb{\beta}^{\top} \bb{\xi})^{d/2}} + \frac{1}{4} (\bb{\beta}^{\top} \bb{\xi})^2 \mathrm{tr}^2\{\mat{H} D^2 f(\bb{\xi})\}\right].
\]
If the bandwidth $\mat{H}$ is restricted to be a scalar matrix, viz.~$\mat{H} = h^2 \mat{I}_d$ for some positive real $h\in (0,\infty)$, then the asymptotically optimal choice of $\mat{H}$ with respect to $\mathrm{MSE}$ is
\[
\mat{H}_{\mathrm{opt}}(\bb{\xi}) = \{h_{\mathrm{opt}}(\bb{\xi})\}^2 \mat{I}_d,
\]
where
\[
h_{\mathrm{opt}}(\bb{\xi}) = \left[\frac{d}{n} \times \frac{f(\bb{\xi}) / (4\pi \bb{\beta}^{\top} \bb{\xi})^{d/2}}{(\bb{\beta}^{\top} \bb{\xi})^2 \{\Delta f(\bb{\xi})\}^2}\right]^{1/(d+4)},
\]
and where $\Delta$ denotes the Laplacian operator. Consequently, one has, as $n\to \infty$,
\[
\begin{aligned}
\mathrm{MSE}\left\{\hat{f}_{n,\mat{H}_{\mathrm{opt}}(\bb{\xi})}(\bb{\xi})\right\}
&= \frac{4 + d}{4d} \left\{\frac{d}{n} \times \frac{f(\bb{\xi})}{(4\pi \bb{\beta}^{\top} \bb{\xi})^{d/2}}\right\}^{4/(d+4)} \\
&\qquad\times \left[(\bb{\beta}^{\top} \bb{\xi})^2 \{\Delta f(\bb{\xi})\}^2\right]^{d/(d+4)} + \oo(n^{-4/(d+4)}).
\end{aligned}
\]
\end{corollary}

Next, by integrating the MSE, one obtains the asymptotics of the MISE.

\begin{corollary}[Mean integrated squared error]\label{cor:MISE.optimal.density}
Assume that $f$ is twice differentiable and that its second order partial derivatives are uniformly continuous and bounded on the half-space $\mathcal{H}(\bb{\beta})$. Furthermore, assume that the following integrals are finite:
\[
\int_{\mathcal{H}(\bb{\beta})} \frac{f(\bb{\xi})}{(4\pi \bb{\beta}^{\top} \bb{\xi})^{d/2}} \, \rd \bb{\xi} < \infty, \quad
\int_{\mathcal{H}(\bb{\beta})} \frac{1}{4} (\bb{\beta}^{\top} \bb{\xi})^2 \mathrm{tr}^2\{\mat{H} D^2 f(\bb{\xi})\} \rd \bb{\xi} < \infty.
\]
Then, as $n\to \infty$, one has
\[
\begin{aligned}
\mathrm{MISE}\big(\hat{f}_{n,\mat{H}}\big)
&= n^{-1} |\mat{H}|^{-1/2} \int_{\mathcal{H}(\bb{\beta})} \frac{f(\bb{\xi})}{(4\pi \bb{\beta}^{\top} \bb{\xi})^{d/2}} \, \rd \bb{\xi} + \oo_{\bb{\beta}}(n^{-1} |\mat{H}|^{-1/2}) \\
&\qquad+ \int_{\mathcal{H}(\bb{\beta})} \frac{1}{4} (\bb{\beta}^{\top} \bb{\xi})^2 \mathrm{tr}^2\{\mat{H} D^2 f(\bb{\xi})\} \rd \bb{\xi} + \oo_{\bb{\beta}}(\|\mat{H}\|_2^2).
\end{aligned}
\]
The asymptotically optimal choice of $\mat{H}$ with respect to $\mathrm{MISE}$ is
\begin{equation}\label{eq:AMISE.full}
\begin{aligned}
\mat{H}_{\mathrm{opt}}
&= \argmin_{\mat{H}\in \mathcal{S}_{++}^d} \left[n^{-1} |\mat{H}|^{-1/2} \int_{\mathcal{H}(\bb{\beta})} \frac{f(\bb{\xi})}{(4\pi \bb{\beta}^{\top} \bb{\xi})^{d/2}} \, \rd \bb{\xi} \right. \\
&\hspace{30mm}\left. + \int_{\mathcal{H}(\bb{\beta})} \frac{1}{4} (\bb{\beta}^{\top} \bb{\xi})^2 \mathrm{tr}^2\{\mat{H} D^2 f(\bb{\xi})\} \rd \bb{\xi}\right].
\end{aligned}
\end{equation}
If the bandwidth $\mat{H}$ is restricted to be a scalar matrix, viz.~$\mat{H} = h^2 \mat{I}_d$ for some positive real $h\in (0,\infty)$, then the asymptotically optimal choice of $\mat{H}$ with respect to $\mathrm{MISE}$ is
\[
\mat{H}_{\mathrm{opt}} = (h_{\mathrm{opt}})^2 \mat{I}_d,
\]
where
\begin{equation}\label{eq:AMISE_isotropic}
h_{\mathrm{opt}} = \left[\frac{d}{n} \times \frac{\int_{\mathcal{H}(\bb{\beta})} f(\bb{\xi}) / (4\pi \bb{\beta}^{\top} \bb{\xi})^{d/2} \rd \bb{\xi}}{\int_{\mathcal{H}(\bb{\beta})} (\bb{\beta}^{\top} \bb{\xi})^2 \{\Delta f(\bb{\xi})\}^2 \rd \bb{\xi}}\right]^{1/(d+4)}.
\end{equation}
Consequently, one has, as $n\to \infty$,
\[
\begin{aligned}
\mathrm{MISE}\big(\hat{f}_{n,\mat{H}_{\mathrm{opt}}}\big)
&= \frac{4 + d}{4d} \left\{\frac{d}{n} \int_{\mathcal{H}(\bb{\beta})} \frac{f(\bb{\xi})}{(4\pi \bb{\beta}^{\top} \bb{\xi})^{d/2}} \rd \bb{\xi}\right\}^{4/(d+4)} \\
&\qquad\times \left[\int_{\mathcal{H}(\bb{\beta})} (\bb{\beta}^{\top} \bb{\xi})^2 \{\Delta f(\bb{\xi})\}^2 \rd \bb{\xi}\right]^{d/(d+4)} + \oo(n^{-4/(d+4)}).
\end{aligned}
\]
\end{corollary}

\begin{remark}\label{rem:minimax}
The asymptotic rate of the optimal MISE, $n^{-4/(d+4)}$, established in Corollary~\ref{cor:MISE.optimal.density}, is known to be minimax for second-order kernel density estimators under the assumption that the second-order partial derivatives of the target density are uniformly continuous and bounded; see, e.g., \citet{MR300360}, \cite{MR594650} or \cite{Tsybakov2009}. For a detailed treatment of the minimaxity of the related Dirichlet kernel density estimator on the simplex, considering a wide range of $L^p$ loss functions and degrees of smoothness of the target density, refer to \citet{MR4544604}. Extending these authors' analysis to the present setting is an interesting direction for future research.
\end{remark}

\begin{remark}\label{rem:plug.in.bandwidth}
In Corollaries~\ref{cor:bias.var.implies.MSE.density}--\ref{cor:MISE.optimal.density}, the asymptotically optimal bandwidth matrices depend on the unknown density $f$ and are therefore referred to as oracle bandwidths. In practice, a common strategy is to evaluate these expressions by replacing $f$ with a pilot density; for example the MIG density $\smash{k_{\bb{\beta},\widehat{\bb{\xi}}_n,\widehat{\mat{\Omega}}_n}}$, where the parameter $\bb{\beta}\in \R^d$ is assumed known and $(\widehat{\bb{\xi}}_n,\widehat{\mat{\Omega}}_n)$ is an estimator for the pair $(\bb{\xi},\mat{\Omega})$ as given in Proposition~\ref{prop:MLE}. With this substitution, the integrals in \eqref{eq:AMISE.full} and \eqref{eq:AMISE_isotropic} over the half-space $\mathcal{H}(\bb{\beta})$ can be approximated via Monte Carlo simulation by drawing samples from the corresponding MIG distribution using the algorithm developed in Section~\ref{sec:rng.MIG}. Alternatively, one may simulate data from a multivariate truncated Gaussian distribution over the set $\{\bb{x}\in\R^d : \bb{\beta}^{\top}\bb{x} > \delta\}$ for some buffer $\delta \in [0, \infty)$. These approaches are variants of the plug-in selection method; for additional options, see \citet[Chapter~3]{MR3822372}.
\end{remark}

\begin{remark}\label{rem:scaling.isotropic}
If margins have different scale parameters or exhibit correlations, using an isotropic model with bandwidth matrix $\mat{H}= h^2 \mat{I}_d$ will be suboptimal. However, the data can be rotated, $\bb{X}_i^{\star}=\mat{L}^{-1}\bb{X}_i$, as proposed by \cite{Duong.Hazelton:2003}, to work out the optimal bandwidth for $\bb{X}_1^{\star},\ldots,\bb{X}_n^{\star}$, where $\mat{L}$ is the lower triangular Cholesky root of the sample covariance matrix, $\mat{S}_n = \mat{L} \mat{L}^{\top}$. If $\bb{X} \sim \mathrm{MIG}(\bb{\beta}, \bb{\xi}, \mat{\Omega})$, it is easy to see that $\mat{L}^{-1}\bb{X} \sim \mathrm{MIG}(\mat{L}^{\!\top}\!\bb{\beta}, \mat{L}^{-1}\bb{\xi}, \mat{L}^{-1}\mat{\Omega}\mat{L}^{-\!\top})$ \citep[Property~2]{MR2019130}, as the MIG family is closed under affine transformations. Thus, one can first rotate the data, estimate an isotropic bandwidth in the rotated space with the sample $\bb{X}_1^{\star},\ldots,\bb{X}_n^{\star}$, and then back-transform the bandwidth matrix to the original data orientation. Alternatively, scaling can be considered instead of rotation; for this, a similar procedure is followed by replacing $\mat{S}_n$ with a diagonal matrix with the component-specific variances.
\end{remark}

Asymptotic normality is established by verifying the Lindeberg condition for double arrays. The result, which is proved in Appendix~\ref{sec:B3}, is as follows.

\begin{theorem}[Asymptotic normality]\label{thm:asymp.normality}
Let the real vector $\bb{\xi}\in \mathcal{H}(\bb{\beta})$ be such that $f(\bb{\xi})\in (0,\infty)$. On the one hand, assume that $f$ is Lipschitz continuous and bounded on the half-space $\mathcal{H}(\bb{\beta})$. If $n^{1/2} |\mat{H}|^{1/4}\to \infty$ as $n\to \infty$, then
\[
n^{1/2} |\mat{H}|^{1/4} \left[\hat{f}_{n,\mat{H}}(\bb{\xi}) - \EE\big\{\hat{f}_{n,\mat{H}}(\bb{\xi})\big\}\right] \rightsquigarrow \mathcal{N}\left[ 0,\frac{f(\bb{\xi})}{(4\pi \bb{\beta}^{\top} \bb{\xi})^{d/2}}\right].
\]
On the other hand, assume that $f$ is twice differentiable and that its second order partial derivatives are uniformly continuous and bounded on the half-space $\mathcal{H}(\bb{\beta})$. If $n^{1/2} |\mat{H}|^{1/4}\to \infty$ and $n^{1/2} |\mat{H}|^{1/4} \|\mat{H}\|_2\to 0$ as $n\to \infty$, then the last equation together with Proposition~\ref{prop:asymptotic.bias} imply
\[
n^{1/2} |\mat{H}|^{1/4} \big\{\hat{f}_{n,\mat{H}}(\bb{\xi}) - f(\bb{\xi})\big\} \rightsquigarrow \mathcal{N}\left[ 0,\frac{f(\bb{\xi})}{(4\pi \bb{\beta}^{\top} \bb{\xi})^{d/2}}\right].
\]
\end{theorem}

\subsection{Simulation study}\label{sec:simulation.study}

In this section, a simulation study is presented to assess the finite-sample performance of the MIG kernel density estimator. Section~\ref{subsec:simulation.competing.estimators} introduces the competing estimators, Section~\ref{subsec:simulation.target.distributions} lists the target distributions, Section~\ref{subsec:simulation.bandwidth.selection} details the bandwidth selection criteria, and Section~\ref{subsec:simulation.performance.measures} defines the performance metrics. The simulation results are summarized in Section~\ref{subsec:simulation.results}, with additional practical considerations discussed in Section~\ref{subsec:simulation.discussion}.

\subsubsection{Competing estimators} \label{subsec:simulation.competing.estimators}

To evaluate the performance of the MIG kernel density estimator, it is useful to include competing estimators for benchmarking purposes.

One way to overcome the half-space support constraint is to consider a mapping onto $\R^d$, coupled with either a traditional multivariate Gaussian kernel density estimator or a product-kernel density estimator. Specifically, consider the linear map $\bb{x} \mapsto \mat{Q}\bb{x}$ from $\mathcal{H}(\bb{\beta})$ to $\R_{+} \times \R^{d-1}$ as described in Proposition~\ref{prop:MIG.stochastic.representation}, and define $T(\bb{x}) = (\ln(x_1),\, x_2, \ldots, x_d)^{\top}$, which maps $\R_{+} \times \R^{d-1}$ onto $\R^d$. Then, define the transformation $g(\bb{x}) = T(\mat{Q}\bb{x})$, which maps $\mathcal{H}(\bb{\beta})$ onto $\R^d$, with associated Jacobian $J_g(\bb{x}) = \|\bb{\beta}\|_2 / (\bb{\beta}^{\top}\bb{x})$. If $\phi_d(\,\cdot \, ; \bb{\mu}, \mat{H})$ denotes a multivariate Gaussian density with mean vector $\bb{\mu}$ and covariance matrix $\mat{H}$, the first competitor in the simulation study is the following {\it transformation} kernel density estimator on the half-space $\mathcal{H}(\bb{\beta})$:
\begin{equation}\label{eq:trans.Gaussian.KDE}
\widehat{f}_{n, \mat{H}}^{(\mathrm{trans})}(\bb{\xi}) = \frac{1}{n} \sum_{i=1}^n J_g(\bb{X}_i) \, \phi_d\big\{g(\bb{X}_i) ; g(\bb{\xi}), \mat{H}\big\}.
\end{equation}

The second competitor in the simulation study is a {\it truncation} kernel density estimator on the half-space $\mathcal{H}(\bb{\beta})$ that redistributes the mass of the Gaussian kernel $\phi_d$ that would otherwise spill over the support:
\begin{equation}\label{eq:trunc.Gaussian.KDE}
\widehat{f}_{n, \mat{H}}^{(\mathrm{trunc})}(\bb{\xi}) = \frac{1}{n} \sum_{i=1}^n \frac{\phi_d(\bb{X}_i ; \bb{\xi}, \mat{H})}{\Phi(0 ; -\bb{\beta}^{\top}\bb{\xi}, \bb{\beta}^{\top}\mat{H}\bb{\beta})},
\end{equation}
where $\Phi(\,\cdot \, ; \mu, \sigma^2)$ denotes the univariate cdf of the $\mathcal{N}(\mu,\sigma^2)$ distribution. Note that, because $\bb{\xi} \in \mathcal{H}(\bb{\beta})$, the mode of every truncated Gaussian component in the sum necessarily lies within the half-space.

\subsubsection{List of target distributions}\label{subsec:simulation.target.distributions}

Let $\mat{R}_d(\rho)= (1-\rho) \mat{I}_d +\rho\bb{1}_d\bb{1}_d^{\top}$ be an equicorrelation matrix with correlation coefficient $\rho \in (-1/(d-1), 1)$, and let $\bb{\iota}_d = (1,\ldots,d)^{\top}$. In the simulation study, the observations are drawn from four target distributions supported on $\mathcal{H}(\bb{\beta})$, with $\bb{\beta} = \bb{1}_d$, in dimension $d \in \{2,3\}$, namely
\begin{enumerate}[\quad${F}_1$:]
\item a {S}tudent-$t_5$ distribution truncated over $\mathcal{H}(\bb{\beta})$, with location vector $\bb{\beta} + 0.1\bb{1}_d$ and scale matrix $\diag(\sqrt{\bb{\iota}_d}) \, \mat{R}_d(0.5) \diag(\sqrt{\bb{\iota}_d})$;
\item a two-component mixture distribution with weights $(0.6,0.4)$, where
\begin{itemize}\setlength\itemsep{0em}
\item [a)] the first component is a {S}tudent-$t_5$ distribution truncated over $\mathcal{H}(\bb{\beta})$ with location vector $\bb{1}_d$ and scale matrix $\mat{R}_d(0.5)$;
\item [b)] the second component is a Gaussian distribution truncated over $\mathcal{H}(\bb{\beta})$, with scale matrix $\smash{\diag(\sqrt{\bb{\iota}_d}) \, \mat{R}_d(-0.2) \diag(\sqrt{\bb{\iota}_d})}$ and location vector $\bb{\beta} + 10\bb{1}_d - \bb{\iota}_d$;
\end{itemize}
\item a truncated multivariate skewed Gaussian distribution
\citep{Azzalini.Capitanio:2002} with scale matrix $\mat{\Omega} = \smash{\diag(\sqrt{\bb{\iota}_d}) \, \mat{R}_d(0.9) \diag(\sqrt{\bb{\iota}_d})}$, location vector $\bb{\xi} = 4\bb{1}_d$, and slant parameter $\bb{\alpha} = -5\bb{1}_d$;
\item a MIG distribution with parameters $\bb{\xi} = 2\bb{1}_d$ and $\mat{\Omega}=\mat{R}_d (0.5)$.
\end{enumerate}
Figure~\ref{fig:simulation.2d.densities} shows the log-density curves for the four data generating mechanisms in the bivariate setting. Distributions $F_1$ and $F_4$ have modes close to the boundary of the support. Only $F_4$ satisfies the integrability conditions of Corollary~\ref{cor:MISE.optimal.density}.

\vspace{-1mm}
\begin{figure}[H]
\centering
\includegraphics[width = 1\linewidth]{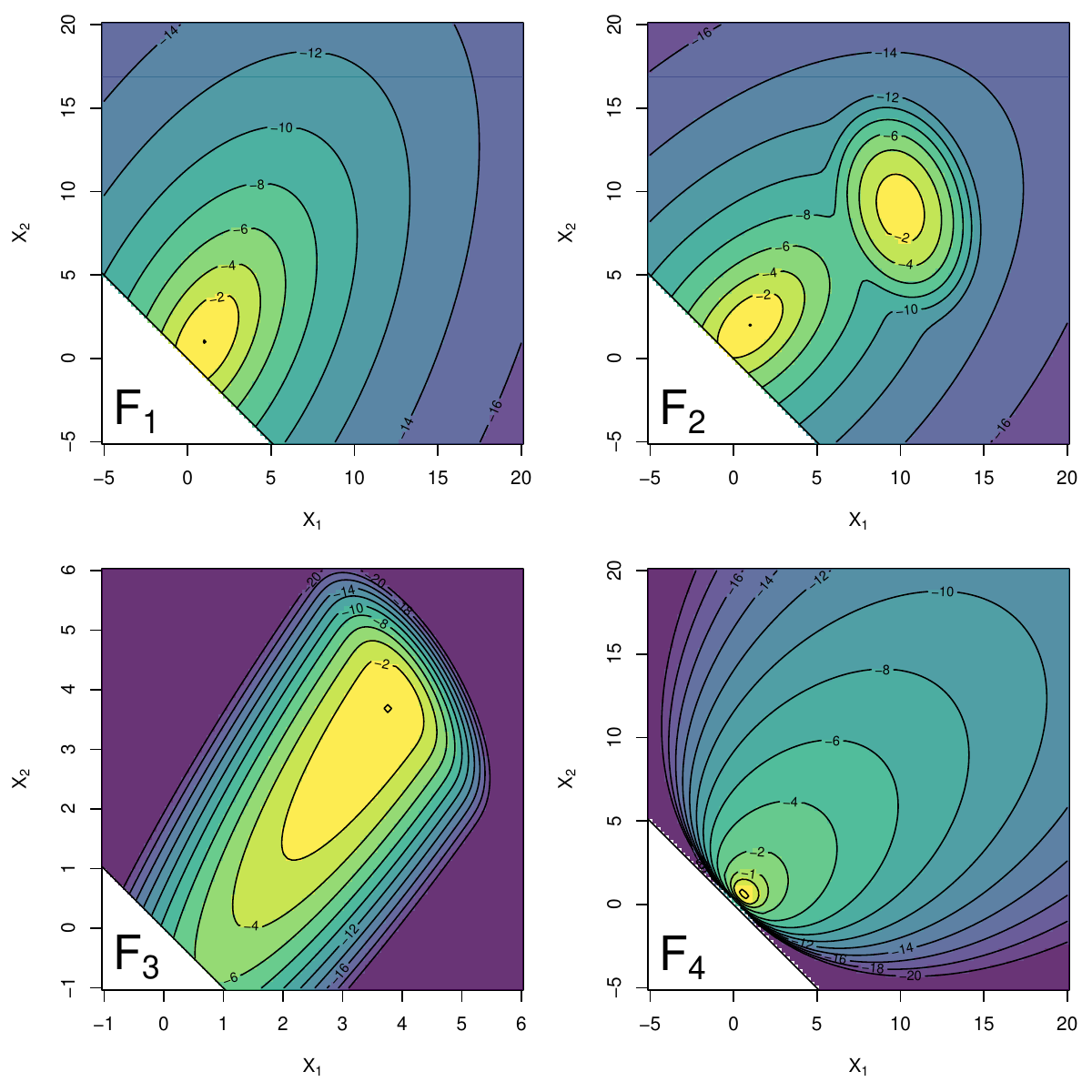}\vspace{-4mm}
\caption{Log-densities of the bivariate distributions $F_1,\ldots,F_4$. Contour lines indicate deviations in log-density from the mode. The color gradient ranges continuously from purple (lowest density) to yellow (highest density).\vspace{-10mm}}
\label{fig:simulation.2d.densities}
\end{figure}

\vspace{-3mm}
\subsubsection{Bandwidth selection}\label{subsec:simulation.bandwidth.selection}

This section describes several choices of bandwidth matrices $\mat{H}$ for the MIG kernel density estimator and its competitors, given a sample $\bb{x}_1,\ldots,\bb{x}_n\in \mathcal{H}(\bb{\beta})$.

Among these, bandwidth matrices that minimize the asymptotic mean integrated squared error (AMISE) derived in Corollary~\ref{cor:MISE.optimal.density} are considered, using the plug-in approach discussed in Remark~\ref{rem:plug.in.bandwidth} with MIG pilot densities. In addition, bandwidths that maximize the leave-one-out likelihood cross-validation (LCV) score \citep{Habbema.Hermans.Broek:1974} are examined, viz.
\begin{equation}\label{eq:LCV}
\argmax_{\mat{H}\in \mathcal{S}_{++}^d} \frac{1}{n} \sum_{i=1}^n \ln\big\{\hat{f}^{-(i)}_{n-1, \mat{H}}(\bb{x}_i)\big\},
\end{equation}
where
\[
\hat{f}^{-(i)}_{n-1, \mat{H}}(\bb{x}_i) = \frac{1}{n-1} \sum_{\substack{j=1 \\ j \neq i}}^n k_{\bb{\beta}, \bb{x}_i, \mat{H}}(\bb{x}_j);
\]
see, e.g., Section~2.1 of~\citet{Zhang.King.Hyndman:2006} for more details. The LCV criterion is sensitive to extremes and outliers and may lead to oversmoothing when such observations are present. Other options include the least-squares cross-validation (LSCV) criterion \citep{Rudemo:1982,Bowman:1984}, viz.
\begin{equation}\label{eq:LSCV}
\argmax_{\mat{H}\in \mathcal{S}_{++}^d} \left\{\frac{1}{n}\sum_{i=1}^n \hat{f}^{-(i)}_{n-1, \mat{H}}(\bb{x}_i) - \frac{1}{2}\int_{\mathcal{H}(\bb{\beta})}\hat{f}^2_{n, \mat{H}}(\bb{\xi}) \, \rd \bb{\xi}\right\},
\end{equation}
and the robust likelihood cross-validation (RLCV) criterion of \cite{Wu:2019}, which interpolates between the LCV and LSCV,
\begin{equation}\label{eq:RLCV}
\argmax_{\mat{H}\in \mathcal{S}_{++}^d} \frac{1}{n}\sum_{i=1}^n \ln^{\star}\big\{\hat{f}_{n, \mat{H}}(\bb{\xi})\big\} - b_n^*(\mat{H}),
\end{equation}
where
\begin{equation}
b_n^*(\mat{H}) = \int_{\mathcal{H}(\bb{\beta})} \left[\ind_{\{\hat{f}_{n, \mat{H}}(\bb{\xi}) \geq a_n\}} + \ind_{\{\hat{f}_{n, \mat{H}}(\bb{\xi}) < a_n\}} \frac{\hat{f}_{n, \mat{H}}(\bb{\xi})}{2a_n}\right] \hat{f}_{n, \mat{H}}(\bb{\xi}) \, \rd \bb{\xi}, \label{eq:bn}
\end{equation}
\[
\ln^{\star}(x) = \ind_{\{x \geq a_n\}}\ln(x) + \ind_{\{x < a_n\}} \left\{\ln(a_n) - 1 + \frac{x}{a_n}\right\}, \nonumber
\]
using the automatic threshold proposed by \cite{Wu:2019},
\[
a_n = \frac{\Gamma(d/2) \{\ln(n)\}^{1-d/2}}{(2\pi)^{d/2} |\widehat{\bb{\Sigma}}_n|^{1/2} \, n},
\]
with $\widehat{\bb{\Sigma}}_n = n^{-1} \sum_{i=1}^n (\bb{x}_i - \bar{\bb{x}}_n) (\bb{x}_i - \bar{\bb{x}}_n)^{\top}$ denoting the empirical covariance matrix.

The integrals \eqref{eq:LSCV} and \eqref{eq:bn} in the LSCV and RLCV criteria can be evaluated via Monte Carlo, as pointed out in Remark~\ref{rem:plug.in.bandwidth}. For all methods other than the AMISE minimization approach, the more popular and simpler kernel density estimator plug-in, as in \cite{Wu:2019}, was used based on the sample.

Estimators that minimize the AMISE (MIG kernel only) as well as those optimizing the LCV, LSCV, and RLCV criteria were examined. With the exception of the transformation estimator, $\smash{\widehat{f}_{n, \mat{H}}^{(\mathrm{trans})}(\bb{\xi})}$, the entries of the unstructured $d \times d$ bandwidth matrix were optimized using a parametrization based on its Cholesky root with positivity constraints on the diagonal \citep{Pinheiro.Bates:1996}. Also considered was the isotropic AMISE bandwidth matrix of the form $h^2 \mat{S}_n$ based on the transformed data, as discussed in Remark~\ref{rem:scaling.isotropic}.

For clarity, a formal list of the kernel methods and corresponding bandwidth selection approaches is provided below:

\begin{enumerate}[\quad A:]
\item MIG kernel with a spherical transformation and isotropic diagonal bandwidth matrix minimizing the AMISE \eqref{eq:AMISE_isotropic}, with MIG plug-in for $f$;
\item MIG kernel, full bandwidth matrix minimizing the AMISE \eqref{eq:AMISE.full}, with MIG plug-in for $f$;
\item MIG kernel, full bandwidth matrix minimizing the LCV criterion \eqref{eq:LCV};
\item MIG kernel, full bandwidth matrix minimizing the RLCV criterion \eqref{eq:RLCV};
\item truncated Gaussian kernel \eqref{eq:trunc.Gaussian.KDE}, full bandwidth matrix, kernel plug-in, LCV criterion \eqref{eq:LCV};
\item truncated Gaussian kernel \eqref{eq:trunc.Gaussian.KDE}, full bandwidth matrix, kernel plug-in, RLCV criterion \eqref{eq:RLCV};
\item truncated Gaussian kernel \eqref{eq:trunc.Gaussian.KDE}, full bandwidth matrix, kernel plug-in, LSCV criterion \eqref{eq:LSCV};
\item Gaussian kernel on half-space \eqref{eq:trans.Gaussian.KDE}, diagonal bandwidth matrix, kernel plug-in, LCV criterion \eqref{eq:LCV};
\item Gaussian kernel on half-space \eqref{eq:trans.Gaussian.KDE}, diagonal bandwidth matrix, kernel plug-in, RLCV criterion \eqref{eq:RLCV};
\item Gaussian kernel on half-space \eqref{eq:trans.Gaussian.KDE}, diagonal bandwidth matrix, kernel plug-in, LSCV criterion \eqref{eq:LSCV}.
\end{enumerate}

\subsubsection{Performance measures}\label{subsec:simulation.performance.measures}

To assess the performance of the MIG kernel density estimator and its competitors, three metrics (or measures) are considered in the simulation study.

For each dimension $d \in \{2, 3\}$, target distribution $F_i$ with $i \in \{1, \ldots, 4\}$, sample size $n \in \{250, 500, 1000\}$, and kernel/bandwidth method ($\text{method}\in \{A,\ldots,J\}$), the first metric included in the study is the root mean integrated squared error (RMISE) and is computed as
\[
\mathrm{RMISE} = \sqrt{\int_{\mathcal{H}(\bb{\beta})} \big\{\hat{f}_n^{(\text{method})}(\bb{x}) - f_i(\bb{x})\big\}^2 \, \rd \bb{x}}.
\]

The second metric considered is the boundary root mean integrated squared error (BRMISE), obtained by restricting the integration in the RMISE to the boundary region $\mathcal{B}(\bb{\beta}) = \big\{\bb{x} \in \mathcal{H}(\bb{\beta}) : \bb{x}^{\top} \bb{\beta} \leq (n/d)^{-1/(d+4)}\big\}$, and given by
\[
\mathrm{BRMISE} = \sqrt{\int_{\mathcal{B}(\bb{\beta})} \big\{\hat{f}_n^{(\text{method})}(\bb{x}) - f_i(\bb{x})\big\}^2 \, \rd \bb{x}}.
\]
Both the RMISE and BRMISE are estimated via numerical integration using the $h$-adaptive integration routines of Steven G. Johnson \citep{Genz.Malik:1980,Berntsen.Espelid.Genz:1991,cubatureC} after rotating the space with the projection matrix $\mat{Q}$ from Proposition~\ref{prop:MIG.stochastic.representation}, with an absolute error tolerance of $10^{-4}$ or up to a maximum of $5\times 10^5$ evaluations.

The third metric is the Kullback--Leibler divergence (KLD),
\[
\mathrm{KLD}(f_i \parallel \hat{f}_n^{(\text{method})}) = \int_{\mathcal{H}(\bb{\beta})} \left[\ln\big\{f_i(\bb{x})/\hat{f}_n^{(\text{method})}(\bb{x})\big\}\right] f_i(\bb{x}) \, \rd \bb{x},
\]
which quantifies the information lost when approximating the true (target) density $f_i$ for some $i \in \{1, \dots,4\}$ with the estimated density $\smash{\hat{f}_n^{(\text{method})}}$. It is computed via Monte Carlo integration using $10^4$ random draws from $F_i$.

\subsubsection{Results}\label{subsec:simulation.results}

The results of the simulation study appear in boxplots representing the $10^3$ measurements, calculated via numerical integration, for each of the three metrics (RMISE, BRMISE, KLD) defined in Section~\ref{subsec:simulation.performance.measures}, each $d\in \{2,3\}$, each $n\in \{250,500,1000\}$, and each of the combinations of estimators and selection methods (A--J) listed in Section~\ref{subsec:simulation.bandwidth.selection}; see Figure~\ref{fig:RMISE} for RMISE, Figure~\ref{fig:BRMISE} for BRMISE, and Figure~\ref{fig:KLD} for KLD. The findings are briefly summarized below.

\begin{enumerate}
\item No combination of kernel and bandwidth selection method universally dominates the others for any of the performance metrics considered.

\item The RMISE and BRMISE decrease with sample size, as expected.

\item In some scenarios, the variability of the performance measure is very much sample dependent, and can paradoxically increase when the sample size gets larger. The RLCV bandwidth selection method is seemingly less sensitive to these changes.

\item Overall, the truncated Gaussian kernel density estimator has the lowest median RMISE for the distributions $F_1$, $F_2$, and $F_3$ among the three kernel families considered. It performs the worst for $F_4$ due to two factors, namely the higher boundary bias and the sensitivity to outliers generated from the MIG distribution.

\item The MIG kernel density estimator is often more variable than its counterparts, except for data generated from $F_4$. This is in concordance with the asymptotic variance of the MIG kernel density estimator in Proposition~\ref{prop:asymptotic.variance} and discussed in Remark~\ref{rem:boundary.variance}.

\item By nature, the LSCV and the AMISE minimization methods yield lower RMISE than other criteria. The LCV criterion returns lower KLD estimates but the AMISE method ($B$) remains competitive.

\item Using a full bandwidth matrix does better than using a spherical rotation with a single parameter for AMISE minimization, but this performance comes with increased computational cost.

\item The MIG kernel density estimator has higher BRMISE than its counterparts for certain bandwidth selection methods when the data generated has lots of mass near the boundary. However, the MIG kernel density estimator with AMISE minimization has the smallest boundary bias of the three methods considered in the simulation, in concordance with the asymptotic bias expression derived in Proposition~\ref{prop:asymptotic.bias} and discussed in Remark~\ref{rem:boundary.bias}.
\end{enumerate}

\begin{figure}[!htp]
\centering
\includegraphics[width = 0.89\linewidth]{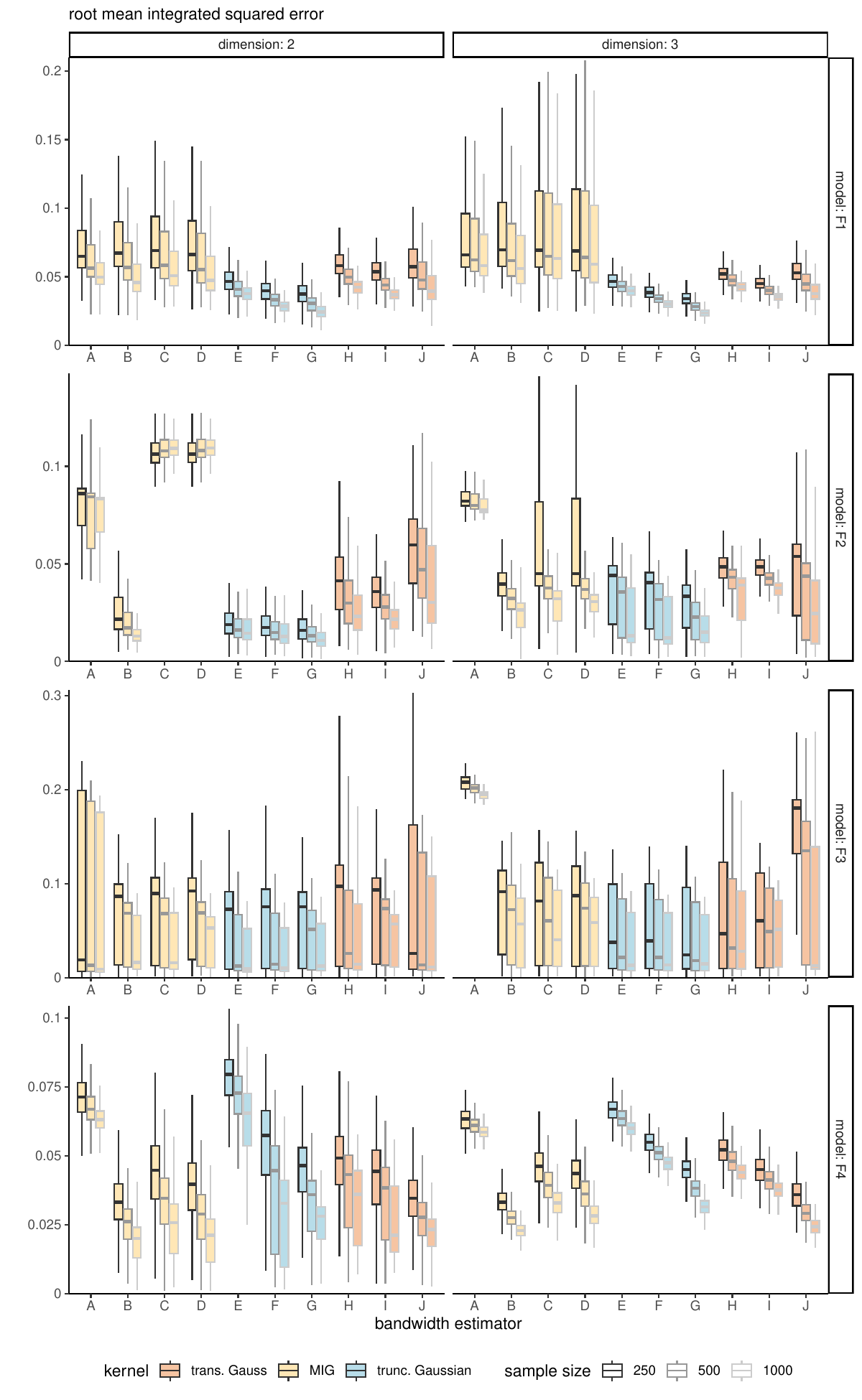}
\caption{Boxplots of $10^3$ RMISE measurements, estimated via numerical integration, for $d\in \{2,3\}$, $n\in \{250,500,1000\}$, and the combinations of estimators and selection methods ($\mathrm{A}$--$\mathrm{J}$) listed in Section~\ref{subsec:simulation.bandwidth.selection}.}
\label{fig:RMISE}
\end{figure}

\begin{figure}[!htp]
\centering
\includegraphics[width = 0.89\linewidth]{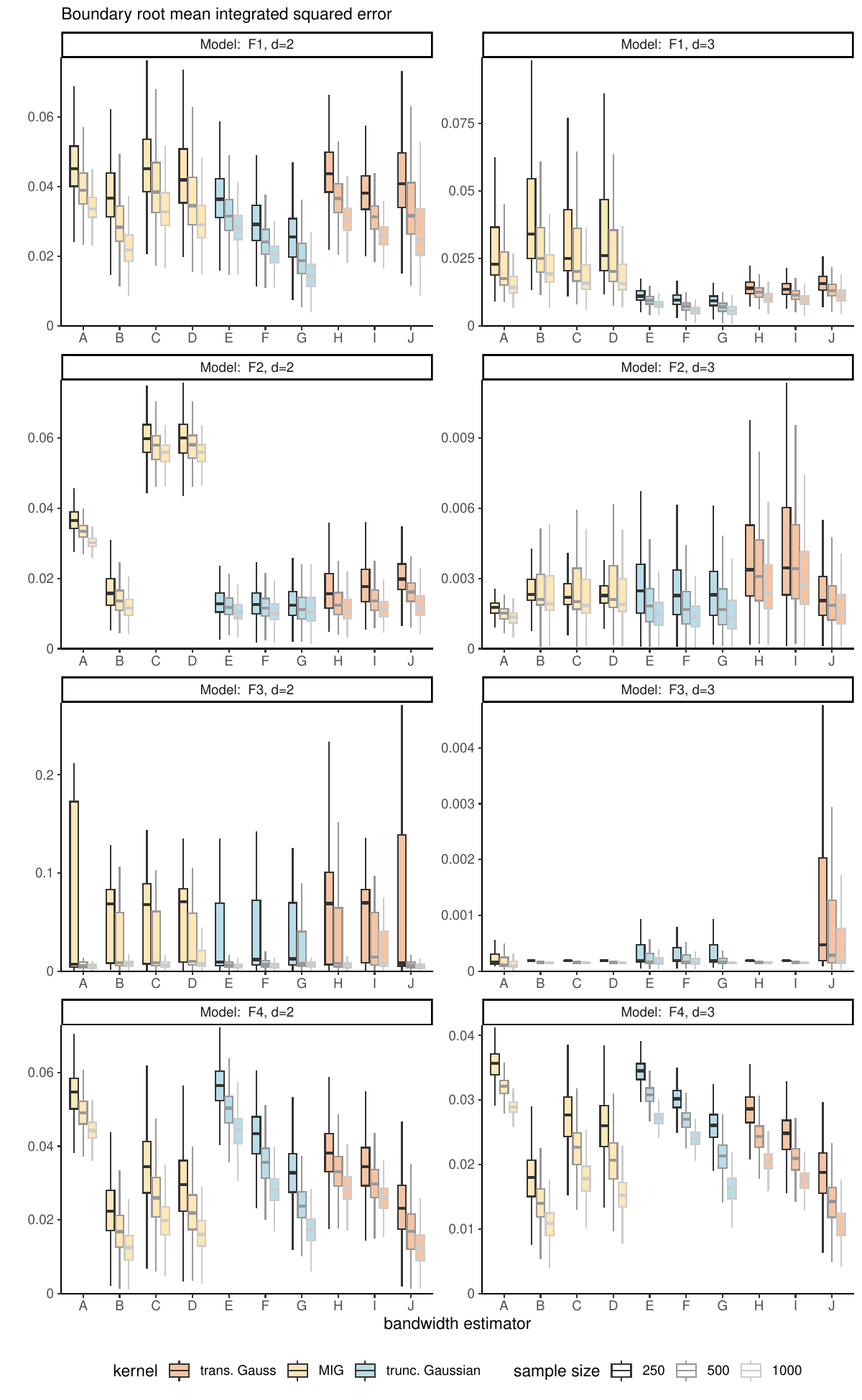}
\caption{Boxplots of $10^3$ BRMISE measurements, estimated via numerical integration restricted to the boundary region $\mathcal{B}(\bb{\beta})$, for $d\in \{2,3\}$, $n\in \{250,500,1000\}$, and the combinations of estimators and selection methods ($\mathrm{A}$--$\mathrm{J}$) listed in Section~\ref{subsec:simulation.bandwidth.selection}.}
\label{fig:BRMISE}
\end{figure}

\begin{figure}[!htp]
\centering
\includegraphics[width = 0.89\linewidth]{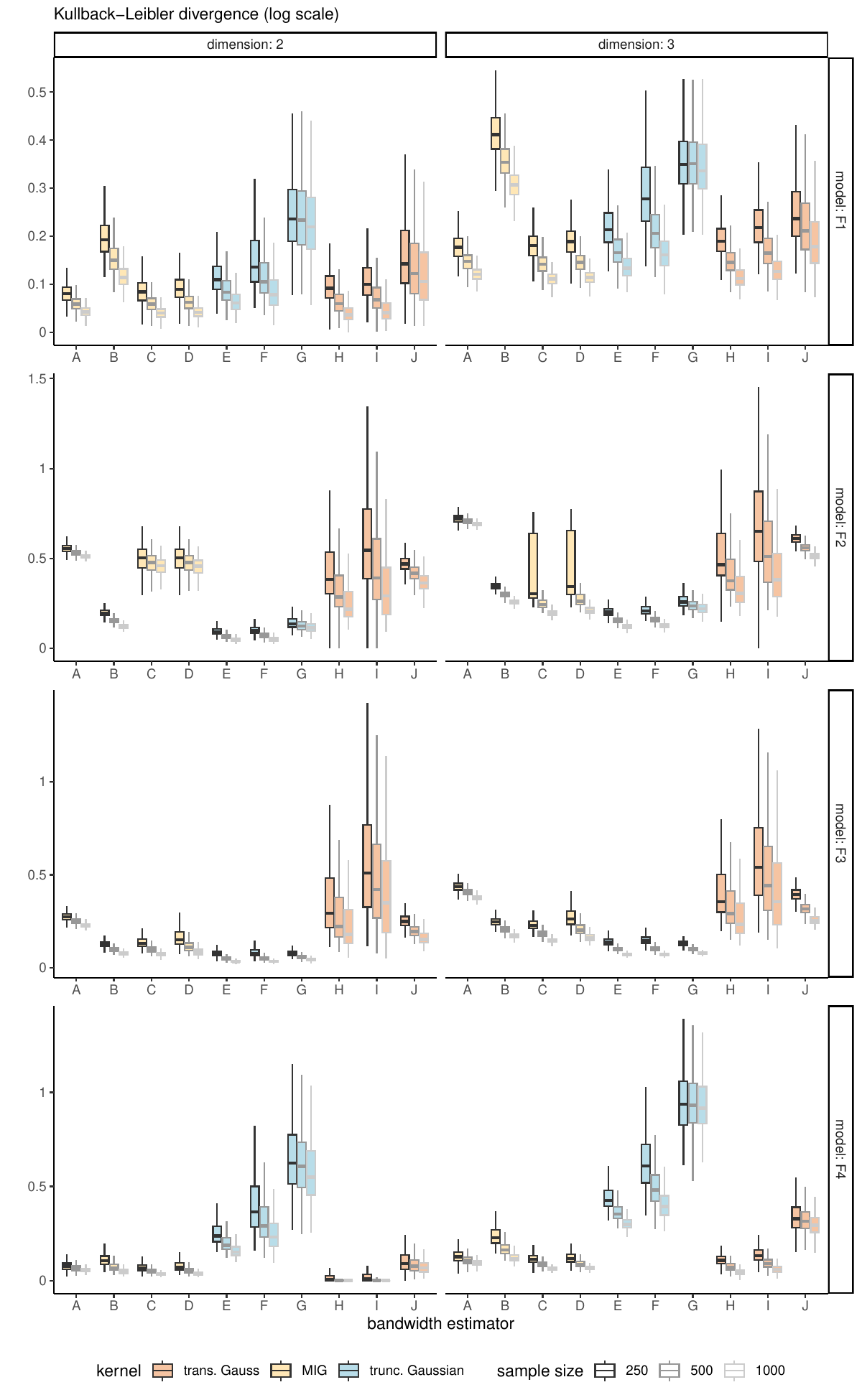}
\caption{Boxplots of $10^3$ log-transformed KLD measurements, $\max\{0,\ln(\mathsf{KLD}+1)\}$, estimated via Monte Carlo sampling from the target distributions, for $d\in \{2,3\}$, $n\in \{250,500,1000\}$, and the combinations of estimators and selection methods ($\mathrm{A}$--$\mathrm{J}$) listed in Section~\ref{subsec:simulation.bandwidth.selection}.}
\label{fig:KLD}
\end{figure}

\subsubsection{Discussion}\label{subsec:simulation.discussion}

This section outlines remarks on certain practical aspects of the computational complexity, numerical stability (robustness) and estimation times of the bandwidth selection procedures, in addition to other numerical considerations.

\begin{enumerate}
\item \textit{Complexity:} The LCV criterion is $\OO(n)$ and the number of bandwidth entries to optimize for a $d$-dimensional model is  $d(d+1)/2$, so $\OO(d^2)$. As expected, the spectral norm of the optimal bandwidth matrix increases with the dimension $d$ and decreases as the sample size $n$ grows. Using a diagonal bandwidth matrix reduces to $d$ the number of parameters to be estimated, and leads to gains in speed at the expense of poorer performance (higher MISE).

\item \textit{Robustness}: The bandwidth matrix estimation criteria (especially LSCV) which require Monte Carlo integration over the half-space for the bias term estimation are disturbingly sensitive to the choice of plug-in. For the Gaussian estimator for the rotated data (\texttt{hsgauss}; Appendix~\ref{app:R.code}), only a diagonal bandwidth was considered because the orthogonal transformation matrix yields uncorrelated components (except for the additional log-transformed radius). Preliminary tests (not reported) showed that using the full matrix yielded little improvement over this approach.

\item \textit{Estimation times}: Estimation of the bandwidth matrix Cholesky factor is fast, although the RLCV and LSCV criteria are about three times as slow as the LCV criterion due to the need to perform Monte Carlo integration. Using the kernel as plug-in for the target density $f_i$ in integrals can lead to a considerable speed-up because $n$ is typically small, at the expense of additional variability. For $d=3$, and a sample size $n=1000$, all calculation times were under a minute, with an average of, e.g., three seconds for the truncated Gaussian kernel density estimator with LCV.

\item \textit{Other considerations}: It is numerically challenging to evaluate the RMISE and BRMISE by numerical integration, and the problem becomes more challenging when $d$ increases as the mass of the integral becomes increasingly diffuse. Different numerical routines give answers that are sometimes orders of magnitude apart, so some care is needed. Preliminary rotation of the space is preferable, as integrating over $\R_{+} \times \R^{d-1}$ is easier.
\end{enumerate}

\begin{remark}
The MIG distribution becomes increasingly asymmetrical as its mode approaches the boundary of $\mathcal{H}(\bb{\beta})$; see Figure~\ref{fig:MIG.LCV.problems} for an illustration. This growing asymmetry affects the properties of the MIG kernel density estimator when the mode of the underlying data generating process lies on the boundary; in particular, the log-likelihood cross-validation score becomes fully driven by points close to the boundary, whereas Gaussian-based estimators are more affected by outliers and values far from the mode. If there are many observations very close to the boundary of the half-space, using a truncated Gaussian or truncated Student-$t$ kernel would be preferable and might lead to lower RMISE values.
\end{remark}

\begin{figure}[!t]
\centering
\includegraphics[width=\linewidth]{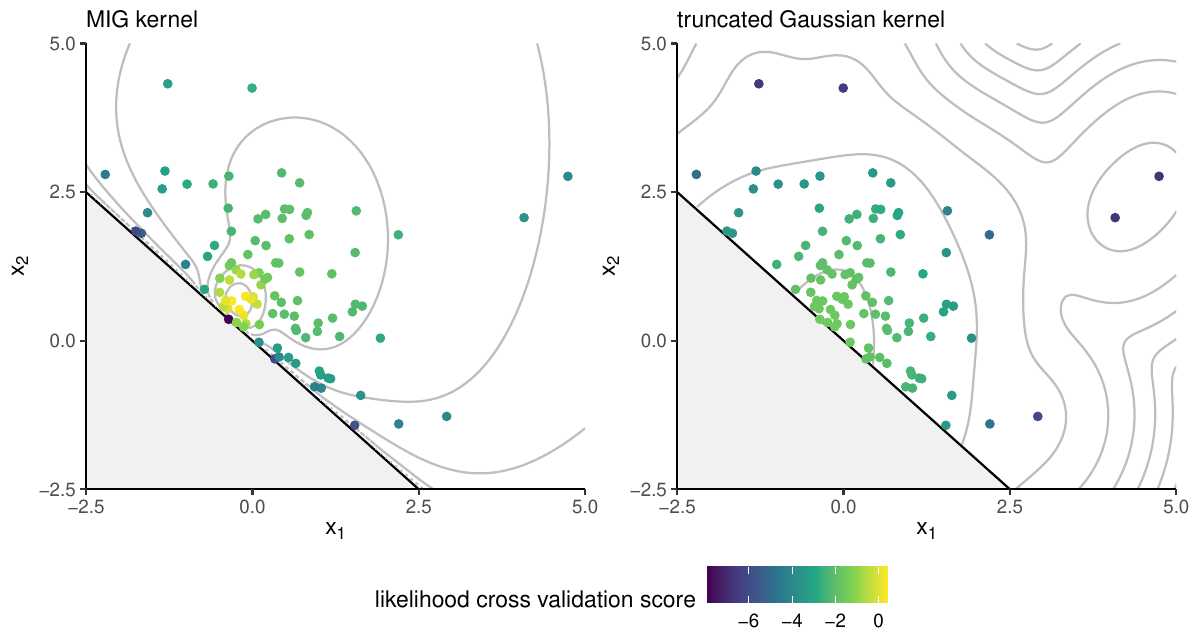}
\caption{Contour plot of MIG (left) and truncated Gaussian (right) kernel density estimators with bandwidth matrices chosen according to the likelihood cross validation (LCV) score based on a sample of size 100 truncated Student-$t$ with three degrees of freedom. The colors indicate the LCV scores of each data points: the lower their value, the more influential the observation in the bandwidth fit.}
\label{fig:MIG.LCV.problems}
\end{figure}

\subsection{Real-data application}\label{sec:data.application}

A geomagnetic storm is a major disturbance of Earth's magnetosphere due to a very efficient exchange of energy from the solar wind into the space environment surrounding Earth. Following \citet{WDCG}, consider data collected between 1957 and 2014 on the absolute magnitude of extreme geomagnetic storms lasting over 48 hours, measured in disturbance-storm time (dst).

A generalized Pareto distribution is fitted to exceedances above $u=220$ dst, with density given, for all $x \in (0, \infty)$, by
\[
f(x) = \sigma_u^{-1}(1 + \xi x/\sigma_u)_{+}^{-1/\xi-1}.
\]
Coupled with a maximum data information prior, $10^3$ independent draws from the posterior distribution of $(\sigma_u, \xi)$ were obtained using the generalized ratio-of-uniform method \citep{revdbayes,Wakefield:1991}. The bivariate posterior is supported on the (truncated) half-space
\[
\mathcal{H}_2 (1, m_n) = \{(\sigma_u, \xi)\in (0,\infty) \times \R : \sigma_u + \xi m_n > 0\}.
\]
where $m_n = \max(x_1, \ldots, x_n)$.

The left panel of Figure~\ref{fig:geomagnetic} shows the $10^3$ posterior draws used for estimating the kernel density. Given the small sample size, $n = 52$, the posterior distribution of the generalized Pareto is far from Gaussian, notwithstanding the support constraint. The panel shows both a scatterplot of the observations and contour curves for the MIG density with maximum-likelihood-estimated parameters.

The right panel of Figure~\ref{fig:geomagnetic} exhibits hexagonal bins of $10^5$ sample points simulated from the posterior distribution, along with contour curves based on the training posterior draws. These curves were obtained by computing the full bandwidth matrix that minimizes the LCV (black) and the AMISE (gray) using the MIG distribution to approximate the latter via Monte Carlo. Both bandwidths have correlations of roughly $-2/3$, which matches the sample correlation of the training set. However, the standard deviation in the shape direction ($\xi$) for the LCV bandwidth is roughly twice that of the AMISE.

\begin{figure}[t!]
\centering
\includegraphics[width = \linewidth]{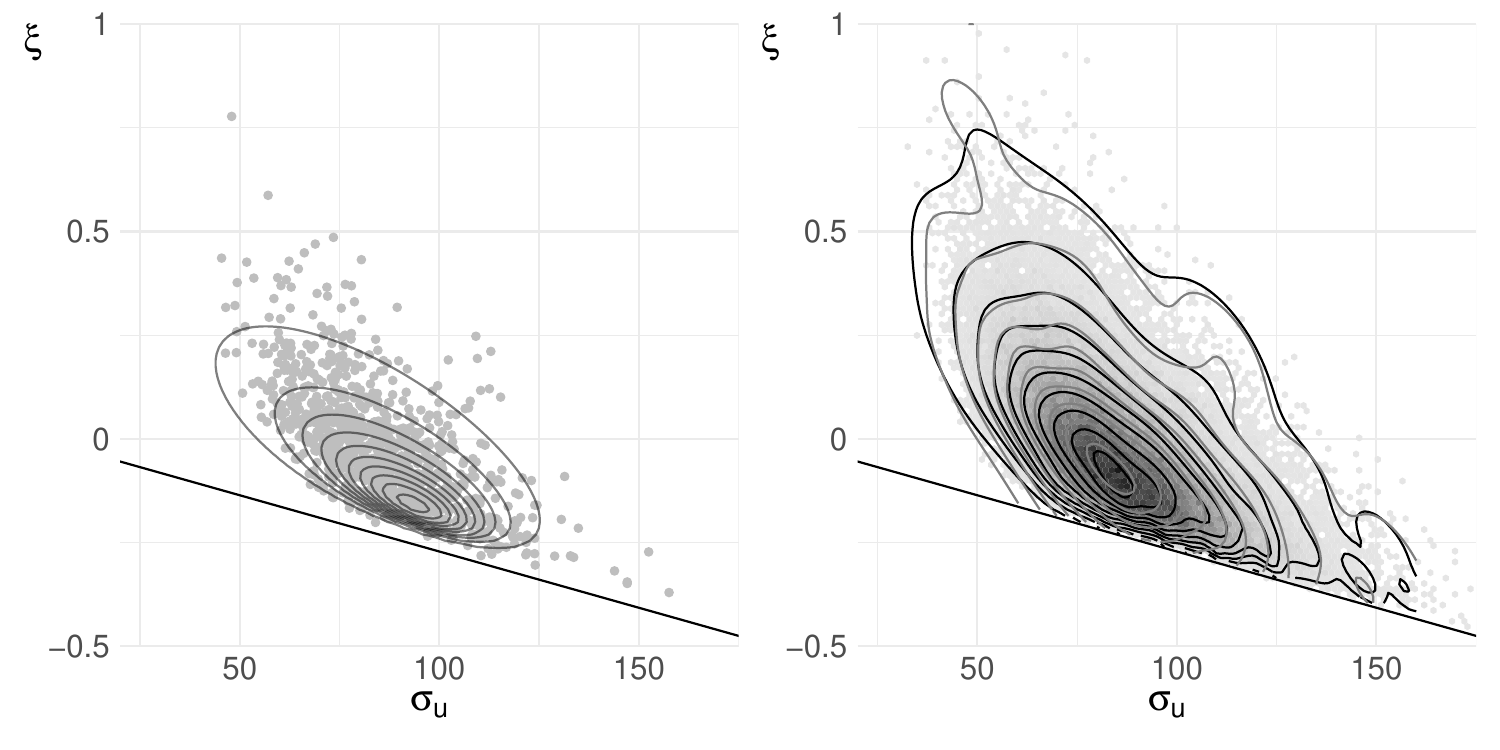}\vspace{-2mm}
\caption{Left: scatter plot of $10^3$ posterior sample points (training set) from a generalized Pareto model with scale $\sigma_u$ and shape $\xi$, along with posterior density contour curves based on an MIG distribution. Right: bivariate histogram with hexagonal bins obtained from $10^5$ posterior sample points generated using the generalized ratio-of-uniform method, with MIG kernel density estimator contour curves based on the LCV (black) and AMISE (gray), based on the training set.}\label{fig:geomagnetic}
\end{figure}

As one can see, the resulting density estimates faithfully reflect the distribution of the posterior draws. The MIG kernel density estimator is, by construction, more flexible than fixed parametric models such as the MIG density fitted in the left panel of Figure~\ref{fig:geomagnetic}; moreover, this result is achieved at a small cost. In more complex settings where exact sampling from the posterior is not available, kernel smoothing is a viable, cost-effective alternative that can lead to better approximations of the posterior density. While it is possible to use other density estimators, such as elliptical kernel density estimators truncated over the half-space, which can also generalize to include additional linear inequality constraints, the MIG kernel does not require estimation of the normalizing constant and naturally adapts to the geometry of the half-space due to its asymmetry.

\section{Exact simulation algorithm}\label{sec:rng.MIG}

This section presents a new algorithm for generating MIG random vectors that is faster and more accurate than the approach proposed by \cite{MR2019130}, who exploited the Brownian hitting location representation mentioned in Section~\ref{sec:intro}. A separation-of-variable algorithm is also proposed for efficient evaluation of the MIG cdf using sequential importance sampling; see Remark~\ref{rem:cdf.estimation}. Both methods rely on Theorem~1~(3) of \cite{MR2019130}.

\begin{proposition}[Sampling from the MIG distribution]\label{prop:exact.sim}
Let the vectors $\bb{\beta}\in \R^d$, $\bb{\xi}\in \mathcal{H}(\bb{\beta})$, and the positive definite matrix $\mat{\Omega}\in \mathcal{S}_{++}^d$ be given. Recall from Proposition~\ref{prop:MIG.stochastic.representation} the $d \times d$ projection matrix~$\mat{Q}$, the inverse Gaussian random variable $R$, and the $(d-1)$-vector $\bb{Z}$ which is Gaussian given $R = r$. Given that $\mat{Q}$ is invertible, one has $\mat{Q}^{-1}(R, \bb{Z}^{\top}){\vphantom{Z}}^{\top} \sim \mathrm{MIG}(\bb{\beta}, \bb{\xi}, \mat{\Omega})$.
\end{proposition}

The result of Proposition~\ref{prop:exact.sim} is a direct consequence of the stochastic representation of \citet{MR2019130} for the MIG distribution, stated in Proposition~\ref{prop:MIG.stochastic.representation}. The numerical implementation is also straightforward: (i) the inverse of $\mat{Q}$ can be computed using its singular value decomposition; (ii) an inverse Gaussian random variable $R$ is drawn using a dedicated algorithm; and, (iii) conditional on that draw, a $(d-1)$-dimensional Gaussian random vector $\bb{Z}$ is generated as in \eqref{eq:Z2}. Then, one sets $\bb{X} = \mat{Q}^{-1} (R, \bb{Z}^{\top}){\vphantom{Z}}^{\top} \sim \mathrm{MIG}(\bb{\beta}, \bb{\xi}, \mat{\Omega})$.

\begin{remark}
In the simulation study of Section~\ref{sec:simulation.study}, observations from the MIG distribution (viz. $F_4$) were generated by drawing inverse Gaussian random variables using the \textsf{R} package \texttt{statmod} \citep{Giner.Smyth:2016}.
\end{remark}

\begin{remark}
Compared with Minami's algorithm \citep[Theorem~2]{MR2019130}, which simulates the first-hitting location on a hyperplane of a $d$-variate Brownian motion with correlated components, the procedure described below Proposition~\ref{prop:exact.sim} offers two main advantages: it delivers exact samples and achieves significantly higher speed. Minami's approach is approximate because it relies on discretizing the Brownian paths, leading to slow computing times even in low dimension. In contrast, the exact algorithm involves inverting the covariance matrix $\mat{\Omega}$ and performing cross-products in quadratic forms, which yields a computational complexity of $\OO(d^3 + n\,d^2)$, linear in the sample size $n$.

Table~\ref{table.1} reports the mean computation time and corresponding standard deviation for generating $n = 100$ MIG random vectors with randomized parameters, across different dimensions $d$. Minami's Brownian-motion method was run with a time-discretization step of $10^{-3}$ to approximate first-hitting locations.

\begin{table}[!t]
\setlength{\tabcolsep}{6.3pt}
\caption{Mean time (standard deviation) in milliseconds over 25 replications for generating $n = 100$ observations from a MIG distribution with randomized parameters, as a function of $d$.}
\centering
\begin{tabular}[t]{llllll}
\toprule
Method & \multicolumn{1}{l}{$d=2$} & \multicolumn{1}{l}{$d=4$} & \multicolumn{1}{l}{$d=8$} & \multicolumn{1}{l}{$d=16$} & \multicolumn{1}{l}{$d=32$} \\
\midrule
Exact & 0.80 (0.27) & 1.31 (0.37) & 1.61 (0.36) & 1.98 (0.42) & 3.22 (0.40) \\
Minami & 259 (45) & 2275 (984) & 14228 (5024) & 38373 (21001) & 197823 (366493) \\
\bottomrule
\end{tabular}
\label{table.1}
\end{table}

The table shows that the exact generator is up to five orders of magnitude faster than Minami's Brownian-motion algorithm; the same pattern is visible in Figure~\ref{fig:rmig_timing}, which plots the mean times with 95\% confidence intervals and log-log least-squares lines.
\end{remark}

\begin{figure}[!b]
\centering
\includegraphics[width=1\textwidth]{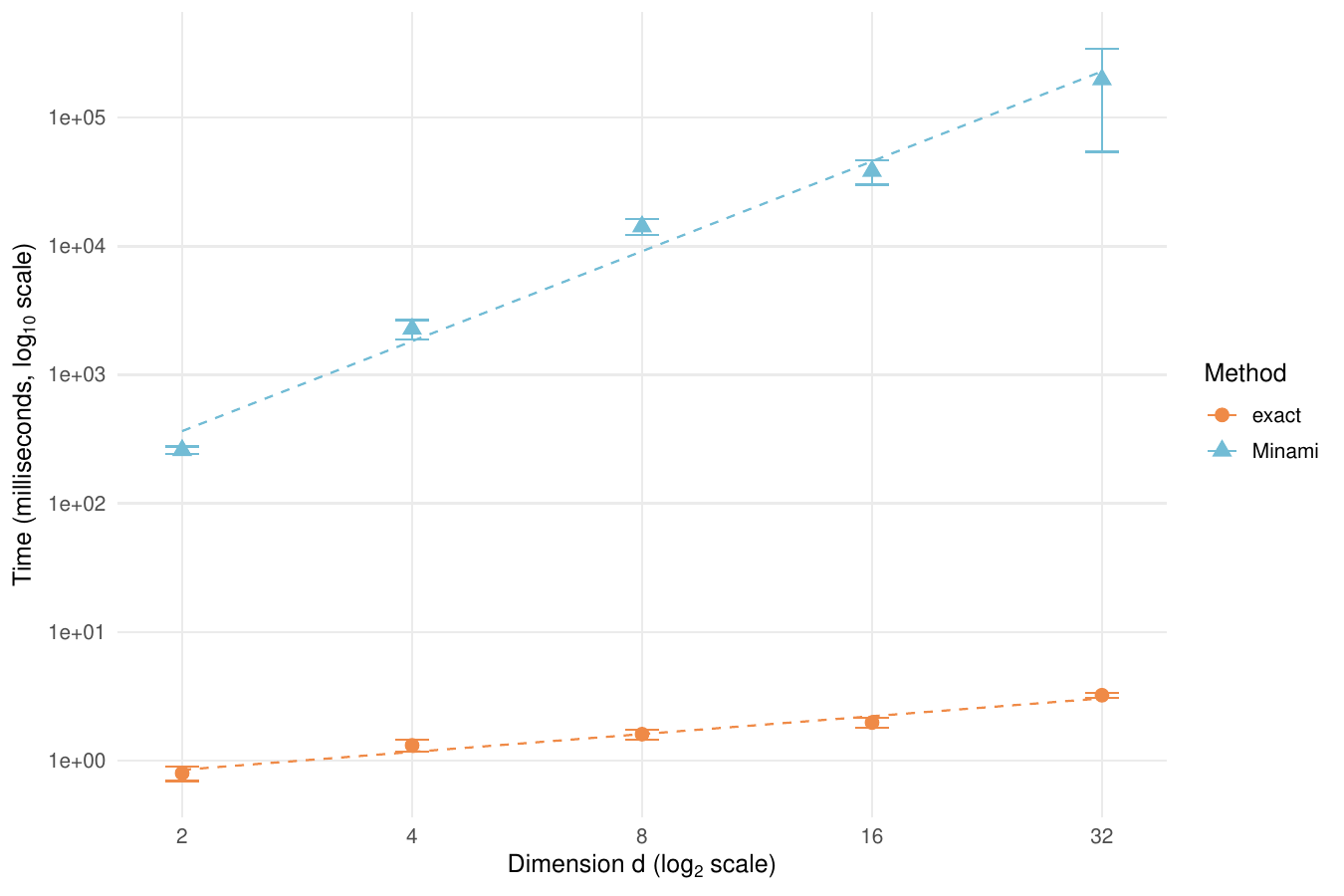} \\
\caption{Computation time for \texttt{rmig} generation versus dimension. Points show the mean run time across 25 repetitions, and vertical bars 95\% Wald-based confidence intervals for the mean. Dashed lines give ordinary least-squares fits in log-log scale. The exact algorithm is compared with Minami's Brownian motion method.}
\label{fig:rmig_timing}
\end{figure}

\begin{remark}\label{rem:cdf.estimation}
As simulation is cheap, the cdf of a MIG random vector can be readily estimated via Monte Carlo by drawing vectors from the corresponding model. It is tempting to leverage the same stochastic representation to evaluate $\Pr(\bb{X} \leq \bb{q})$, where $\bb{X} \sim \mathrm{MIG}(\bb{\beta}, \bb{\xi}, \mat{\Omega})$, using sequential importance sampling \citep[e.g.,][]{Genz.Bretz:2009} for the Gaussian component. Unfortunately, the linear transformation $\mat{Q}_2$ of the region of integration, $(-\bb{\infty}_d, \bb{q}] \cap \mathcal{H}(\bb{\beta})$, gives rise to a polytope and obtaining the transformed region of integration requires solving repeated linear programs, making the approach uncompetitive beyond the bivariate case, for which explicit expressions can be derived for the integration bounds. More details are provided in Appendix~\ref{app:cdf.evaluation}.
\end{remark}

\section{Conclusion and outlook}\label{sec:conclusion}

This paper introduced a new asymmetric kernel density estimator for data supported on half-spaces, using the MIG distribution of \citet{MR2019130} as the locally adaptive kernel. The proposed estimator, which intrinsically respects the geometry of half-spaces, addresses the well-known boundary bias problem from which traditional kernel estimators suffer. Several theoretical results were established for the estimator, including its asymptotic normality and an expression for its MISE, as well as local limit theorems and probability metric bounds for the MIG distribution which were instrumental in deriving the asymptotic variance. A thorough simulation study demonstrated the MIG kernel density estimator's good finite-sample performance under various configurations, alongside two new boundary-adapted competitors that we designed for comparison. A real-data application showcased its potential for use in a Bayesian setting.

While the proposed estimator shows competitive performance relative to its counterparts in the simulation study, it is not without limitations. First, although this paper considered bandwidth criterions such LCV, RLCV, LSCV, and AMISE minimization, some practical implementations could require more automated or data-driven selection methods that balance computational demands with statistical efficiency. For example, as the dimension $d$ grows, it is an open question to determine what kind of bandwidth matrix (scalar, diagonal, full) should be optimized and under what criterion. More sophisticated selection procedures could also be implemented, but then the computational cost always remains an issue, especially in high dimensions.

Second, establishing uniform strong consistency results on growing sequences of compacts and related large deviation bounds could provide deeper insight into the estimator's convergence behavior. This would be valuable for theoretical completeness and for guiding large-sample and high-dimensional applications.

Third, although half-spaces are common in areas such as constrained inference and machine learning, there is scope for extending asymmetric kernel methods to more complicated geometric domains, such as convex cones or non-compact manifolds; see, e.g, \citet{Kim_Richards_2008_worshop}, \cite{MR2838725}, \cite{MR3012414}, \cite{MR4358612}, \cite{MR4172886}, and \cite{MR4886483}.

Fourth, extending the asymmetric kernel approach presented herein to a nonparametric regression setting, where the explanatory variables are supported on half-spaces and the response is real, presents a promising avenue for future research. Such an extension would provide a natural and efficient way to model functional relationships while addressing boundary bias, which is a persistent issue in regression problems with constrained supports.

Fifth, for both the density estimation and regression settings, the assumptions of independent and identically distributed data could be relaxed to accommodate spatially or temporally correlated samples. A good place to start investigating the theoretical properties would be the book by \citet[Chapters~2--3]{MR1640691}. This step would broaden the applicability of the MIG kernel estimator to structured data problems, such as spatial geostatistics or time series analysis on half-spaces, where boundary constraints still play an important role.

By tackling these open problems, the new MIG kernel density estimator may be further refined and extended to a broader range of statistical applications. Although it already provides a flexible tool for half-space data modeling, further methodological and theoretical developments are needed to fully exploit its potential in high-dimensional or correlated-data contexts.

\begin{appendix}

\section{Proofs of the normal approximations}\label{app:proofs.normal.approximations}

\subsection{Proofs in the univariate case}\label{app:proofs.univariate}

\subsubsection{Proof of Theorem~\ref{thm:LLT}}

Let the positive reals $\mu,\lambda\in (0,\infty)$ and $x\in \smash{B_{\mu,\lambda}(\sqrt{\lambda/\mu})}$ be given. Recall the definition of $\delta_x$ in~\eqref{eq:CLT.univariate}, and observe that
\[
x\in \smash{B_{\mu,\lambda}(\sqrt{\lambda/\mu})} \quad \Rightarrow \quad |\delta_x| \sqrt{\frac{\mu}{\lambda}} < 1.
\]
From the expression of the IG and Gaussian densities in~\eqref{eq:IG.distribution.two.parameters} and \eqref{eq:normal.distribution}, respectively, and using the decomposition
\[
\frac{x}{\mu} = 1 + \delta_x \sqrt{\frac{\mu}{\lambda}},
\]
one has
\[
\begin{aligned}
\mathrm{LR}(x)
&= \frac{3}{2} \ln\left(\frac{\mu}{x}\right) + \frac{(x - \mu)^2}{2 \mu^3/\lambda} \left(1 - \frac{\mu}{x}\right) \\
&= -\frac{3}{2} \ln\left(1 + \delta_x \sqrt{\frac{\mu}{\lambda}}\right) + \frac{1}{2} \delta_x^2 \left\{1 - \left(1 + \delta_x \sqrt{\frac{\mu}{\lambda}}\right)^{-1}\right\}.
\end{aligned}
\]
By applying the following Taylor expansions, valid for $|y| < 1$,
\begin{equation}\label{eq:Taylor.expansions.univariate}
-\ln(1 + y) = \sum_{k=1}^{\infty} \frac{(-1)^k}{k} y^k, \qquad 1 - (1 + y)^{-1} = - \sum_{k=1}^{\infty} (-1)^k y^k,
\end{equation}
one obtains
\[
\begin{aligned}
\mathrm{LR}(x)
&= \frac{3}{2} \sum_{k=1}^{\infty} \frac{(-1)^k}{k} \left(\delta_x \sqrt{\frac{\mu}{\lambda}}\right)^k - \frac{1}{2} \delta_x^2 \sum_{k=1}^{\infty} (-1)^k \left(\delta_x \sqrt{\frac{\mu}{\lambda}}\right)^k \\
&= \sum_{k=1}^{\infty} (-1)^k \left(\frac{3}{2k} - \frac{\delta_x^2}{2}\right) \left(\delta_x \sqrt{\frac{\mu}{\lambda}}\right)^k,
\end{aligned}
\]
which proves~\eqref{eq:LLT.univariate.1}.

\newpage
One obtains~\eqref{eq:LLT.univariate.2} and~\eqref{eq:LLT.univariate.3} using the fact that the errors made when truncating the Taylor expansions in~\eqref{eq:Taylor.expansions.univariate} are uniform on the compact set $|y| \leq \tau \sqrt{\mu/\lambda}$ for any fixed real $\tau\in (0,\sqrt{\lambda/\mu})$.
This concludes the proof.

\subsubsection{Proof of Theorem~\ref{thm:total.variation}}

Given the relationships that exist between the different probability metrics in the statement of the theorem, see, e.g., \citet[p.~421]{doi:10.2307/1403865}, it is sufficient to establish the Hellinger distance bound.

Everywhere below, one writes $B = \overline{B_{\mu,\lambda}}(0.5 \sqrt{\lambda/\mu})$ for simplicity. Let $X\sim \PP_{\mu,\lambda}$. By the bound on the squared Hellinger distance found at the bottom of p.~726 of \cite{MR1922539}, one knows that
\begin{equation}\label{eq:first.bound.Hellinger.distance}
\mathfrak{H}^2(\PP_{\mu,\lambda},\QQ_{\mu,\lambda}) \leq 2 \, \PP(X\in B^{\hspace{0.2mm}\mathrm{c}}) + \EE\left[\ln\left\{\frac{\rd \PP_{\mu,\lambda}}{\rd \QQ_{\mu,\lambda}}(X)\right\} \, \ind_{\{X\in B\}}\right],
\end{equation}
where $B^{\hspace{0.2mm}\mathrm{c}}$ refers to the complement of $B$. Then, by applying Chebyshev's inequality together with the fact that $\Var(\delta_X) = 1$ in~\eqref{eq:CLT.univariate}, one finds that
\begin{equation}\label{eq:concentration.bound}
\PP(X\in B^{\hspace{0.2mm}\mathrm{c}}) = \PP\left(|\delta_X| > 0.5 \sqrt{\lambda/\mu}\right) \leq \frac{4\mu}{\lambda} \, .
\end{equation}
Also, by Theorem~\ref{thm:LLT} and the identity $\ind_{\{X\in B\}} = 1 - \ind_{\{X\in B^{\hspace{0.2mm}\mathrm{c}}\}}$, one has
\begin{equation}\label{eq:estimate.I.begin}
\begin{aligned}
\EE \left[ \ln\left\{\frac{\rd \PP_{\mu,\lambda}}{\rd \QQ_{\mu,\lambda}}(X)\right\} \, \ind_{\{X\in B\}} \right]
&= \sqrt{\frac{\mu}{\lambda}} \, \EE \left\{ \left(\frac{3 \delta_X}{2} - \frac{\delta_X^3}{2}\right) \ind_{\{X\in B^{\hspace{0.2mm}\mathrm{c}}\}} \right\} \\
&\qquad- \sqrt{\frac{\mu}{\lambda}}\, \EE\left(\frac{3 \delta_X}{2} - \frac{\delta_X^3}{2}\right) + \OO(\mu/\lambda).
\end{aligned}
\end{equation}
The second term on the right-hand side of~\eqref{eq:estimate.I.begin} is $\OO(\mu/\lambda)$ because $\EE (\delta_X) = 0$ by~\eqref{eq:esp.var} and one knows that $\EE (\delta_X^3) = 3 \sqrt{\mu/\lambda}$ from Section~15.4 of~\cite{MR1299979}. The first term on the right-hand side of~\eqref{eq:estimate.I.begin} is $\OO(\mu/\lambda)$ by Cauchy--Schwarz, \eqref{eq:concentration.bound}, and the fact that all finite moments of $\delta_X$ are $\OO(1)$ being asymptotically Gaussian. Indeed,
\[
\begin{aligned}
\left[\EE\left\{\left(\frac{3 \delta_X}{2} - \frac{\delta_X^3}{2}\right) \ind_{\{X\in B^{\hspace{0.2mm}\mathrm{c}}\}}\right\}\right]^{\hspace{0.2mm}2}
&\leq \EE\left\{\left(\frac{3 \delta_X}{2} - \frac{\delta_X^3}{2}\right)^2\right\} \times \PP(X\in B^{\hspace{0.2mm}\mathrm{c}}) \\
&\leq \OO(1) \times \frac{4\mu}{\lambda} = \OO(\mu/\lambda).
\end{aligned}
\]
Therefore,
\begin{equation}\label{eq:end}
\EE\left[\ln\left\{\frac{\rd \PP_{\mu,\lambda}}{\rd \QQ_{\mu,\lambda}}(X)\right\} \, \ind_{\{X\in B\}}\right] = \OO(\mu/\lambda).
\end{equation}
Combining the estimates~\eqref{eq:concentration.bound} and~\eqref{eq:end} into~\eqref{eq:first.bound.Hellinger.distance}, one can conclude that the squared Hellinger distance is of order $\OO(\mu/\lambda)$, as claimed.

\subsection{Proofs in the multivariate case}\label{app:proofs.multivariate}

\subsubsection{Proof of Theorem~\ref{thm:LLT.multivariate}}

Let $\mu,\omega\in (0,\infty)$, $\bb{\beta}\in \R^d$, $\bb{\xi}_0\in \mathcal{H}(\bb{\beta})$ and $\bb{x}\in \smash{B_{\mu,\omega,\bb{\beta},\bb{\xi}_0,\mat{\Omega}_0}(\sqrt{\mu/\omega})}$ be given. Recall the definition of $\bb{\delta}_{\bb{x}}$ in~\eqref{eq:CLT.multivariate}, and note that
\[
\bb{x}\in \smash{B_{\mu,\omega,\bb{\beta},\bb{\xi}_0,\mat{\Omega}_0}(\sqrt{\mu/\omega})}
\quad \Rightarrow \quad
\frac{\big|\bb{\beta}^{\top} \mat{\Omega}_0^{1/2} \bb{\delta}_{\bb{x}}\big|}{\sqrt{\bb{\beta}^{\top} \bb{\xi}_0}} \sqrt{\frac{\omega}{\mu}} < 1.
\]
From the expressions of the MIG and multivariate Gaussian densities in \eqref{eq:general.multivariate.density} and \eqref{eq:normal.distribution.multivariate} respectively, and using the decomposition
\[
\frac{\bb{\beta}^{\top} \bb{x}}{\mu \bb{\beta}^{\top} \bb{\xi}_0} = 1 + \frac{\bb{\beta}^{\top} \mat{\Omega}_0^{1/2} \bb{\delta}_{\bb{x}}}{\sqrt{\bb{\beta}^{\top} \bb{\xi}_0}} \sqrt{\frac{\omega}{\mu}},
\]
one finds that
\[
\begin{aligned}
\mathrm{LR}(\bb{x})
&= \left(\frac{d}{2} + 1\right) \ln\left(\frac{\mu \bb{\beta}^{\top} \bb{\xi}_0}{\bb{\beta}^{\top} \bb{x}}\right) \\
&\qquad+ \frac{1}{2} (\bb{x} - \mu\bb{\xi}_0)^{\top} (\mu\omega \bb{\beta}^{\top} \bb{\xi}_0)^{-1}\mat{\Omega}_0^{-1} (\bb{x} - \mu\bb{\xi}_0) \left(1 - \frac{\mu \bb{\beta}^{\top} \bb{\xi}_0}{\bb{\beta}^{\top} \bb{x}}\right) \\
&= -\left(\frac{d}{2} + 1\right) \ln\left(1 + \frac{\bb{\beta}^{\top} \mat{\Omega}_0^{1/2} \bb{\delta}_{\bb{x}}}{\sqrt{\bb{\beta}^{\top} \bb{\xi}_0}} \sqrt{\frac{\omega}{\mu}}\right) \\
&\qquad+ \frac{1}{2} \bb{\delta}_{\bb{x}}^{\top} \bb{\delta}_{\bb{x}} \left\{1 - \left(1 + \frac{\bb{\beta}^{\top} \mat{\Omega}_0^{1/2} \bb{\delta}_{\bb{x}}}{\sqrt{\bb{\beta}^{\top} \bb{\xi}_0}} \sqrt{\frac{\omega}{\mu}}\right)^{-1}\right\}\!.
\end{aligned}
\]
By applying the following Taylor expansions, valid for $|y| < 1$,
\begin{equation}\label{eq:Taylor.expansions.multivariate}
-\ln(1 + y) = \sum_{k=1}^{\infty} \frac{(-1)^k}{k} y^k, \qquad 1 - (1 + y)^{-1} = - \sum_{k=1}^{\infty} (-1)^k y^k,
\end{equation}
one obtains
\[
\begin{aligned}
\mathrm{LR}(\bb{x})
&= \left(\frac{d}{2} + 1\right) \sum_{k=1}^{\infty} \frac{(-1)^k}{k} \left(\frac{\bb{\beta}^{\top} \mat{\Omega}_0^{1/2} \bb{\delta}_{\bb{x}}}{\sqrt{\bb{\beta}^{\top} \bb{\xi}_0}} \sqrt{\frac{\omega}{\mu}}\right)^k \\
&\qquad- \frac{1}{2} \bb{\delta}_{\bb{x}}^{\top} \bb{\delta}_{\bb{x}} \sum_{k=1}^{\infty} (-1)^k \left(\frac{\bb{\beta}^{\top} \mat{\Omega}_0^{1/2} \bb{\delta}_{\bb{x}}}{\sqrt{\bb{\beta}^{\top} \bb{\xi}_0}} \sqrt{\frac{\omega}{\mu}}\right)^k \\
&= \sum_{k=1}^{\infty} (-1)^k \left(\frac{d + 2}{2 k} - \frac{\bb{\delta}_{\bb{x}}^{\top} \bb{\delta}_{\bb{x}}}{2}\right) \left(\frac{\bb{\beta}^{\top} \mat{\Omega}_0^{1/2} \bb{\delta}_{\bb{x}}}{\sqrt{\bb{\beta}^{\top} \bb{\xi}_0}} \sqrt{\frac{\omega}{\mu}}\right)^k,
\end{aligned}
\]
which proves~\eqref{eq:LLT.multivariate.1}.

One obtains~\eqref{eq:LLT.multivariate.2} and~\eqref{eq:LLT.multivariate.3} using the fact that the errors made when truncating the Taylor expansions in~\eqref{eq:Taylor.expansions.multivariate} are uniform on the compact set $|y| \leq \tau \sqrt{\omega/\mu}$ for any fixed real $\tau\in (0, \sqrt{\mu/\omega})$.
This concludes the proof.

\subsubsection{Proof of Theorem~\ref{thm:total.variation.multivariate}}

Given the relationships that exist between the different probability metrics in the statement of the theorem, see, e.g., \citet[p.~421]{doi:10.2307/1403865}, it is sufficient to establish the Hellinger distance bound.

Everywhere below, one writes $B = B_{\mu, \omega, \bb{\beta}, \bb{\xi}_0, \mat{\Omega}_0}(0.5 \sqrt{\mu/\omega})$ for simplicity. Let $\bb{X}\sim \PP_{\mu, \omega, \bb{\beta}, \bb{\xi}_0, \mat{\Omega}_0}$. By the bound on the squared Hellinger distance found at the bottom of p.~726 of~\cite{MR1922539}, one knows that
\begin{equation}\label{eq:first.bound.Hellinger.distance.multivariate}
\begin{aligned}
&\mathfrak{H}^2(\PP_{\mu, \omega, \bb{\beta}, \bb{\xi}_0, \mat{\Omega}_0},\QQ_{\mu, \omega,\bb{\beta}, \bb{\xi}_0, \mat{\Omega}_0}) \\
&\qquad\leq 2 \, \PP(\bb{X}\in B^{\hspace{0.2mm}\mathrm{c}}) + \EE\left[\ln\left\{\frac{\rd \PP_{\mu, \omega, \bb{\beta}, \bb{\xi}_0, \mat{\Omega}_0}}{\rd \QQ_{\mu, \omega, \bb{\beta}, \bb{\xi}_0, \mat{\Omega}_0}}(\bb{X})\right\} \, \ind_{\{\bb{X}\in B\}}\right].
\end{aligned}
\end{equation}

By a union bound on the summands of the quantity $\bb{\beta}^{\top} \mat{\Omega}_0 \bb{\delta}_{\bb{X}}$, followed by an application of Chebyshev's inequality for each summand and the exploitation of the fact that the margins of the MIG distribution are asymptotically Gaussian with mean $\mu \xi_{0,j}$ and variance proportional to $\mu\omega$ by \eqref{eq:CLT.multivariate}, one deduces that
\begin{equation}\label{eq:concentration.bound.multivariate}
\begin{aligned}
&\PP(\bb{X}\in B^{\hspace{0.2mm}\mathrm{c}}) \\
&\quad= \PP\left(\left|\bb{\beta}^{\top} \mat{\Omega}_0 \bb{\delta}_{\bb{X}}\right| > 0.5 \sqrt{\frac{\mu}{\omega}}\sqrt{\bb{\beta}^{\top} \bb{\xi}_0}\right) \\
&\quad\leq \sum_{i=1}^d \PP\left\{\left|(\bb{\beta}^{\top} \mat{\Omega}_0)_i (\bb{\delta}_{\bb{X}})_i\right| > \frac{0.5}{d} \sqrt{\frac{\mu}{\omega}}\sqrt{\bb{\beta}^{\top} \bb{\xi}_0}\right\} \\
&\quad\leq \sum_{i=1}^d \sum_{j=1}^d \PP\left\{\left|(\bb{\beta}^{\top} \mat{\Omega}_0)_i (\mat{\Omega}_0^{-1/2})_{ij} \frac{(X_j - \mu \bb{\xi}_{0,j})}{\sqrt{\mu\omega \bb{\beta}^{\top} \bb{\xi}_0}}\right| > \frac{0.5}{d^{\hspace{0.2mm}2}} \sqrt{\frac{\mu}{\omega}}\sqrt{\bb{\beta}^{\top} \bb{\xi}_0}\right\} \\
&\quad= \OO_{\bb{\beta},\bb{\xi}_0,\mat{\Omega}_0}(\omega/\mu).
\end{aligned}
\end{equation}
Also, by Theorem~\ref{thm:LLT.multivariate} and the identity $\ind_{\{\bb{X}\in B\}} = 1 - \ind_{\{\bb{X}\in B^{\hspace{0.2mm}\mathrm{c}}\}}$, one has
\begin{equation}\label{eq:estimate.I.begin.multivariate}
\begin{aligned}
&\EE\left[\ln\left\{\frac{\rd \PP_{\mu, \omega,\bb{\beta}, \bb{\xi}_0, \mat{\Omega}_0}}{\rd \QQ_{\mu,\omega, \bb{\beta}, \bb{\xi}_0, \mat{\Omega}_0}}(\bb{X})\right\} \, \ind_{\{\bb{X}\in B\}}\right] \\[1mm]
&\qquad= \sqrt{\frac{\omega}{\mu}} \frac{\bb{\beta}^{\top} \mat{\Omega}_0^{1/2}}{\sqrt{\bb{\beta}^{\top} \bb{\xi}_0}} \EE\left[\left\{\frac{(d + 2) \bb{\delta}_{\bb{X}}}{2} - \frac{\bb{\delta}_{\bb{X}} \bb{\delta}_{\bb{X}}^{\top} \bb{\delta}_{\bb{X}}}{2}\right\} \ind_{\{\bb{X}\in B^{\hspace{0.2mm}\mathrm{c}}\}}\right] \\
&\qquad\qquad- \sqrt{\frac{\omega}{\mu}} \frac{\bb{\beta}^{\top} \mat{\Omega}_0^{1/2}}{\sqrt{\bb{\beta}^{\top} \bb{\xi}_0}} \EE\left[\left\{\frac{(d + 2) \bb{\delta}_{\bb{X}}}{2} - \frac{\bb{\delta}_{\bb{X}} \bb{\delta}_{\bb{X}}^{\top} \bb{\delta}_{\bb{X}}}{2}\right\}\right] \\
&\qquad\qquad+ \OO_{\bb{\beta},\bb{\xi}_0,\mat{\Omega}_0}(\omega/\mu).
\end{aligned}
\end{equation}

The second term on the right-hand side of~\eqref{eq:estimate.I.begin.multivariate} is $\OO_{\bb{\beta},\bb{\xi}_0,\mat{\Omega}_0}(\omega/\mu)$ because $\EE (\bb{\delta}_{\bb{X}}) = \bb{0}_d$ by~\eqref{eq:esp.var.multivariate}, and using the fact that $\EE (\delta_{x,i} \delta_{x,j} \delta_{x,k}) = \smash{\OO_{\bb{\beta},\bb{\xi}_0,\mat{\Omega}}(\sqrt{\omega/\mu})}$ for all $i,j,k\in \{1,\ldots,d\}$, which can be deduced from Property~1 of~\cite{MR2019130}. The first term on the right-hand side of~\eqref{eq:estimate.I.begin.multivariate} is $\OO_{\bb{\beta},\bb{\xi}_0,\mat{\Omega}_0}(\omega/\mu)$ by Cauchy--Schwarz, Eq.~\eqref{eq:concentration.bound.multivariate} and the fact that all finite joint moments of the components of $\bb{\delta}_{\bb{X}}$ are $\OO_{\bb{\beta},\bb{\xi}_0,\mat{\Omega}_0}(1)$ being asymptotically Gaussian. Indeed,
\[
\begin{aligned}
&\left(\EE\left[\left\{\frac{(d + 2) \bb{\delta}_{\bb{X}}}{2} - \frac{\bb{\delta}_{\bb{X}} \bb{\delta}_{\bb{X}}^{\top} \bb{\delta}_{\bb{X}}}{2}\right\} \ind_{\{\bb{X}\in B^{\hspace{0.2mm}\mathrm{c}}\}}\right]\right)^2 \\
&\qquad\qquad\leq \EE\left[\left\{\frac{(d + 2) \bb{\delta}_{\bb{X}}}{2} - \frac{\bb{\delta}_{\bb{X}} \bb{\delta}_{\bb{X}}^{\top} \bb{\delta}_{\bb{X}}}{2}\right\}^2\right] \times \PP\left(\bb{X}\in B^{\hspace{0.2mm}\mathrm{c}}\right) \\
&\qquad\qquad= \OO_{\bb{\beta},\bb{\xi}_0,\mat{\Omega}_0}(1) \times \OO_{\bb{\beta},\bb{\xi}_0,\mat{\Omega}_0}(\omega/\mu) \\[1mm]
&\qquad\qquad= \OO_{\bb{\beta},\bb{\xi}_0,\mat{\Omega}_0}(\omega/\mu).
\end{aligned}
\]
Therefore,
\begin{equation}\label{eq:end.multivariate}
\EE\left[\ln\left\{\frac{\rd \PP_{\mu, \omega,\bb{\beta}, \bb{\xi}_0, \mat{\Omega}_0}}{\rd \QQ_{\mu,\omega, \bb{\beta}, \bb{\xi}_0, \mat{\Omega}_0}}(\bb{X})\right\} \, \ind_{\{\bb{X}\in B\}}\right] = \OO_{\bb{\beta},\bb{\xi}_0,\mat{\Omega}_0}(\omega/\mu).
\end{equation}
Combining the estimates from~\eqref{eq:concentration.bound.multivariate} and~\eqref{eq:end.multivariate} into~\eqref{eq:first.bound.Hellinger.distance.multivariate}, one can conclude that the squared Hellinger distance is of order $\OO_{\bb{\beta},\bb{\xi}_0,\mat{\Omega}_0}(\omega/\mu)$, as claimed.

\section{Proofs of the asymptotics for the MIG kernel density estimator}\label{app:proofs.application.asymptotics}

\subsection{Proof of Proposition~\ref{prop:asymptotic.bias}}

For any vectors $\bb{\beta}\in \R^d$, $\bb{\xi}\in \mathcal{H}(\bb{\beta})$, and any positive definite matrix $\mat{H}\in \mathcal{S}_{++}^d$, define the random vector
\[
\bb{Y}_{\!\bb{\xi}} = (Y_1,\ldots,Y_d) \sim \mathrm{MIG}(\bb{\beta},\bb{\xi},\mat{H}).
\]
As in \eqref{eq:esp.var.multivariate}, note that, for all $i,j\in \{1,\ldots,d\}$,
\begin{equation}\label{eq:expectation.covariance.explicit.estimate}
\EE(Y_i) = \xi_i, \qquad \EE\{(Y_i - \xi_i) (Y_j - \xi_j)\} = \Cov(Y_i,Y_j) = \bb{\beta}^{\top} \bb{\xi} \, \mat{H}_{ij}.
\end{equation}
By a second order stochastic mean value theorem \citep[Theorem~18.18]{MR2378491}, one has
\begin{equation} \label{eq:second.order.MVT}
\begin{aligned}
&f(\bb{Y}_{\!\bb{\xi}}) - f(\bb{\xi}) \\
&\quad= \sum_{i=1}^d (Y_i - \xi_i) \frac{\partial}{\partial \xi_i} f(\bb{\xi}) + \frac{1}{2} \sum_{i,j=1}^d (Y_i - \xi_i) (Y_j - \xi_j) \frac{\partial^2}{\partial \xi_i \partial \xi_j} f(\bb{\xi}) \\
&\qquad+ \frac{1}{2} \sum_{i,j=1}^d (Y_i - \xi_i) (Y_j - \xi_j) \left\{\frac{\partial^2}{\partial \xi_i \partial \xi_j} f(\bb{\zeta}_{\bb{\xi}}) - \frac{\partial^2}{\partial \xi_i \partial \xi_j} f(\bb{\xi})\right\},
\end{aligned}
\end{equation}
for some random vector $\bb{\zeta}_{\bb{\xi}}\in \mathcal{H}(\bb{\beta})$ on the line segment joining $\bb{Y}_{\!\bb{\xi}}$ and $\bb{\xi}$. Given that it was assumed that the second order partial derivatives of $f$ are uniformly continuous on the half-space $\mathcal{H}(\bb{\beta})$, then for any given $\e\in (0,\infty)$, there exists a real number $\delta_{d,\bb{\beta},\e}\in (0,1]$ such that, uniformly for $\bb{\xi},\bb{\xi}'\in \mathcal{H}(\bb{\beta})$,
\[
\|\bb{\xi}' - \bb{\xi}\|_1 \leq \delta_{d,\bb{\beta},\e} \quad \Rightarrow \quad \left|\frac{\partial^2}{\partial \xi_i \partial \xi_j} f(\bb{\xi}') - \frac{\partial^2}{\partial \xi_i \partial \xi_j} f(\bb{\xi})\right| < \e.
\]
Therefore, by fixing $\e\in (0,\infty)$, taking the expectation on both sides of \eqref{eq:second.order.MVT}, and then using~\eqref{eq:expectation.covariance.explicit.estimate}, one can bound the recentered bias as follows:
\begin{align}\label{eq:prop.Berntein.decomposition}
&\Bigg|\Bias\big\{\hat{f}_{n,\mat{H}}(\bb{\xi})\big\} - \frac{1}{2} \bb{\beta}^{\top} \bb{\xi} \sum_{i,j=1}^d \mat{H}_{ij} \frac{\partial^2}{\partial \xi_i \partial \xi_j} f(\bb{\xi})\Bigg| \notag \\
&= \Bigg|\frac{1}{2} \sum_{i,j=1}^d \EE\{(Y_i - \xi_i) (Y_j - \xi_j)\} \left\{\frac{\partial^2}{\partial \xi_i \partial \xi_j} f(\bb{\zeta}_{\bb{\xi}}) - \frac{\partial^2}{\partial \xi_i \partial \xi_j} f(\bb{\xi})\right\}\Bigg| \notag \\
&\leq \frac{1}{2} \sum_{i,j=1}^d \EE\left[|Y_i - \xi_i| |Y_j - \xi_j| \, \Big|\tfrac{\partial^2}{\partial \xi_i \partial \xi_j} f(\bb{\zeta}_{\bb{\xi}}) - \tfrac{\partial^2}{\partial \xi_i \partial \xi_j} f(\bb{\xi})\Big| \ind_{\{\|\bb{Y}_{\!\bb{\xi}} - \bb{\xi}\|_1 \leq \delta_{d,\bb{\beta},\e}\}}\right] \notag \\
&\quad+ \frac{1}{2} \sum_{i,j=1}^d \EE\left[|Y_i - \xi_i| |Y_j - \xi_j| \, \Big|\tfrac{\partial^2}{\partial \xi_i \partial \xi_j} f(\bb{\zeta}_{\bb{\xi}}) - \tfrac{\partial^2}{\partial \xi_i \partial \xi_j} f(\bb{\xi})\Big| \ind_{\{\|\bb{Y}_{\!\bb{\xi}} - \bb{\xi}\|_1 > \delta_{d,\bb{\beta},\e}\}}\right] \notag \\
&\equiv \Delta_1 + \Delta_2.
\end{align}

If one applies the uniform continuity of the second order partial derivatives together with the fact that $\|\bb{\zeta}_{\bb{\xi}} - \bb{\xi}\|_1 \leq \|\bb{Y}_{\!\bb{\xi}} - \bb{\xi}\|_1$, followed by the Cauchy--Schwarz inequality, then~\eqref{eq:expectation.covariance.explicit.estimate}, and then Jensen's inequality, one gets
\begin{equation*} 
\begin{aligned}
\Delta_1
&\leq \frac{\e}{2} \sum_{i,j=1}^d \sqrt{\EE(|Y_i - \xi_i|^2)} \sqrt{\EE(|Y_j - \xi_j|^2)} = \frac{\e}{2} \bb{\beta}^{\top} \bb{\xi} \left\{\sum_{i=1}^d \sqrt{\mat{H}_{ii}}\right\}^2 \\
&\leq \frac{\e}{2} \bb{\beta}^{\top} \bb{\xi} \left\{d \sum_{i=1}^d \mat{H}_{ii}\right\} \leq \frac{\e \, d^{\hspace{0.2mm}2}}{2} \bb{\beta}^{\top} \bb{\xi} \, \|\mat{H}\|_2.
\end{aligned}
\end{equation*}

Next, the second order partial derivatives were also assumed to be bounded, say by $M_d\in (0,\infty)$.
Therefore, applying the bound $M_d$ on the two second order partial derivatives in $\Delta_2$, followed by a moment version of Chernoff's inequality, and then using the elementary inequality $x \leq \exp(x)$, one finds
\[
\begin{aligned}
\Delta_2
&\leq \frac{2 M_d}{2} \, \EE\left(\|\bb{Y}_{\!\bb{\xi}} - \bb{\xi}\|_1^2 \, \ind_{\{\|\bb{Y}_{\!\bb{\xi}} - \bb{\xi}\|_1 > \delta_{d,\bb{\beta},\e}\}}\right) \\
&\leq M_d \, \inf_{s\in (0,\infty)} \exp(-s \delta_{d,\bb{\beta},\e}) \EE\left\{\|\bb{Y}_{\!\bb{\xi}} - \bb{\xi}\|_1^2 \exp(s \|\bb{Y}_{\!\bb{\xi}} - \bb{\xi}\|_1)\right\} \\
&\leq M_d \, \inf_{s\in (0,\infty)} \exp(-s \delta_{d,\bb{\beta},\e}) \EE\left[\exp\{(2 + s) \|\bb{Y}_{\!\bb{\xi}} - \bb{\xi}\|_1\}\right] \\
&\leq e^2 \, M_d \, \inf_{t\in (2,\infty)} \exp(-t \delta_{d,\bb{\beta},\e}) \EE\left[\exp\{t \|\bb{Y}_{\!\bb{\xi}} - \bb{\xi}\|_1\}\right].
\end{aligned}
\]
Let $\bb{Z}\sim \mathcal{N}_d(\bb{0}_d,\bb{\beta}^{\top}\bb{\xi} \, \mat{H})$. By the asymptotic normality in \eqref{eq:CLT.multivariate} and the continuous mapping theorem, one obtains
\[
\Delta_2 = \OO\left[\inf_{t\in (2,\infty)} \exp(-t \delta_{d,\bb{\beta},\e}) \EE\left\{\exp(t \|\bb{Z}\|_1)\right\}\right].
\]
By optimizing this Chernoff bound, one has, say, $\Delta_2 = \OO(\|\mat{H}\|_2^{100})$ provided that the sequences $\e = \e(\mat{H})\in (0,\infty)$ and $\delta_{d,\bb{\beta},\e} = \delta_{d,\bb{\beta},\e}(\mat{H})\in (0,1]$ are chosen to converge slowly enough to $0$ as $\|\mat{H}\|_2\to 0$.

It follows that $\Delta_1 + \Delta_2$ in~\eqref{eq:prop.Berntein.decomposition} is $\oo(\bb{\beta}^{\top}\bb{\xi} \|\mat{H}\|_2)$. This concludes the proof.

\subsection{Proof of Proposition~\ref{prop:asymptotic.variance}}

Given that the target density $f$ is assumed to be bounded, the supremum norm $\|f\|_{\infty}$ is finite. Simple calculations show that
\[
\begin{aligned}
\Var\big\{\hat{f}_{n,\mat{H}}(\bb{\xi})\big\}
&= n^{-1} \, \EE\big\{k_{\bb{\beta},\bb{\xi}, \mat{H}}(\bb{X})^2\big\} - n^{-1} \left[\EE\big\{k_{\bb{\beta},\bb{\xi}, \mat{H}}(\bb{X})\big\}\right]^{\hspace{0.2mm}2} \\
&= n^{-1} \, \EE\big\{k_{\bb{\beta},\bb{\xi}, \mat{H}}(\bb{X})^2\big\} - \OO(n^{-1} \|f\|_{\infty}^2).
\end{aligned}
\]
One deduces from the multivariate normal local approximation in Theorem~\ref{thm:LLT.multivariate} that
\[
\begin{aligned}
&\EE\big\{k_{\bb{\beta},\bb{\xi}, \mat{H}}(\bb{X})^2\big\} \\
&\quad= \int_{\mathcal{H}(\bb{\beta})} \left[\frac{\exp\big\{-\frac{1}{2} (\bb{x} - \bb{\xi})^{\top} (\bb{\beta}^{\top} \bb{\xi} \, \mat{H})^{-1} (\bb{x} - \bb{\xi})\big\}}{(2\pi)^{d/2} |\bb{\beta}^{\top} \bb{\xi} \, \mat{H}|^{1/2}}\right]^{\hspace{0.2mm}2} f(\bb{x}) \, \rd \bb{x} + \oo_{\bb{\beta},\bb{\xi}}(1) \\
&\quad= \frac{\{f(\bb{\xi}) + \OO_{\bb{\beta},\bb{\xi}}(\|\mat{H}\|_2)\}}{(2\pi)^{d/2} |2 \bb{\beta}^{\top} \bb{\xi} \, \mat{H}|^{1/2}} \int_{\mathcal{H}(\bb{\beta})} \frac{\exp\big\{-\frac{1}{2} (\bb{x} - \bb{\xi})^{\top} (\tfrac{1}{2} \bb{\beta}^{\top} \bb{\xi} \, \mat{H})^{-1} (\bb{x} - \bb{\xi})\big\}}{(2\pi)^{d/2} |\tfrac{1}{2} \bb{\beta}^{\top} \bb{\xi} \, \mat{H}|^{1/2}} \rd \bb{x} \\
&\quad\qquad+ \oo_{\bb{\beta},\bb{\xi}}(1) \\[0.5mm]
&\quad= \frac{\{f(\bb{\xi}) + \OO_{\bb{\beta},\bb{\xi}}(\|\mat{H}\|_2)\}}{(4\pi \bb{\beta}^{\top} \bb{\xi})^{d/2} |\mat{H}|^{1/2}} \{1 + \oo_{\bb{\beta},\bb{\xi}}(1)\} + \oo_{\bb{\beta},\bb{\xi}}(1).
\end{aligned}
\]
The conclusion follows.

\subsection{Proof of Theorem~\ref{thm:asymp.normality}}
\label{sec:B3}

First note that
\begin{equation*} 
\hat{f}_{n,\mat{H}}(\bb{\xi}) - \EE\big\{\hat{f}_{n,\mat{H}}(\bb{\xi})\big\} = \frac{1}{n} \sum_{i=1}^n Z_{i,\mat{H}}(\bb{\xi}),
\end{equation*}
where
\[
Z_{i,\mat{H}}(\bb{\xi}) = k_{\bb{\beta},\bb{\xi},\mat{H}}(\bb{X}_i) - \EE\big\{k_{\bb{\beta},\bb{\xi},\mat{H}}(\bb{X}_i)\big\}.
\]
Like the observations $\bb{X}_1,\ldots,\bb{X}_n$, the random variables $Z_{1,\mat{H}}(\bb{\xi}),\ldots,Z_{n,\mat{H}}(\bb{\xi})$ are independent and identically distributed.

Calling on the material in, e.g., Section 1.9.3 of~\cite{MR595165}, the asymptotic normality of $n^{1/2} |\mat{H}|^{1/4} [\hat{f}_{n,\mat{H}}(\bb{\xi}) - \EE\{\hat{f}_{n,\mat{H}}(\bb{\xi})\}]$ will be proved if one verifies the following Lindeberg condition for double arrays: for every real $\e\in (0,\infty)$,
\begin{equation}\label{eq:pre.asymp.normality.Lindeberg.condition}
\lim_{n\to \infty} \frac{1}{s_{\mat{H}}^2} \EE\left[|Z_{1,\mat{H}}(\bb{\xi})|^2 \, \ind_{\{|Z_{1,\mat{H}}(\bb{\xi})| > \e n^{1/2} s_{\mat{H}}\}}\right] = 0,
\end{equation}
where $s_{\mat{H}}^2 = \EE\{|Z_{1,\mat{H}}(\bb{\xi})|^2\}$ and $\|\mat{H}\|_2\to 0$ as $n\to \infty$.

The condition~\eqref{eq:pre.asymp.normality.Lindeberg.condition} is verified below.
From the uniform bound on $k_{\bb{\beta},\bb{\xi},\mat{H}}$ in Lemma~\ref{lem:uniform.bound}, it is known that
\[
|Z_{1,\mat{H}}(\bb{\xi})| = \OO_{\bb{\beta},\bb{\xi}}\big(|\mat{H}|^{-1/2}\big).
\]
Furthermore, by the asymptotics of the pointwise variance in Proposition~\ref{prop:asymptotic.variance} under the assumption that $f$ is Lipschitz continuous and bounded on the half-space $\mathcal{H}(\bb{\beta})$, one has
\[
s_{\mat{H}} = |\mat{H}|^{-1/4} \sqrt{\frac{f(\bb{\xi})}{(4\pi \bb{\beta}^{\top} \bb{\xi})^{d/2}}} \, \{1 + \oo_{\bb{\beta},\bb{\xi}}(1)\}.
\]
Therefore, by combining the last two equations, one gets
\begin{equation}\label{eq:pre.asymp.normality.Lindeberg.condition.verify}
\frac{|Z_{1,\mat{H}}(\bb{\xi})|}{n^{1/2} s_{\mat{H}}} = \OO_{\bb{\beta},\bb{\xi}}\big(n^{-1/2} |\mat{H}|^{-1/4}\big) \to 0,
\end{equation}
whenever $n^{1/2} |\mat{H}|^{1/4}\to \infty$ as $n\to \infty$. The convergence to zero in~\eqref{eq:pre.asymp.normality.Lindeberg.condition.verify} shows that~\eqref{eq:pre.asymp.normality.Lindeberg.condition} holds because $\e\in (0,\infty)$ being fixed implies that the indicator function in~\eqref{eq:pre.asymp.normality.Lindeberg.condition} is equal to zero for $n$ large enough, independently of $\omega$.

Hence, by Proposition~\ref{prop:asymptotic.variance} again, one deduces that
\[
\begin{aligned}
&n^{1/2} |\mat{H}|^{1/4} \left[\hat{f}_{n,\mat{H}}(\bb{\xi}) - \EE\big\{\hat{f}_{n,\mat{H}}(\bb{\xi})\big\}\right] \\
&\qquad= n^{1/2} |\mat{H}|^{1/4} \frac{1}{n} \sum_{i=1}^n Z_{i,\mat{H}}(\bb{\xi}) \rightsquigarrow \mathcal{N}_d \left[ 0,\frac{f(\bb{\xi})}{(4\pi \bb{\beta}^{\top} \bb{\xi})^{d/2}} \right].
\end{aligned}
\]
This concludes the proof.

\section{Technical lemmas and derivations}\label{app:technical.lemmas}

\subsection{Uniform bound on the MIG density}\label{app:uniform.bound.MIG.density}

A uniform bound on the MIG density in \eqref{eq:general.multivariate.density} is stated and proved below.

\begin{lemma}\label{lem:uniform.bound}
For any vectors $\bb{\beta}\in \R^d$ and $\bb{\xi}\in \mathcal{H}(\bb{\beta})$, and any positive definite matrix $\mat{\Omega}\in \mathcal{S}_{++}^d$, one has
\[
\begin{aligned}
\sup_{\bb{x}\in \mathcal{H}(\bb{\beta})} k_{\bb{\beta},\bb{\xi},\mat{\Omega}}(\bb{x})
\leq \frac{\bb{\beta}^{\top} \bb{\xi}}{(2\pi)^{d/2} \, |\mat{\Omega}|^{1/2}}
\times \max\left\{\frac{2}{\bb{\beta}^{\top} \bb{\xi}},\frac{(d/2+1)}{e} \frac{8 \|\mat{\Omega}\|_2 \|\bb{\beta}\|_2^2}{(\bb{\beta}^{\top} \bb{\xi})^2}\right\}^{d/2 + 1}.
\end{aligned}
\]
In particular, if one assumes that $\|\mat{\Omega}\|_2\to 0$, then one has
\[
\sup_{\bb{x}\in \mathcal{H}(\bb{\beta})} k_{\bb{\beta},\bb{\xi},\mat{\Omega}}(\bb{x}) = \OO_{\bb{\beta},\bb{\xi}}\big(|\mat{\Omega}|^{-1/2}\big).
\]
\end{lemma}

\begin{proof}[Proof of Lemma~\ref{lem:uniform.bound}]
There are two cases to consider, namely
\[
\mathrm{(i)} \quad \|\bb{x} - \bb{\xi}\|_2 < \frac{\bb{\beta}^{\top} \bb{\xi}}{2 \|\bb{\beta}\|_2}, \qquad \mathrm{(ii)} \quad \|\bb{x} - \bb{\xi}\|_2 \geq \frac{\bb{\beta}^{\top} \bb{\xi}}{2 \|\bb{\beta}\|_2}.
\]

In the first case, an application of the Cauchy--Schwarz inequality yields
\[
|\bb{\beta}^{\top} \bb{x} - \bb{\beta}^{\top} \bb{\xi}| = |\bb{\beta}^{\top} (\bb{x} - \bb{\xi})| \leq \|\bb{\beta}\|_2 \|\bb{x} - \bb{\xi}\|_2 < \frac{\bb{\beta}^{\top} \bb{\xi}}{2}.
\]
From this, one deduces that $\bb{\beta}^{\top} \bb{x} \geq \bb{\beta}^{\top} \bb{\xi}/2$, or equivalently $(\bb{\beta}^{\top} \bb{x})^{-1} \!\leq 2/ (\bb{\beta}^{\top} \bb{\xi})$.
Therefore, if the exponential is bounded above by $1$ in $k_{\bb{\beta},\bb{\xi},\mat{\Omega}}(\bb{x})$, one gets
\begin{equation}\label{eq:lem:uniform.bound.1}
k_{\bb{\beta},\bb{\xi},\mat{\Omega}}(\bb{x}) \leq \frac{\bb{\beta}^{\top} \bb{\xi}}{(2\pi)^{d/2} \, |\mat{\Omega}|^{1/2}} \left(\frac{2}{\bb{\beta}^{\top} \bb{\xi}}\right)^{d/2+1}.
\end{equation}

In the second case, one initially observes that the eigenvalues of $\mat{\Omega}^{-1}$ are the inverses of the eigenvalues of $\mat{\Omega}$. Given that $\|\mat{\Omega}\|_2$ corresponds to the largest eigenvalue of $\mat{\Omega}$, $\|\mat{\Omega}\|_2^{-1}$ corresponds to the smallest eigenvalue of $\mat{\Omega}^{-1}$. One deduces the lower bound
\[
(\bb{x} - \bb{\xi})^{\top} \mat{\Omega}^{-1} (\bb{x} - \bb{\xi})
\geq \|\mat{\Omega}\|_2^{-1} \|\bb{x} - \bb{\xi}\|_2^2
\geq \|\mat{\Omega}\|_2^{-1} \frac{(\bb{\beta}^{\top} \bb{\xi})^2}{4 \|\bb{\beta}\|_2^2}.
\]
Hence, with the notation $y = (\bb{\beta}^{\top} \bb{x})^{-1}$, one has
\[
\begin{aligned}
k_{\bb{\beta},\bb{\xi},\mat{\Omega}}(\bb{x})
&= \frac{\bb{\beta}^{\top} \bb{\xi}}{(2\pi)^{d/2} \, |\mat{\Omega}|^{1/2}} \, y^{d/2 + 1} \exp\left\{- \frac{y}{2} (\bb{x} - \bb{\xi})^{\top} \mat{\Omega}^{-1} (\bb{x} - \bb{\xi})\right\} \\
&\leq \frac{\bb{\beta}^{\top} \bb{\xi}}{(2\pi)^{d/2} \, |\mat{\Omega}|^{1/2}} \, y^{d/2 + 1} \exp\left\{- y \frac{(\bb{\beta}^{\top} \bb{\xi})^2}{8 \|\mat{\Omega}\|_2 \|\bb{\beta}\|_2^2}\right\}.
\end{aligned}
\]
For any positive constant $c\in (0,\infty)$, straightforward calculations show that the function $y\mapsto y^{d/2+1} \exp(-y/c)$ is maximized at the point $y = (d/2+1) c$.
Therefore, one finds
\begin{equation}\label{eq:lem:uniform.bound.2}
k_{\bb{\beta},\bb{\xi},\mat{\Omega}}(\bb{x}) \leq \frac{\bb{\beta}^{\top} \bb{\xi}}{(2\pi)^{d/2} \, |\mat{\Omega}|^{1/2}} \left\{\frac{(d/2+1)}{e} \frac{8 \|\mat{\Omega}\|_2 \|\bb{\beta}\|_2^2}{(\bb{\beta}^{\top} \bb{\xi})^2}\right\}^{d/2 + 1}.
\end{equation}
The conclusion then follows from~\eqref{eq:lem:uniform.bound.1} and~\eqref{eq:lem:uniform.bound.2}.
\end{proof}

\subsection{Derivation of the maximum likelihood estimator}\label{app:MLE}

\begin{proof}[Proof of Proposition~\ref{prop:MLE}]
Given the observations $\bb{x}_1,\ldots,\bb{x}_n$, the $\mathrm{MIG}(\bb{\beta},\bb{\xi},\mat{\Omega})$ log-likelihood with respect to the pair $(\bb{\xi},\mat{\Omega})$ is
\[
\begin{aligned}
\ell(\bb{\xi},\mat{\Omega})
&= \sum_{i=1}^n \ln\{k_{\bb{\beta},\bb{\xi},\mat{\Omega}}(\bb{x}_i)\} \\
&= -\frac{nd}{2} \ln(2\pi) + n \ln(\bb{\beta}^{\top} \bb{\xi}) + \frac{n}{2} \ln|\mat{\Omega}^{-1}| - (d/2 + 1) \sum_{i=1}^n \ln(\bb{\beta}^{\top} \bb{x}_i) \\
&\qquad- \frac{1}{2} \sum_{i=1}^n \frac{1}{\bb{\beta}^{\top} \bb{x}_i} (\bb{x}_i - \bb{\xi})^{\top} \mat{\Omega}^{-1} (\bb{x}_i - \bb{\xi}).
\end{aligned}
\]
First, consider the partial derivative with respect to $\mat{\Omega}^{-1}$. Given that
\[
\frac{\partial}{\partial \mat{\Omega}^{-1}} |\mat{\Omega}^{-1}| = |\mat{\Omega}^{-1}| (\mat{\Omega}^{-1})^{-1} = |\mat{\Omega}^{-1}| \mat{\Omega},
\]
and
\[
\frac{\partial}{\partial \mat{\Omega}^{-1}} (\bb{x}_i - \bb{\xi})^{\top} \mat{\Omega}^{-1} (\bb{x}_i - \bb{\xi}) = (\bb{x}_i - \bb{\xi}) (\bb{x}_i - \bb{\xi})^{\top},
\]
one has
\[
\begin{aligned}
\frac{\partial}{\partial \mat{\Omega}^{-1}} \, \ell(\bb{\xi},\mat{\Omega})
&= \frac{n}{2} \, \mat{\Omega} - \frac{1}{2} \sum_{i=1}^n \frac{1}{\bb{\beta}^{\top} \bb{x}_i} (\bb{x}_i - \bb{\xi})(\bb{x}_i - \bb{\xi})^{\top}
= \mat{0}_{d\times d} \\
& \Leftrightarrow \quad
\mat{\Omega}
=\frac{1}{n}\sum_{i=1}^n \frac{1}{\bb{\beta}^{\top} \bb{x}_i}(\bb{x}_i - \bb{\xi})(\bb{x}_i - \bb{\xi})^{\top}.
\end{aligned}
\]
Alternatively, one can write the expression for $\mat{\Omega}$ as follows:
\[
\mat{\Omega} = \frac{1}{n}\sum_{i=1}^n \frac{\bb{x}_i \bb{x}_i^{\top}}{\bb{\beta}^{\top} \bb{x}_i}-\frac{1}{n}\sum_{i=1}^n \frac{\bb{x}_i\bb{\xi}^{\top}}{\bb{\beta}^{\top} \bb{x}_i} -\frac{1}{n}\sum_{i=1}^n \frac{\bb{\xi}\bb{x}_i^{\top}}{\bb{\beta}^{\top} \bb{x}_i} + \frac{1}{n}\sum_{i=1}^n \frac{\bb{\xi}\bb{\xi}^{\top}}{\bb{\beta}^{\top} \bb{x}_i}.
\]

Consider now the partial derivative with respect to $\bb{\xi}$. Given that
\[
\frac{\partial}{\partial \bb{\xi}}\bb{\beta}^{\top} \bb{\xi}=\bb{\beta},
\]
and
\[
\frac{\partial}{\partial \bb{\xi}}(\bb{x}_i - \bb{\xi})^{\top} \mat{\Omega}^{-1} (\bb{x}_i - \bb{\xi})=-2\mat{\Omega}^{-1} (\bb{x}_i - \bb{\xi}),
\]
one has
\[
\frac{\partial}{\partial \bb{\xi}}\ell(\bb{\xi},\mat{\Omega}) = \bb{0}_d \quad \Leftrightarrow \quad \frac{n\bb{\beta}}{\bb{\beta}^{\top} \bb{\xi}} +\mat{\Omega}^{-1}\sum_{i=1}^n \frac{1}{\bb{\beta}^{\top} \bb{x}_i} (\bb{x}_i - \bb{\xi}) = \bb{0}_d.
\]
Now the right-hand identity holds if and only if
\[
\frac{n\mat{\Omega}\bb{\beta}}{\bb{\beta}^{\top} \bb{\xi}} + \sum_{i=1}^n \frac{1}{\bb{\beta}^{\top} \bb{x}_i} (\bb{x}_i - \bb{\xi}) = \bb{0}_d
\]
or, equivalently,
\[
\sum_{i=1}^n \frac{\bb{x}_i}{\bb{\beta}^{\top} \bb{\xi}}-\sum_{i=1}^n \frac{\bb{x}_i}{\bb{\beta}^{\top} \bb{x}_i}-\sum_{i=1}^n \frac{\bb{\xi}}{\bb{\beta}^{\top} \bb{\xi}} + \sum_{i=1}^n \frac{\bb{\xi}}{\bb{\beta}^{\top} \bb{x}_i}
+\sum_{i=1}^n \frac{\bb{x}_i}{\bb{\beta}^{\top} \bb{x}_i} -\sum_{i=1}^n \frac{\bb{\xi}}{\bb{\beta}^{\top} \bb{x}_i} = \bb{0}_d.
\]
Given that
\[
\frac{1}{\bb{\beta}^{\top} \bb{\xi}}\sum_{i=1}^n \bb{x}_i-\frac{n\bb{\xi}}{\bb{\beta}^{\top} \bb{\xi}} = \bb{0}_d \quad
\Leftrightarrow \quad \bb{\xi} = \frac{1}{n}\sum_{i=1}^n \bb{x}_i,
\]
it follows that the maximum likelihood estimator for the pair $(\bb{\xi},\mat{\Omega})$ is $(\bb{\xi}_n^{\star},\mat{\Omega}_n^{\star})$, where
\[
\bb{\xi}_n^{\star} = \frac{1}{n}\sum_{i=1}^n \bb{X}_i,
\quad
\mat{\Omega}_n^{\star} = \frac{1}{n} \sum_{i=1}^n \frac{1}{\bb{\beta}^{\top} \bb{X}_i}(\bb{X}_i - \bar{\bb{X}}_n)(\bb{X}_i - \bar{\bb{X}}_n)^{\top}.
\]
This concludes the proof.
\end{proof}

\subsection{Hessian of the MIG density}\label{app:gradient.Hessian}
As shown in Corollary~\ref{cor:MISE.optimal.density}, the optimal bandwidth matrix with respect to the MISE depends on the Hessian of the target density. Therefore, in the simulation study (Section~\ref{subsec:simulation.bandwidth.selection}), the Hessian of the MIG density~\eqref{eq:general.multivariate.density} is needed to implement the bandwidth selection methods $A$ and $B$, which are based on AMISE minimization with an MIG plug-in. An explicit formula is required, since numerical differentiation may produce infinite values.

The gradient of the logarithm of the $\mathrm{MIG}(\bb{\beta}, \bb{\xi}, \mat{\Omega})$ density ($k \equiv k_{\bb{\beta},\bb{\xi},\mat{\Omega}}$) with respect to $\bb{x}\in \mathcal{H}(\bb{\beta})$ is
\[
\frac{\partial}{\partial \bb{x}} \ln\{ k(\bb{x})\} = - \frac{(d/2 + 1)\bb{\beta} + \mat{\Omega}^{-1}(\bb{x}-\bb{\xi})}{\bb{\beta}^{\top}\bb{x}} +
\frac{(\bb{x}-\bb{\xi})^{\top}\mat{\Omega}^{-1}(\bb{x}-\bb{\xi})\bb{\beta}}{2(\bb{\beta}^{\top}\bb{x})^2},
\]
while the Hessian is
\[
\begin{aligned}
\frac{\partial}{\partial \bb{x} \partial \bb{x}^{\top}} \ln\{ k(\bb{x})\}
&= - \frac{\mat{\Omega}^{-1}}{\bb{\beta}^{\top}\bb{x}} +
\frac{(d/2 + 1)\bb{\beta}\bb{\beta}^{\top} \!+ \mat{\Omega}^{-1}(\bb{x}-\bb{\xi})\bb{\beta}^{\top} \!+ \bb{\beta}(\bb{x}-\bb{\xi})^{\top}\mat{\Omega}^{-1}}{(\bb{\beta}^{\top}\bb{x})^2} \\
&\qquad- \frac{(\bb{x}-\bb{\xi})^{\top}\mat{\Omega}^{-1}(\bb{x}-\bb{\xi}) \bb{\beta}\bb{\beta}^{\top}}{(\bb{\beta}^{\top}\bb{x})^3}.
\end{aligned}
\]
The Hessian of the density is then easily found to be
\[
\begin{aligned}
\frac{\partial^2}{\partial \bb{x}\partial \bb{x}^{\top}} k(\bb{x})
&= \frac{\partial^2}{\partial \bb{x}\partial \bb{x}^{\top}} \exp\big[\ln\{k(\bb{x})\}\big] \\
&= \frac{\partial}{\partial \bb{x}} \left(\exp\big[\ln\{k(\bb{x})\}\big] \, \frac{\partial }{\partial \bb{x}^{\top}} \ln\{k(\bb{x})\}\right) \\
&= k(\bb{x}) \left[ \frac{\partial }{\partial \bb{x}} \ln\{k(\bb{x})\}\frac{\partial }{\partial \bb{x}^{\top}} \ln\{k(\bb{x})\} + \frac{\partial ^2}{\partial \bb{x}\partial \bb{x}^{\top}} \ln\{k(\bb{x})\}\right].
\end{aligned}
\]

\section{Cumulative distribution function estimator for the MIG distribution}\label{app:cdf.evaluation}

In this section, the stochastic representation in Proposition~\ref{prop:exact.sim} is leveraged to evaluate the cdf of a $d$-variate MIG random vector, $\bb{X} \sim \mathrm{MIG}(\bb{\beta}, \bb{\xi}, \mat{\Omega})$.

\subsection{General strategy}

Let $f_R \equiv k_{\bb{\beta}^{\top}\!\bb{\xi},\hspace{0.3mm}(\bb{\beta}^{\top}\!\bb{\xi})^2\!/(\bb{\beta}^{\top}\!\mat{\Omega}\bb{\beta})}$ and $F_R$ denote the density and cdf, respectively, of the univariate inverse Gaussian random variable $R$ given in \eqref{eq:Z2}. By Proposition~\ref{prop:exact.sim}, the cdf of $\bb{X} \sim \mathrm{MIG}(\bb{\beta}, \bb{\xi}, \mat{\Omega})$ can be written, at any $\bb{q}\in \R^d$, as
\begin{equation}\label{eq:cdf.MIG}
\begin{aligned}
&\Pr(\bb{X} \leq \bb{q}) \\
&\quad= \int_{r_{\min}}^{r_{\max}} \int_{\R^{d-1}} \ind_{[\bb{z}_{\min}(r), \bb{z}_{\max}(r)]}(\bb{z}) \phi_{d-1}\{\bb{z} ; \bb{\mu}(r), r\bb{\Sigma}\}f_R(r) \rd \bb{z}\rd r,
\end{aligned}
\end{equation}
where $\bb{\mu} $ and $\bb{\Sigma}$ are given in \eqref{eq:cond.mean.var},
\[
r_{\min} = \max\Big(0, \min_{\bb{x}\in (-\bb{\infty}_d, \bb{q}]} \bb{\beta}^{\top} \bb{x}\Big), \quad r_{\max} = \max\Big(0, \max_{\bb{x}\in (-\bb{\infty}_d, \bb{q}]} \bb{\beta}^{\top} \bb{x}\Big),
\]
with $-\bb{\infty}_d = (-\infty,\ldots,-\infty)^{\top}$, and $\phi_{k}( \, \cdot \, ; \bb{m}, \mat{C})$ is the density of a $k$-variate Gaussian random vector with mean $\bb{m}$ and covariance matrix $\mat{C}$. Derivation of the integration bounds for $R$ and $\bb{Z} \mid \{R = r\}$ requires some work, as the region of integration is a rotated polytope.

Given that $\bb{Z} \mid \{R = r\}$ is a location and scale mixture of Gaussian random vectors, separation-of-variables \citep{Genz.Bretz:2009} can be used to evaluate the innermost Gaussian integral in \eqref{eq:cdf.MIG}. Let $\mat{L}$ denote the lower triangular Cholesky root of the covariance matrix $\bb{\Sigma} = \mat{L} \mat{L}^{\top}$, and let $\bb{Y} \sim \mathcal{N}_{d-1}(\bb{0}_{d-1}, \mat{I}_{d-1})$ be a standard Gaussian random vector. The triangular system of equations leads to a sequential decomposition of the region of integration, namely
\[
\Pr(\bb{Z} \leq \bb{z} \mid R = r) = \int_{a_1}^{b_1} \cdots \int_{a_{d-1}}^{b_{d-1}} \phi_{d-1}(\bb{y} ; \bb{0}_{d-1}, \mat{I}_{d-1}) \rd \bb{y},
\]
where
\[
\begin{aligned}
a_1 &= L_{11}^{-1}r^{-1/2}\{z_{\min,1}(r)-{\mu}_i(r)\}, \\
b_1 &= L_{11}^{-1}r^{-1/2}\{z_{\max,1}(r)-{\mu}_i(r)\},
\end{aligned}
\]
and, for every $j\in \{2, \dots, d-1\}$,
\[
\begin{aligned}
a_j &= {L_{jj}^{-1}}\left[r^{-1/2}\{z_{\min,j}(r, z_1, \ldots, z_{j-1}) - \mu_j(r)\} - \sum_{i=1}^{j-1}L_{ji}y_i \right], \\
b_j &= {L_{jj}^{-1}}\left[r^{-1/2}\{z_{\max,j}(r, z_1, \ldots, z_{j-1}) - \mu_j(r)\} - \sum_{i=1}^{j-1}L_{ji}y_i \right].
\end{aligned}
\]

This suggests that a sensible Monte Carlo estimator can be constructed using sequential importance sampling with
\[
\begin{aligned}
g_{\mathsf{sov}}(r, \bb{y})
&= \frac{f_{R}(r)}{F_R(r)} f(y_1\mid r) \prod_{j=2}^{d-1} f(y_j \mid r, y_1, \ldots, y_{j-1}) \\
&= \frac{f_{R}(r)}{F_R(r)} \prod_{j=1}^{d-1} \frac{\phi(y_j)\ind_{[a_j, b_j]}(y_j)}{\Phi(b_j) - \Phi(a_j)}
\end{aligned}
\]
as the importance sampling density, where $\phi$ and $\Phi$ denote the univariate standard normal density and cdf, respectively. Univariate draws from the truncated inverse Gaussian distribution can be obtained using rejection sampling. Based on $T$ draws from the cdf associated to $g_{\mathsf{sov}}$, the separation-of-variable estimator of the integral in \eqref{eq:cdf.MIG} is
\begin{equation}\label{eq:sov}
F_R(r) \times \frac{1}{T} \sum_{t=1}^T \prod_{j=1}^{d-1} \big\{ \Phi (b_j^{(t)}) - \Phi (a_j^{(t)}) \big\}.
\end{equation}
The separation-of-variables algorithm is described on p.~50 of~\citet{Genz.Bretz:2009} using a randomized quasi Monte Carlo procedure; see also \citet{Hintz.Hofert.Lemieux:2021}, who discuss alternative quasi Monte Carlo schemes and reordering strategies for scale mixtures.

\subsection{Integration bounds in dimension \texorpdfstring{$d = 2$}{d = 2} and \texorpdfstring{$d \geq 3$}{d >= 3}}

It is instructive to consider first the bivariate case. Given $\bb{\beta}=(\beta_1, \beta_2)^{\top}\!\in \R^2$, denote the radius $R = \bb{\beta}^{\top}\bb{X}\in (0,\infty)$. Take $\mat{Q}_2 = (-\beta_2, \beta_1)/|\bb{\beta}|_2$ to be a vector of unit length orthogonal to $\bb{\beta}$.

The probability of falling in the set $(-\bb{\infty}_{2}, \bb{q}]$ is
\[
\Pr(\bb{X} \leq \bb{q}) = \int_{r_{\min}}^{r_{\max}} \Pr\{ z_{\min}(r) \leq Z \leq z_{\max}(r) \mid R = r\}f_R(r) \rd r.
\]
Figure~\ref{fig:pmig_rotation} illustrates how the region of integration changes, depending on the position of $(\bb{-\infty}_2, \bb{q}]$ with respect to the half-space $\mathcal{H}_2(\bb{\beta})$. First, obtain the coordinates of the vertices defining the region of integration, $(\bb{-\infty}_2, \bb{q}] \cap \mathcal{H}_2(\bb{\beta})$. For each of these vertices, apply the linear transformation $\mat{Q}_2$ and observe where these points are mapped to get integration bounds for $R$ and $Z \mid R$.

\begin{figure}[!htp]
\centering
\includegraphics[width = 0.94\linewidth]{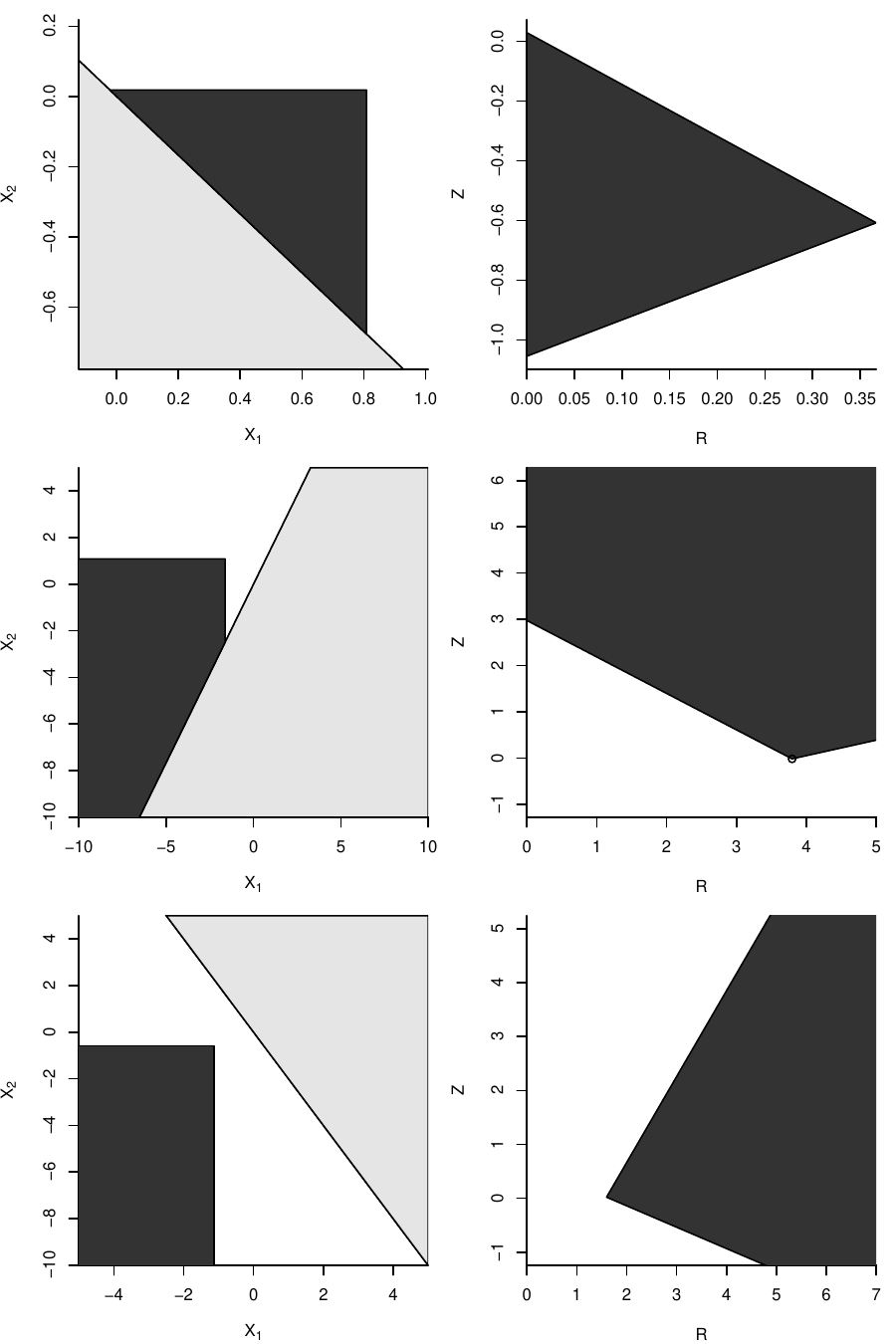}
\caption{Changes in the region of integration (black) induced by the linear transformation $\mat{Q}$ from the original space $(\bb{-\infty}_2, \bb{q}] \cap \mathcal{H}_2(\bb{\beta})$ (left) to $\mat{Q} \{(\bb{-\infty}_2, \bb{q}] \cap \mathcal{H}_2(\bb{\beta})\}$ (right) for different combinations of $\bb{\beta}$ with $\beta_1 >0, \beta_2>0$ (top), $\beta_1 <0, \beta_2 > 0$ (middle) and $\beta_1 < 0, \beta_2 < 0$ (bottom). The shaded light gray regions in the left panels indicate values outside of the support of the distribution.}
\label{fig:pmig_rotation}
\end{figure}

Let $\bb{e}_1 = (1, 0)^{\top}$ and $\bb{e}_2 = (0,1)^{\top}$ be the two standard unit vectors in $\R^2$. A careful case-by-case analysis leads to the following integration bounds:
\begin{enumerate}
\item If $\beta_1>0, \beta_2>0$, then $\Pr(\bb{X} \leq \bb{q}) = 0$ if $\bb{\beta}^{\top}\bb{q} \leq 0$. Otherwise, the vertices of the triangle defining the region of integration are $\{\bb{q}, \bb{v}_1 = (q_1, -\beta_1q_1/\beta_2)^{\top}$, $\bb{v}_2 = (-\beta_2q_2/\beta_1, q_2)^{\top}\}$. In the transformed space, the region of integration is a rotated triangle which goes from $r=0$, with $z_{\min}(0) = \mat{Q}_2 \bb{v}_1$ and $z_{\max}(0) = \mat{Q}_2 \bb{v}_2$, to $r_{\max}=\bb{\beta}^{\top}\bb{q}$, with
    \[
    \begin{aligned}
    &z_{\min}(r) = \{r/(\bb{\beta}^{\top}\bb{e}_2)\} \mat{Q}_2\bb{e}_2 + z_{\min}(0), \\
    &z_{\max}(r) = \{r/(\bb{\beta}^{\top}\bb{e}_1)\} \mat{Q}_2\bb{e}_1 + z_{\max}(0).
    \end{aligned}
    \]
\item If $\beta_1 < 0$ and $\beta_2 > 0$, there are two cases. If $\bb{q} \notin \mathcal{H}_2(\bb{\beta})$, then the region of integration is a triangle, and so is its transformed counterpart. The bounds go from $r_{\min}=0$ to $r_{\max} = \infty$, with $z_{\min}(r) = \bb{\beta}^{\top}\bb{q} - (r/\beta_1) \mat{Q}_2\bb{e}_1$. If $\bb{q} \in \mathcal{H}_2(\bb{\beta})$, then $R$ ranges from $r_{\min} = \bb{\beta}^{\top} \bb{q}$ to $r_{\max} = \infty$, while
    \[
    \begin{aligned}
    &z_{\min}(r) = \max\{\mat{Q}_2\bb{v}_2 + (r/\beta_1) \mat{Q}_2\bb{e}_1, \mat{Q}_2\bb{v}_1 + (r/\beta_2) \mat{Q}_2\bb{e}_2\}, \\
    &z_{\max}(r) = \infty.
    \end{aligned}
    \]
\item If $\beta_1 > 0$ and $\beta_2 < 0$, then $R$ ranges from $r_{\min}=0$ to $r_{\max} = \infty$, and
    \[
    \begin{aligned}
    &z_{\min}(r) = -\infty, \\
    &z_{\max}(r) = \min\{\mat{Q}_2\bb{v}_2 + (r/\beta_1) \mat{Q}_2\bb{e}_1, \mat{Q}_2\bb{v}_1 + (r/\beta_2) \mat{Q}_2\bb{e}_2\}.
    \end{aligned}
    \]
\item If $\beta_1 < 0$ and $\beta_2 < 0$, then the integration goes from $r_{\min} = \max(0, \bb{\beta}^{\top}\bb{q})$ to $r_{\max} = \infty$, where
    \[
    \begin{aligned}
    &z_{\min}(r) = \mat{Q}_2\bb{v}_2 + (r/\beta_1) \mat{Q}_2\bb{e}_1, \\
    &z_{\max}(r) = \mat{Q}_2\bb{v}_1 + (r/\beta_2) \mat{Q}_2\bb{e}_2.
    \end{aligned}
    \]
\item If $\beta_1 =0$, then $r_{\min} = 0$ and $r_{\max} = \bb{\beta}^{\top}\bb{q}$ with
    \begin{enumerate}[(a)]\setlength\itemsep{0em}
    \item $z_{\min}= - q_1$, $z_{\max} = \infty$ if $\beta_2 > 0$;
    \item $z_{\min} = -\infty$, $z_{\max} = q_1$ if $\beta_2 < 0$.
    \end{enumerate}
    \item If $\beta_2 = 0$, then $r_{\min} = \bb{\beta}^{\top}\bb{q}$ and $r_{\max} = \infty$, with
    \begin{enumerate}[(a)]\setlength\itemsep{0em}
    \item $z_{\min}= -\infty$, $z_{\max} = q_2$ if $\beta_1 > 0$;
    \item $z_{\min} = -q_2$, $z_{\max} = \infty$ if $\beta_1 < 0$.
    \end{enumerate}
\end{enumerate}

\bigskip

In dimension $d \geq 3$, sequential sampling can also be used, but this approach is expensive relative to regular Monte Carlo. Considering that integration bounds for $r$, $z_1$, \ldots, $z_{d-1}$ are found by solving linear programs for each simulated point, this requires solving $\OO(T d)$ linear programs in dimension $d$ if $T$ Monte Carlo replications are drawn.

The forward sampling algorithm is as follows. First, obtain the lower and upper bounds for the radius $R$,
\[
r_{\min} = \max(0, \min_{\bb{x}} \bb{\beta}^{\top} \bb{x}) \quad \text{and} \quad r_{\max} = \max(0, \max_{\bb{x}} \bb{\beta}^{\top} \bb{x}),
\]
by solving the two optimization problems subject to the inequality constraint $\bb{x} \leq \bb{q}$. Next, generate $T$ univariate inverse Gaussian random variables truncated between $r_{\min}$ and $r_{\max}$. Then, for each simulated value of $r$, proceed by forward sampling from the truncated conditional Gaussian distribution of $z_{j} \mid z_{1}, \ldots, z_{j-1}, r$. The truncation bounds can be obtained again by solving a linear program, where
\[
z_{\min,i}(r) = \min_{\bb{x}} (\mat{Q}_2\bb{x})_i \quad \text{and} \quad z_{\max,i}(r) = \max_{\bb{x}} (\mat{Q}_2\bb{x})_i,
\]
subject to the equality constraints $\bb{\beta}^{\top}\bb{x} = r$, $(\mat{Q}_2\bb{x})_j = z_j$ for all $j\in \{1,\ldots,i-1\}$ ~($i>2$), and the same inequality constraint $\bb{x} \leq \bb{q}$. Finally, evaluate \eqref{eq:sov} with the simulated $T$ draws to obtain the estimator.

\section{Reproducibility}\label{app:R.code}
The \textsf{R} codes used to generate the figures, the simulation study results and the real-data application can be found \href{https://github.com/lbelzile/mig-kernel}{here}. The \textsf{R} package \texttt{mig} implements the kernel density estimator and is available from the \href{https://cran.r-project.org/web/packages/mig/index.html}{CRAN} \citep{mig:cran} and from GitHub at \url{https://github.com/lbelzile/mig}.
\section{List of abbreviations}\label{app:abbreviations}
\begin{tabular}{llll}
&AMISE  &\hspace{20mm} &asymptotic mean integrated squared error \\
&BRMISE &\hspace{20mm} &boundary root mean integrated squared error \\
&cdf    &\hspace{20mm} &cumulative distribution function \\
&IG     &\hspace{20mm} &inverse Gaussian \\
&iid    &\hspace{20mm} &independent and identically distributed \\
&KLD    &\hspace{20mm} &Kullback--Leibler divergence \\
&LCV    &\hspace{20mm} &likelihood cross-validation \\
&LSCV   &\hspace{20mm} &least square cross-validation \\
&MIG    &\hspace{20mm} &multivariate inverse Gaussian \\
&MISE   &\hspace{20mm} &mean integrated squared error \\
&RLCV   &\hspace{20mm} &robust likelihood cross-validation \\
&RMISE  &\hspace{20mm} &root mean integrated squared error \\
\end{tabular}

\end{appendix}

\begin{acks}[Acknowledgments]
The authors are grateful to the four anonymous reviewers and the Associate Editor for their constructive comments. Their valuable suggestions were instrumental in refining the quality and depth of the present study.
\end{acks}

\begin{funding}
Belzile acknowledges funding from the Natural Sciences and Engineering Research Council of Canada through Discovery Grant RGPIN-2022-05001. Ge\-nest's research and Ouimet's current postdoctoral fellowship are funded through the Canada Research Chairs Program (Grant 950-231937) and the Natural Sciences and Engineering Research Council of Canada (Discovery Grant RGPIN-2024-04088 to C.~Ge\-nest). Ouimet also wishes to acknowledge past support from a CRM-Simons postdoctoral fellowship from the Centre de recherches math\'ematiques (Montr\'eal, Canada) and the Simons Foundation.
\end{funding}

\bibliographystyle{imsart-nameyear}
\bibliography{bib}

\end{document}